\documentclass{amsart}%
\usepackage{amsfonts}
\usepackage[table]{xcolor}
\usepackage{hyperref}
\usepackage[all,cmtip]{xy}
\usepackage[margin=1.34in]{geometry}
\usepackage{amsmath}
\usepackage{amssymb}
\usepackage{graphicx}%
\newtheorem{theorem}{Theorem}
\newtheorem{theoremann}{Theorem}

\numberwithin{theorem}{section}
\theoremstyle{plain}
\newtheorem*{acknowledgement}{Acknowledgement}

\newtheorem*{conjecture}{Conjecture}
\newtheorem{corollary}[theorem]{Corollary}

\newtheorem{definition}[theorem]{Definition}

\newtheorem{lemma}[theorem]{Lemma}

\newtheorem{problem}[theorem]{Problem}
\newtheorem{proposition}[theorem]{Proposition}
\theoremstyle{remark}

\newtheorem{remark}[theorem]{Remark}

\newtheorem{pitfalls}[theorem]{Pitfall}
\newtheorem{vista}[theorem]{Vista}

\newtheorem{example}[theorem]{Example}

\numberwithin{equation}{section}
\definecolor{lg}{rgb}{0.8,0.8,0.8}
\newcommand*{\Scale}
[2][4]{\scalebox{#1}{$#2$}}
\begin{document}
\title[Local compactness and finite generation]{Local compactness as the $K(1)$-local dual of\linebreak finite generation}
\author{Oliver Braunling}
\address{University of Wuppertal, Germany}
\email{\texttt{oliver.braunling@gmail.com}}

\begin{abstract}
Suppose $R$ is any localization of the ring of integers of a number field. We
show that the $K$-theory of finitely generated $R$-modules, and the $K$-theory
of locally compact $R$-modules, are Anderson duals in the $K(1)$-local
homotopy category. The same is true for $p$-adic and finite fields.

\end{abstract}
\maketitle

\section{Introduction}

Let $p$ be an odd prime. We prove that the $K$-theories of the following two
exact categories are $K(1)$-local mutual Anderson duals:%
\[
\text{finitely generated }R\text{-modules}\quad\leftrightarrow\quad
\text{locally compact }R\text{-modules}%
\]
in the special case where $\operatorname*{Spec}R$ is a
\textit{one-dimensional} arithmetic scheme flat over $\mathbf{Z[}\frac{1}{p}%
]$, and up to finite-dimensional real vector spaces. Suitably interpreted,
this duality should be expected to hold in arbitrary dimension. Before we say
more about this, here is a precise formulation of what we prove. We write
$\mathcal{O}_{S}$ for the ring of $S$-integers.

\begin{theoremann}
[Duality]\label{thmann_ThmA}Let $p$ be an odd prime, $F$ a number field, $S$ a
(possibly infinite) set of finite places of $F$ such that $\frac{1}{p}%
\in\mathcal{O}_{S}$. Then%
\[
I_{\mathbf{Z}_{p}}L_{K(1)}K(\mathcal{O}_{S})\cong L_{K(1)}K(\mathsf{LCA}%
_{\mathcal{O}_{S}}^{\circ})\text{.}%
\]
If $S$ contains only finitely many places, this is a mutual duality.
\end{theoremann}

As in our previous papers on the subject, $\mathsf{LCA}_{\mathcal{O}_{S}}$
denotes the exact category of $\mathcal{O}_{S}$-modules which are equipped
with a locally compact topology. They need not be finitely generated, but
morphisms as well as the scalar action by $\mathcal{O}_{S}$ are required to be
continuous. The superscript $(-)^{\circ}$ stands for quotienting out all such
modules which topologically are a real vector space. The duality exchanges
pullbacks and pushforwards along finite morphisms.

In the $K(1)$-local world, all popular concepts of duality, notably
$p$-complete Anderson duals, Brown--Comenetz duals or Spanier--Whitehead
duals, agree up to shifts and twists. Thus the \textit{finite generation
}$\leftrightarrow$ \textit{local compactness} dualism is a phenomenon valid
for all these duals. We chose to focus on the Anderson dual since it leads to
the simplest formulation.

The duality can also be phrased in terms of a pairing%
\begin{equation}
\mathsf{LCA}_{\mathcal{O}_{S}}\times\mathsf{Proj}_{\mathcal{O}_{S}%
,fg}\longrightarrow\mathsf{LCA}_{\mathcal{O}_{S}} \label{lwa0}%
\end{equation}
of exact categories, which induces a right $K(\mathcal{O}_{S})$-module
structure%
\begin{equation}
K(\mathsf{LCA}_{\mathcal{O}_{S}})\otimes K(\mathcal{O}_{S})\longrightarrow
K(\mathsf{LCA}_{\mathcal{O}_{S}})\text{.} \label{lwa1}%
\end{equation}
Once one moves to the $K(1)$-local homotopy category, one can set up a trace
map%
\begin{equation}
L_{K(1)}K(\mathsf{LCA}_{\mathcal{O}_{S}})\longrightarrow I_{\mathbf{Z}_{p}%
}\mathbb{S}_{\widehat{p}} \label{lwa2}%
\end{equation}
and composing Eq. \ref{lwa1} with Eq. \ref{lwa2}, one obtains the pairing
inducing the duality.

Our whole approach, and especially the idea to study a pairing as in Eq.
\ref{lwa0} is inspired by Clausen's work \cite{clausenthesis,clausen}, and
confirms most (so far unproven) expectations about a $K(1)$-local duality to
something like `$K$-theory with compact supports' (at least in dimension $1$
and away from $p$). I do not know whether Clausen's setting in \cite{clausen}
would literally exhibit the same duality. Instead of $\mathsf{LCA}%
_{\mathcal{O}_{S}}^{\circ}$ he considers $\infty$-categorical $\mathbf{Z}%
$-linear functors from $F$-vector spaces to bounded complexes in
$\mathsf{LCA}_{\mathbf{Z}}$, and he does not quotient out real vector spaces.
The pairings end up being essentially the same, but I think that the dual
might come out a tiny bit too big.

Blumberg--Mandell \cite{MR4121155} compute $I_{\mathbf{Z}_{p}}L_{K(1)}%
K(\mathcal{O}_{S})$ in the special case where $S$ are solely the places over
$p$, but they express the dual as some abstract fiber called $\left.
\mathit{Fib}(\kappa)\right.  $ in their paper, and do not talk about local
compactness at all. The main novelty of the present text is the explicit model
of the dual in terms of local compactness, as well as allowing much more
general choices of $S$ than in Blumberg--Mandell, including the technically
delicate situation of infinite $S$. To reach a duality for the $K$-theory of
the whole number field $I_{\mathbf{Z}_{p}}L_{K(1)}K(F)$, the situation in
Clausen's work, it is crucial to allow infinite sets $S$.

As the reader will easily notice, these works of Clausen and Blumberg--Mandell
have served as the key inspiration for us. Our perspective, moving away from
class field theory and towards a duality between finite generation and local
compactness for module categories over $\mathcal{O}_{S}$, appears to be new.

More generally, $\mathcal{O}_{S}\mapsto K(\mathsf{LCA}_{\mathcal{O}_{S}})$ is
expected to generalize to arbitrary (qcqs) input schemes and should extend to
something like `$K$-theory with compact supports', see \cite{clausenthesis}.
On regular schemes which are flat, finite type and separated over some ring of
$S$-integers $\mathcal{O}_{S}$ (still assuming $\frac{1}{p}\in\mathcal{O}_{S}%
$), a version of Theorem \ref{thmann_ThmA} is expected to hold. Computations
in \cite{lcahighdim} show that the definition of $\mathsf{LCA}_{\mathcal{O}%
_{S}}$ for higher-dimensional schemes requires derived local compactness or
more sophisticated concepts of topology, e.g., condensed methods or $n$-Tate
local compactness instead of plain local compactness. But this is a topic for
the follow-up article. A local and finite field counterpart of Theorem
\ref{thmann_ThmA} is the following.

\begin{theoremann}
\label{thmann_ThmA_LocalVersion}Let $p$ be an odd prime and $F$ a finite
extension of $\mathbf{Q}_{\ell}$ or $\mathbf{F}_{\ell}$ for any prime
$\ell\neq p$, or a finite extension of the reals $\mathbf{R}$. Then%
\[
I_{\mathbf{Z}_{p}}L_{K(1)}K(F)\cong L_{K(1)}K(\mathsf{LCA}_{F})\text{.}%
\]
This is a mutual duality.
\end{theoremann}

The statement for $F/\mathbf{Q}_{\ell}$ really is due to Blumberg--Mandell,
even though they phrase it differently. They prove in \cite[Thm.
1.4]{MR4121155} that $K(F)$ is Anderson self-dual, namely%
\[
I_{\mathbf{Z}_{p}}L_{K(1)}K(F)\cong L_{K(1)}K(F)\text{,}%
\]
but for local fields with finite residue field there is no difference between
finite generation and local compactness:\ Every finite-dimensional $F$-vector
space is canonically local compact and any locally compact $F$-module must be
finite-dimensional. So their self-duality statement is indeed an instance of
the \textit{finite generation }$\leftrightarrow$\textit{ local compactness}
dualism. Note that the situation is quite different for, say, finite fields,
where $\mathbf{F}_{q}((t))$ with the $t$-adic topology, is a locally compact
$\mathbf{F}_{q}$-module far from finite generation. Still the duality
phenomenon is valid. The finite field part of Theorem
\ref{thmann_ThmA_LocalVersion} is also discussed in Blumberg--Mandell--Yuan
\cite{bmychromaticconvergence} but without the local compactness
interpretation. Incidentally, they observe a shift $\Sigma$ in the Anderson
dual. As we show in the proof of Theorem \ref{thmann_ThmA_LocalVersion}, this
shift corresponds to a shift on the level of $\mathsf{LCA}_{\mathbf{F}_{q}}$
which is already visible \textit{before} $K(1)$-localization.

The lack of symmetry in Thm. \ref{thmann_ThmA} for infinite $S$ comes from the
general problem that Anderson duality is only reflexive for spectra whose
homotopy groups are finitely generated $\mathbf{Z}_{p}$-modules. I expect that
one can formulate the result in a setting of locally compact condensed spectra
and with genuine Pontryagin duality. Then Thm. \ref{thmann_ThmA} should extend
to hold for infinite $S$ in both directions.

We proceed as follows:

\begin{enumerate}
\item In \S \ref{sect_FibSequence} we compute $K(\mathsf{LCA}_{\mathcal{O}%
_{S}})$ in terms of a fiber sequence in the ordinary stable homotopy category.
This generalizes \cite{kthyartin} to arbitrary $S$. The $K(1)$-local version
of this fiber sequence, for $S$ being only the places above $p$, corresponds
to the abstract fiber $\left.  \mathit{Fib}(\kappa)\right.  $ in
\cite{MR4121155}. However, as we identify the fiber in terms of local
compactness, this is a very different fiber sequence than the one used by Blumberg--Mandell.

\item In \S \ref{sect_LocalDuality} we develop a restricted product version of
the Thomason descent spectral sequence for the ad\`{e}les. For finite $S$ this
is harmless, as the restricted product is both a product and coproduct.
However, for infinite $S$ we face the interplay of infinite limits and
colimits with $K(1)$-localization. For this, we rely on technology from
\cite{MR4296353} and \cite{MR4444265}.

\item The ring of $S$-integers $\mathcal{O}_{S}$ only depends on the finite
places in $S$, and does not change as we add or remove infinite places from
$S$. The same is true for $\mathsf{LCA}_{\mathcal{O}_{S}}$. In
\S \ref{sect_NALCAModules} we introduce a new category $\left.  \mathsf{LC}%
\mathcal{O}_{S}\right.  $, sensitive also to the infinite places in $S$, and
which agrees with $\mathsf{LCA}_{\mathcal{O}_{S}}$ if $S$ contains all
infinite places.

\item In \S \ref{sect_TraceMap} we set up the trace map for the duality.
Blumberg--Mandell do this using a compatible system of sections to a certain
inclusion $\mathbf{Q}_{p}/\mathbf{Z}_{p}\subseteq\pi_{0}$ defined for
$\mathbf{Z[}\frac{1}{p}]$. We favour an idea along Clausen's lines in
\cite{clausen} here: After the trace map has been defined for $\mathbf{Z[}%
\frac{1}{p}]$, one gets the one for general $\mathcal{O}_{S}$ just by
pre-composing with the forgetful functor $\mathsf{LCA}_{\mathcal{O}_{S}%
}\rightarrow\mathsf{LCA}_{\mathbf{Z}\left[  \frac{1}{p}\right]  }$, forgetting
the $\mathcal{O}_{S}$-module structure in favour of the $\mathbf{Z[}\frac
{1}{p}]$-module structure. This is just the pushforward $f_{\ast}$ on the
locally compact side of the duality. We take the time to prove that this ends
up being the same as one would get from the Blumberg--Mandell method (Prop.
\ref{prop_VIsASectionOfE}).

\item In \S \ref{sect_Duality} we set up the duality. This is mostly inspired
by Blumberg--Mandell, but now comes with the additional complication of
infinite $S$. While \cite{MR4121155} develops an \'{e}tale hypersheaf version
of Artin--Verdier duality, we work with the Poitou--Tate sequence. For finite
$S$, these are well-known to be equivalent. However, for infinite $S$ we do
not know any formulation in terms of \'{e}tale (hyper)sheaves because one
would need a recollement where the `open immersion' is an arbitrary
localization, not necessarily of finite type.

\item In the end, the proof can be summarized as follows: Use the fibers which
occur in \cite{kthyartin} and \cite{MR4121155} (as the result of two a priori
unrelated computations) and after generalizing both approaches to accomodate
arbitrary sets of places, prove that the respective fibers match.
\end{enumerate}

Along the way we encounter some potential pitfalls, which repeatedly cause
minor trouble. First, if $A$ is a profinite abelian group,
$\operatorname*{Hom}_{\mathsf{Ab}}(A,\mathbf{Q}/\mathbf{Z})$ only outputs the
correct Pontryagin dual if $A$ is topologically finitely generated. This
causes issues with Brown--Comenetz duals (and is a good reason to work with
condensed spectra instead). Second, $K(1)$-localization of the $K$-theory of
fields yields Galois cohomology, but not Tate cohomology. This causes problems
at infinite places. Third, infinite limits and colimits need to be moved
between ordinary spectra and $K(1)$-local ones. This turns out to be
well-behaved in the end, but not obviously so.

We stay away from $\operatorname*{char}p>0$ in this text, we always assume
$\frac{1}{p}\in\mathcal{O}_{S}$. Clausen's \cite{clausen} contains a wealth of
information what one could do to incorporate this. We stay away from $p=2$ not
just because of its eccentric Moore spectrum. We would also need to control
some $2$-torsion phenomena in Tate cohomology at the real places which seem
unclear how to implement.

\section{Philosophy}

In order to motivate \textit{why} $\mathsf{LCA}_{\mathcal{O}_{S}}$ might be a
good candidate for a dual, let us explain a simple thought experiment. We
shall only consider the $K$-theory of fields. Let $F^{\prime}/F$ be a field
extension. We prefer a geometric language, so we denote this as%
\[
f\colon\operatorname*{Spec}F^{\prime}\longrightarrow\operatorname*{Spec}%
F\text{.}%
\]
As everything is affine, the transfers $f^{\ast},f_{\ast}$ for quasi-coherent
sheaves correspond to exact functors $\mathsf{Mod}_{F^{\prime}}%
\leftrightarrows\mathsf{Mod}_{F}$, and are%
\[
M\mapsto M\otimes_{F}F^{\prime}\qquad\text{and}\qquad M\mapsto M\mid
_{F}\text{,}%
\]
where the latter forgets the $F^{\prime}$-vector space structure in favour of
the $F$-vector space structure. In order to induce non-trivial maps on
$K$-theory, one needs to impose a `concept of finiteness'. For if not, the
$K$-theory $K(\mathsf{Mod}_{F})$ of \textit{all} $F$-vector spaces is zero by
the Eilenberg swindle. Civilized people then chose to demand finite generation
(or siblings thereof, like coherent modules or perfect complexes). For
$F$-vector spaces, this just means that $\dim_{F}(M)<\infty$. This category
has non-zero $K$-theory, but now we need to make sure that this concept of
finiteness is compatible with our transfers. As%
\[
M\mapsto M\otimes_{F}F^{\prime}%
\]
sends any finitely generated $F$-module to a finitely generated $F^{\prime}%
$-module, we get \emph{flat pullbacks} for arbitrary field extensions, not
necessarily finite. Although we only talk about fields here, this is the same
reason why we have flat pullbacks for the $K$-theory of schemes without having
to require that $f$ is of finite type. For the pushforward, the story is
different: $\left.  M\mid_{F}\right.  $ will only be a finite-dimensional
$F$-vector space if $F^{\prime}/F$ is a finite extension. This is the toy
model for why we only get pushforwards for proper maps, and in the case of
affines, finite maps.

All this is of course trivial. But let us use these observations as guidance
to set up a dual theory. Having Thm. \ref{thmann_ThmA} in mind, $\mathsf{LCA}%
_{F}$ should admit transfers which behave exactly in the opposite way. For
$\mathsf{LCA}_{F}$, the `concept of finiteness' is that we consider $F$-vector
spaces equipped with a locally compact topology. This is no constraint per se,
as any $F$-vector space can be equipped with the discrete topology and then is
locally compact. But there are numerous types of topologies on $F$-vector
spaces which act remarkably similar to imposing finite-dimensionality. For
example, pick any prime $p$. Then the class of $F$-vector spaces which are
topologically $p$-torsion in $\mathsf{LCA}_{F}$ (i.e., $\lim_{n\rightarrow
\infty}p^{n}x=0$ holds for all $x$) turns out to agree with the class of those
which, as an LCA\ group\footnote{locally compact abelian}, are isomorphic to a
finite-dimensional $\mathbf{Q}_{p}$-vector space. So in some sense, demanding
local compactness is not entirely unrelated to finite generation.

However, and this is the key point in terms of duality, local compactness acts
oppositely to finite generation when it comes to transfers: If $M$ carries a
locally compact topology, $M\otimes_{F}F^{\prime}\simeq M^{[F^{\prime}:F]}$
can be given the product topology (no worries, there is a more functorial and
elegant way to define this topology). If $[F^{\prime}:F]<\infty$, this will
again be locally compact. This only works if $F^{\prime}/F$ is finite. On the
other hand, for pushforwards, it is easy to equip $\left.  M\mid_{F}\right.  $
with a locally compact topology: We just keep the topology which $M$ already
has. This always works, even if $F^{\prime}/F$ is an infinite field extension.
Thus, the roles which of the transfers requires $F^{\prime}/F$ to be finite,
are reversed. Similarly, there is a pushforward $\mathsf{LCA}_{\mathbf{Q}%
}\rightarrow\mathsf{LCA}_{\mathbf{Z}}$, but no pullback in the reverse direction.

\begin{theoremann}
\label{thmann_D}If $F^{\prime}/F$ is a finite extension, $S$ a (possibly
infinite) set of finite places of $F$, and $S^{\prime}$ the places of
$F^{\prime}$ lying above $S$, then the duality%
\[
I_{\mathbf{Z}_{p}}L_{K(1)}K(\mathcal{O}_{S})\cong L_{K(1)}K(\mathsf{LCA}%
_{\mathcal{O}_{S}}^{\circ})
\]
of Thm. \ref{thmann_ThmA} exchanges the roles of the transfers: The
pushforward $f_{\ast}$ on the left corresponds to the pullback $f^{\ast}$ on
the right, and reversely. The same is true for Thm.
\ref{thmann_ThmA_LocalVersion}.
\end{theoremann}

We note that the above discussion is unrelated to $K(1)$-local phenomena. The
assignment $X\mapsto K(\mathsf{LCA}_{X})$ is sensible already in the ordinary
homotopy category.

\section{Setting the stage\label{sect_SettingTheStage}}

\subsection{Conventions}

All rings are unital and commutative. Ring homomorphisms are required to
preserve the unit. Dedekind domains are required to have Krull dimension
$\leq1$, i.e., we also consider fields to be Dedekind domains. The formulation
\textquotedblleft almost all\textquotedblright\ means \textquotedblleft for
all but finitely many\textquotedblright. Cohomology groups $H^{i}%
(X,\mathcal{F})$ for $X$ a scheme or ring always refer to \'{e}tale
cohomology, frequently reformulated to the setting of Galois cohomology. We
write $\widehat{H}^{i}$ for Tate cohomology, but it rarely plays a role. We
use $\mathbf{Z}/p^{k}(j)$ to refer to the $j$-th Tate twist. This is the same
as what other people might denote by $\mathbf{\mu}_{p^{k}}^{\otimes j}$.
Whenever we write $\mathbf{Z}/p^{k}(\tfrac{j}{2})$ and $j$ is odd, we mean the
zero module. The ring of integers $\mathcal{O}_{v}$ in an archimedean local
field $F_{v}$ is agreed to refer to $F_{v}$ itself. Notation for class field
theory is usually in line with the notation in Milne's book \cite{MR2261462}.
For set-theoretic concerns when considering \textquotedblleft
all\textquotedblright\ $R$-modules over a ring $R$, we refer the reader to
\S \ref{sect_Appendix_SetTheory}.

We write $\mathsf{Sp}$ for the ordinary homotopy category, i.e., the $\infty
$-category of spectra, and $\mathsf{K}$ for the $K(1)$-local homotopy category
and occasionally add a superscript to (co)limits when we want to stress in
what category it is taken. When speaking of a spectrum, we usually have a
symmetric spectrum as in \cite{symspec} in mind, but one could use any
equivalent model. We write $I_{\mathbf{Z}_{p}}$, $I_{\mathbf{Q}_{p}%
/\mathbf{Z}_{p}}$ for the $p$-complete Anderson resp. Brown--Comenetz dual.
For more on $K(1)$-local duals, we summarize the key definitions and
properties in \S \ref{sect_Appendix_DualityInKOneLocalHptyCat}. When we talk
about $K$-theory, we always mean \textit{non-connective} $K$\textit{-theory},
\cite{MR2206639,MR3070515}. Up to taking the $1$-connective cover, this agrees
with Quillen $K$-theory.

We write $\mathsf{D}_{\infty}^{?}(\mathsf{C})$ for the derived $\infty
$-category of $\mathsf{C}$, and $\mathsf{D}^{?}(\mathsf{C}):=Ho\,\mathsf{D}%
_{\infty}^{?}(\mathsf{C})$ for its homotopy category, i.e., the classical
derived category as a triangulated category.

\subsection{Setup}

Suppose $F$ is a number field. We write $\mathcal{O}_{F}$ for the ring of
integers of $F$, and $S_{\infty}$ for its set of infinite places. If $v$ is
any place, $F_{v}$ denotes the metric completion of $F$ at $v$. If $v$ is a
finite place, $\mathcal{O}_{v}$ denotes the ring of integers in $F_{v}$, or
equivalently the metric completion of $\mathcal{O}_{F}$ at $v$. We will
tacitly identify finite places with points in $\operatorname*{Spec}%
\mathcal{O}_{F}$. Let $S\supseteq S_{\infty}$ be a set of places of $F$. The
set $S$ is allowed to be infinite, e.g., it could be all places. Write
$\mathcal{O}_{S}$ for the ring of $S$-integers: It can be defined as%
\begin{equation}
\mathcal{O}_{S}:=\left\{  x\in F\mid x\in\mathcal{O}_{v}\text{ for all
}v\notin S\right\}  \text{.} \label{l_Def_RingOS}%
\end{equation}
If $S$ is finite, this also arises as follows: If we define
$X:=\operatorname*{Spec}\mathcal{O}_{F}$ and $U:=X-(S\setminus S_{\infty})$,
then $\mathcal{O}_{S}\cong\Gamma(U,\underline{\mathcal{O}}_{X})$ (where
$\underline{\mathcal{O}}_{X}$ denotes the structure sheaf of $X$). All open
subschemes of $X$ are of this form. Note that they are all automatically
affine by the finiteness of the class group.

\section{\label{sect_FibSequence}Computations on the locally compact side}

Let $F$ be a number field. Let $S\supseteq S_{\infty}$ be a set of places of
$F$, possibly infinite. We shall write $\mathbb{T}:=\mathbf{R}/\mathbf{Z}$ for
the unit circle, regarded as an LCA\ group.

\begin{definition}
\label{def_1}Let $\mathsf{LCA}_{\mathcal{O}_{S}}$ be the category of locally
compact $\mathcal{O}_{S}$-modules with continuous $\mathcal{O}_{S}$-module
homomorphisms as morphisms.
\end{definition}

\begin{remark}
\label{rmk_WhatDoesLCAMeanForSomeRing}This means that objects are
$\mathcal{O}_{S}$-modules $X$, equipped with a topology such that $X$ is
locally compact as a topological space and $X$ is a topological $\mathcal{O}%
_{S}$-module when $\mathcal{O}_{S}$ is given the discrete topology. The latter
is equivalent to demanding that for every $\alpha\in\mathcal{O}_{S}$, the
multiplication $(\alpha\cdot-)\colon X\rightarrow X$ is continuous. The
category $\mathsf{LCA}_{\mathcal{O}_{S}}$ is equivalent to the ($1$%
-categorical) functor category $\mathsf{Fun}(\left\langle \mathcal{O}%
_{S}\right\rangle $,$\mathsf{LCA}_{\mathbf{Z}})$, where $\left\langle
\mathcal{O}_{S}\right\rangle $ is the category with $\mathcal{O}_{S}$ as its
single object and $\operatorname*{End}(\mathcal{O}_{S},\mathcal{O}%
_{S})=\mathcal{O}_{S}$ its endomorphism ring.
\end{remark}

If only for set-theoretic concerns, one might feel more comfortable to work
with cardinality bounds instead of truly allowing \textit{all} modules. This
can be done, but does not affect any of our results. We explain this in
Appendix \S \ref{sect_Appendix_SetTheory}. This discussion also fills in what
some people have felt as a matter of unease in our previous papers on locally
compact modules.

\begin{remark}
\label{rmk_AlgebraicModules}We may occasionally speak about \emph{algebraic
}$\mathcal{O}_{S}$\emph{-modules} when we want to stress that we work in the
category $\mathsf{Mod}_{\mathcal{O}_{S}}$ (this is particularly relevant when
modules do carry a topology, but we have no yet checked whether the scalar
action is continuous).
\end{remark}

\begin{example}
The special case $F:=\mathbf{Q}$ and $S=\{\infty\}$ yields $\mathcal{O}%
_{S}=\mathbf{Z}$. For example, $\mathbb{T}\in\mathsf{LCA}_{\mathbf{Z}}$. The
homological algebra in the category $\mathsf{LCA}_{\mathbf{Z}}$ was studied in
great detail by Moskowitz \cite{MR0215016} and Hoffmann and Spitzweck
\cite{MR2329311}, which was a point of origin for ensuing work on categories
of locally compact modules, as in \cite{clausen}.
\end{example}

The category $\mathsf{LCA}_{\mathcal{O}_{S}}$ is an $\mathcal{O}_{S}$-linear
quasi-abelian category.\ In particular, it is an $\mathcal{O}_{S}$-linear
exact category. Admissible monics (resp. epics) are closed injective (resp.
open surjective) $\mathcal{O}_{S}$-module homomorphisms. The Pontryagin
dual\footnote{this is $X\mapsto\operatorname*{Hom}(-,\mathbb{T})$, equipped
with the compact-open topology. Beware: Even though $\operatorname*{Hom}%
(A,B)$-groups can always be equipped with the compact-open topology, they will
in general fail to be locally compact even if $A,B$ both are. For
$B:=\mathbb{T}$, they are always locally compact.}%
\[
(-)^{\vee}\colon\mathsf{LCA}_{\mathcal{O}_{S}}\longrightarrow\mathsf{LCA}%
_{\mathcal{O}_{S}}^{op}%
\]
is an exact equivalence of exact categories and moreover renders
$\mathsf{LCA}_{\mathcal{O}_{S}}$ an exact category with duality. All objects
are reflexive. This can all be imported from the analogous properties for
$\mathsf{LCA}_{\mathcal{O}_{F}}$.

\begin{definition}
\label{def_MainCats}Suppose $X\in\mathsf{LCA}_{\mathcal{O}_{S}}$.

\begin{enumerate}
\item We call $X$ \emph{discrete} (resp. \emph{compact}) if it is discrete
(resp. compact) as a topological space.\footnote{We spelled this out to avoid
any confusion with the term of a `compact object' in a category, or the use of
`discrete' in an $\infty$-category.} Similarly for \emph{connected} or
\emph{totally disconnected}.

\item Call $X$ a \emph{vector module} if its underlying LCA group is
isomorphic to a finite-dimensional real vector space with the usual topology.

\item We call $X$ \emph{vector-free} if $X$ cannot be written as a direct sum
$X\simeq V\oplus X^{\prime}$, where $V$ is a non-zero vector module.

\item Call $X$ \emph{quasi-adelic} if $X\simeq V\oplus Q$ in $\mathsf{LCA}%
_{\mathcal{O}_{S}}$, where $V$ is a vector module and%
\begin{equation}
Q=\mathcal{O}_{S}\cdot C=\left\{  \alpha c\mid\alpha\in\mathcal{O}_{S},c\in
C\right\}  \text{,} \label{lm2}%
\end{equation}
where $C\subseteq Q$ is some compact clopen $\mathcal{O}_{F}$-submodule of
$Q$. The choice of $C$ is not part of the datum, it merely needs to exist.
Call $X$ \emph{adelic} if additionally we can arrange that $\bigcap_{\alpha
\in\mathcal{O}_{S}^{\times}}\alpha C=0$.
\end{enumerate}
\end{definition}

We write $\mathsf{LCA}_{\mathcal{O}_{S},qa}\subseteq\mathsf{LCA}%
_{\mathcal{O}_{S}}$ (resp. $\mathsf{LCA}_{\mathcal{O}_{S},ad}$) for the full
subcategory of quasi-adelic (resp. adelic) objects.

\begin{example}
Every compact $\mathcal{O}_{S}$-module is quasi-adelic, just take $C$ to be
itself and $V=0$.
\end{example}

\begin{definition}
\label{def_cg}An LCA group $G$ is called \emph{compactly generated}%
\footnote{Most unfortunately, this term carries at least two other quite
unrelated meanings. A (say Hausdorff) topological space $X$ is called
\emph{compactly generated} if for any subset $Y\subseteq X$ the property
\textquotedblleft$Y\cap C$ is closed for any compact subspace $C\subseteq
X$\textquotedblright\ implies that $Y$ is closed. Every LCA group is compactly
generated Hausdorff in this sense. In our text, we only use `compactly
generated' in the sense of topological group theory, i.e., as in Def.
\ref{def_cg}. There is a further meaning of compact generation, related to
category theory, which is yet something different, but also less likely to
lead to confusion.} (in brief: \emph{c.g.}) if there exists a compact subset
$K\subseteq G$ such that (1) $g\in K$ implies $-g\in K$, and (2) if we write
$K^{n}:=\operatorname*{im}_{\mathsf{Set}}(K\times K\times\cdots\times
K\overset{\cdot}{\rightarrow}G)$ for the set-theoretic image\footnote{We
stressed that we mean the \textit{set-theoretic image} because the
category-theoretic image in the category $\mathsf{LCA}_{\mathbf{Z}}$ is not
necessarily the same. The category-theoretic image is always the closure of
the set-theoretic one, so this distinction plays a role.} under the $n$-fold
addition map, then $G=\bigcup_{n\geq1}K^{n}$.
\end{definition}

\begin{definition}
\label{def_indcg}Suppose $X\in\mathsf{LCA}_{\mathcal{O}_{S}}$. Call $X$
\emph{ind-c.g.} (for \emph{inductively compactly generated}) if%
\[
X=\mathcal{O}_{S}\cdot H=\left\{  \alpha h\mid\alpha\in\mathcal{O}_{S},h\in
H\right\}  \text{,}%
\]
where $H\subseteq Q$ is some compactly generated clopen $\mathcal{O}_{F}%
$-submodule of $X$. Write $\mathsf{LCA}_{\mathcal{O}_{S},icg}\subseteq
\mathsf{LCA}_{\mathcal{O}_{S}}$ for the full subcategory of ind-c.g. objects.
The choice of $H$ is not part of the datum, it merely needs to exist.
\end{definition}

While most of the above concepts generalize their counterparts from
\cite{kthyartin}, the concept of ind-c.g. modules did not appear in our
earlier papers. It turns out to be a useful concept simplifying various arguments.

\subsection{Ad\`{e}les and adelic modules}

Instead of torsion submodules, the concept of \textit{topological torsion}
plays a crucial role in the theory of LCA groups. Unlike algebraic torsion, it
is preserved under duality: The group $\mathbf{Q}_{p}/\mathbf{Z}_{p}$ is a
discrete $p$-torsion group. Its Pontryagin dual $\mathbf{Z}_{p}$ is
algebraically torsion-free, but still topological $p$-torsion.

\begin{definition}
An LCA group $G$ is called a \emph{topological torsion} group if every open
neighbourhood of zero contains a compact clopen subgroup $U$ such that $G/U$
is discrete and torsion. We call $A\in\mathsf{LCA}_{\mathcal{O}_{S}}$ a
\emph{topological torsion} module if its underlying LCA group is topologically torsion.
\end{definition}

There is also a version of this concept specific to individual primes.

\begin{definition}
Suppose $A\in\mathsf{LCA}_{\mathcal{O}_{S}}$ and $v=P$ the place associated to
a maximal ideal $P\in\operatorname*{Spec}\mathcal{O}_{F}$. We call $a\in A$
\emph{topological }$P$\emph{-torsion} if for every neighbourhood of zero $0\in
U\subset A$ there exists some $N_{U}\geq1$ such that%
\[
P^{N_{U}}\cdot a\subseteq U\text{.}%
\]
The object $A$ is called \emph{topological }$P$\emph{-torsion} if all its
elements are topological $P$-torsion elements. We write%
\[
A^{P}:=\left\{  a\in A\mid a\text{ is topological }P\text{-torsion}\right\}
\]
for the algebraic $\mathcal{O}_{F}$-submodule of topological $P$-torsion elements.
\end{definition}

For $\mathcal{O}_{S}=\mathbf{Z}$, this reduces to the familiar concept of
topological $p$-torsion.

\begin{example}
In general, $A^{P}$ need not be a closed submodule of $A$, so this concept
does not necessarily lend itself to decompose $A$ further within the category
$\mathsf{LCA}_{\mathcal{O}_{S}}$. As a concrete example, for $F=\mathbf{Q}$,
$S=\{\infty\}$, we have $\mathcal{O}_{S}=\mathbf{Z}$, and for any prime $p$,
the group $\mathbb{T}^{(p)}$ is the subgroup of $p$-power torsion elements by
\cite[Prop. 2.6]{MR637201}, i.e., although $\mathbb{T}$ carries a non-discrete
topology, we only get the algebraic torsion elements. This is a dense proper
subgroup of $\mathbb{T}$.
\end{example}

The situation is better for topological torsion submodules inside adelic objects.

\begin{theorem}
\label{thm_BraconnierVilenkinTypeTheorem}Suppose $A\in\mathsf{LCA}%
_{\mathcal{O}_{S}}$ is vector-free adelic.

\begin{enumerate}
\item Then $A$ is a topological torsion module.

\item Both $A$ and $A^{\vee}$ are totally disconnected.

\item For every finite place $v=P$ in $S$, the submodule $A^{P}\subseteq A$ is
a closed $\mathcal{O}_{S}$-submodule.

\item There is an isomorphism%
\begin{equation}
A\simeq\underset{P\in S\,\,}{\left.  \prod\nolimits^{\prime}\right.  }%
(A^{P}:A^{P}\cap C)\text{,} \label{lm3}%
\end{equation}
where $C$ is any compact clopen $\mathcal{O}_{F}$-submodule such that%
\begin{equation}
\bigcap_{\alpha\in\mathcal{O}_{S}^{\times}}\alpha C=0\qquad\text{and}%
\qquad\bigcup_{\alpha\in\mathcal{O}_{S}^{\times}}\alpha C=A\text{.}
\label{lm3a}%
\end{equation}
Furthermore, $A^{P}\cap C$ is a compact clopen $\mathcal{O}_{F}$-submodule of
$A^{P}$ such that%
\begin{equation}
\bigcap_{\alpha\in\mathcal{O}_{S}^{\times}}\alpha(A^{P}\cap C)=0\qquad
\text{and}\qquad\bigcup_{\alpha\in\mathcal{O}_{S}^{\times}}\alpha(A^{P}\cap
C)=A^{P}\text{.} \label{lm3b}%
\end{equation}

\item Each $A^{P}$ is (algebraically and topologically) a finite-dimensional
$\mathbf{Q}_{p}$-vector space for $(p)=P\cap\mathbf{Z}$.
\end{enumerate}
\end{theorem}

\begin{proof}
(1) As $A$ is vector-free adelic, we may write%
\begin{equation}
A=\bigcup_{x\in\mathcal{O}_{S}^{\times}}xC\qquad\text{such that}\qquad
\bigcap_{x\in\mathcal{O}_{S}^{\times}}xC=0 \label{lm3c}%
\end{equation}
for $C\subseteq A$ some compact clopen $\mathcal{O}_{F}$-submodule of $A$. For
all $x\in\mathcal{O}_{S}^{\times}$ the multiplication with $x$ is a
homeomorphism, so all $\mathcal{O}_{F}$-submodules $xC$ are open. We deduce
that the intersection of all open sub\textit{groups} of $A$ is trivial (since
the $xC$ form a cofinal subfamily). By \cite[Theorem 7.8]{MR551496} the
connected component of the identity of $A$ is trivial. Then, by \cite[Theorem
7.3]{MR551496} it follows that $A$ is totally disconnected, and then by
\cite[Theorem 7.7]{MR551496} that every neighbourhood $0\in U$ of the identity
contains a compact open subgroup $W_{U}\subseteq U$. Being open, the quotient
$A/W_{U}$ is a discrete LCA group.\ In the category of LCA\ groups, we deduce%
\[
A/W_{U}=\bigcup_{x\in\mathcal{O}_{S}^{\times}}(xC/W_{U})
\]
and since the image of $xC$ is compact in $A/W_{U}$, it must be finite. We
deduce that $xC/W_{U}$ is a torsion group, and then $A/W_{U}$ is a discrete
torsion group. From \cite[Theorem 3.5 (e)]{MR637201} it follows that $A$ is
topologically torsion, and by (c) loc. cit. that $A^{\vee}$ is totally
disconnected. This shows (1) and (2). Write $A^{p}$ for the topological
$p$-torsion elements in the sense of $\mathsf{LCA}_{\mathbf{Z}}$ (i.e., in the
sense of the theory of LCA groups). Use \cite[Lemma 3.8, Armacost writes
subscript $p$ for our superscript $p$]{MR637201} to deduce (3): At first, the
topological $p$-torsion subgroup $A^{p}$ is shown to be a closed subgroup and
then proven to be a topological $\mathbf{Z}_{p}$-module. From this, we see
that the $\mathcal{O}_{F}$-module structure $\mathcal{O}_{F}\times
A\rightarrow A$ extends to a topological $\mathcal{O}_{F}\otimes_{\mathbf{Z}%
}\mathbf{Z}_{p}$-module structure with respect to the $p$-adic topology%
\begin{equation}
\left(  \underset{n}{\underleftarrow{\lim}}\mathcal{O}_{F}\otimes_{\mathbf{Z}%
}\mathbf{Z}/p^{n}\mathbf{Z}\right)  \times A^{P}\longrightarrow A^{P}\text{.}
\label{ld2}%
\end{equation}
As $\mathcal{O}_{F}$ is a finite $\mathbf{Z}$-algebra, $\mathcal{O}_{F}%
\otimes_{\mathbf{Z}}\mathbf{Z}_{p}\cong\prod_{P}\widehat{\mathcal{O}_{F,P}}$,
where $P$ runs through the primes in $\mathcal{O}_{F}$ lying over $p$. This
canonical splitting provides idempotents to canonically split $A^{p}$ into the
$A^{P}$ of the claim, $\varphi\colon A^{p}\cong\prod_{P}A^{P}$ in
$\mathsf{LCA}_{\mathcal{O}_{F}}$ (the same arguments show that $(A\cap
C)^{p}\cong\prod_{P}(A\cap C)^{P}$ in $\mathsf{LCA}_{\mathcal{O}_{F}}$). If
$x$ is topological $p$-torsion (or $P$-torsion), then so is $\alpha x$ for any
$\alpha\in\mathcal{O}_{S}$, so $A^{P}$ determines an object in $\mathsf{LCA}%
_{\mathcal{O}_{S}}$ and the isomorphism $\varphi$ also holds in $\mathsf{LCA}%
_{\mathcal{O}_{S}}$ (this part of the argument does not work for $A\cap C$,
since $C$ is not necessarily an $\mathcal{O}_{S}$-module). Moreover, as
$A^{p}$ is closed, this also shows that $A^{P}$ is closed inside $A$, proving
(3). A variant of \cite[Theorem 3.10]{MR637201} then shows that there exists
an isomorphism in $\mathsf{LCA}_{\mathcal{O}_{F}}$ such that%
\begin{equation}
A\cap C\cong\underset{P}{\left.  \prod\right.  }(A\cap C)^{P}\text{,}
\label{ld1}%
\end{equation}
where $P$ runs over all maximal ideals of $\mathcal{O}_{F}$. Then
\cite[Theorem 3.13]{MR637201} yields the isomorphism in Eq. \ref{lm3}, but
with two limitations: (L1) $P$ is running over \textit{all} maximal
ideals\footnote{as opposed to running only over those in $S$} and (L2) the
choice of $C$ is the one of Eq. \ref{lm3c}. We remove L2: Note that if
$C_{2}\subseteq C$ is a different choice in Eq. \ref{lm3c}, then $C/C_{2}$ is
discrete (since $C_{2}$ is open) and compact\ (as $C$ is compact), and
therefore finite. Thus, any two choices for $C$, call them $C_{1},C_{2}$, are
both finite-index in $C_{1}+C_{2}$ (which is also open and compact). While
changing $C$ affects Eq. \ref{ld1}, the restricted product $A\simeq
\underset{P}{\left.  \prod\nolimits^{\prime}\right.  }(A^{P},A^{P}\cap C)$ is
invariant both as a group and topologically under modifying $A^{P}\cap C$ by
chains of finite index replacements. This removes L2. We address L1 next. Eq.
\ref{lm3b} follows from $A$ being adelic. Next, we show that $P\notin S$
implies $A^{P}=0$. To see this, note that $(A\cap C)^{P}$ is both an
$\widehat{\mathcal{O}_{F,P}}$-module (by being topological $P$-torsion) and an
$\mathcal{O}_{F}$-module (as both $A$ and $C$ are), but not necessarily an
$\mathcal{O}_{S}$-module. However, in $\widehat{\mathcal{O}_{F,P}}$ all primes
except for $P$ got localized away, so if $P\notin S$, we have the
factorization of rings $\mathcal{O}_{F}\rightarrow\mathcal{O}_{S}%
\rightarrow\mathcal{O}_{F,P}$, so in this case $(A\cap C)^{P}$ \textit{is} an
$\mathcal{O}_{S}$-module. It follows that $\alpha(A^{P}\cap C)=A^{P}\cap C$
for all $\alpha\in\mathcal{O}_{S}^{\times}$ and therefore%
\[
\bigcap_{\alpha\in\mathcal{O}_{S}^{\times}}\alpha(A^{P}\cap C)=A^{P}\cap
C\qquad\text{and}\qquad\bigcup_{\alpha\in\mathcal{O}_{S}^{\times}}\alpha
(A^{P}\cap C)=A^{P}\cap C\text{.}%
\]
By Eq. \ref{lm3b} we deduce that $A^{P}=0$. This removes L1 and hence finally
settles Eq. \ref{lm3} in the format as stated in the theorem, thus (4) is
shown. For (5):\ Let $p$ be the prime lying below $P$. By construction%
\[
T:=A^{P}\cap C
\]
is a topological $\widehat{\mathcal{O}_{F,P}}$-module (see Eq. \ref{ld2} and
the remarks below it about adapting it to $A^{P}\cap C$). Similarly, as $P\in
S$, $A^{P}$ is an $\widehat{\mathcal{O}_{F,P}}\otimes_{\mathcal{O}_{F}%
}\mathcal{O}_{S}$-module. Note that $\widehat{\mathcal{O}_{F,P}}%
\otimes_{\mathcal{O}_{F}}\mathcal{O}_{S}=\operatorname*{Frac}\widehat
{\mathcal{O}_{F,P}}$, so $A^{P}$ is an algebraic\footnote{in the sense of
Remark \ref{rmk_AlgebraicModules}} vector space over (a finite extension of)
$\mathbf{Q}_{p}$. Note that at this point we do \textit{not yet know} that it
is a topological $\mathbf{Q}_{p}$-vector space: multiplication by $\frac{1}%
{p}$ could fail to be continuous\footnote{Depending on $S$, there can exist
$\mathbf{Q}$-vector spaces in $\mathsf{LCA}_{\mathcal{O}_{S}}$ such that
multiplication by a prime $p$ is continuous, but by $\frac{1}{p}$ is not. In
such cases, even though multiplication by $p$ is invertible, it is not a
homeomorphism. See \cite[Example 2.8]{kthyartin} for $F=\mathbf{Q}$ and
$S=\{\infty\}$, so $\mathcal{O}_{S}=\mathbf{Z}$.}. We show this now: First, we
claim that $A^{P}$ is $\sigma$-compact\footnote{a countable union of compact
sets}. This is clear because%
\begin{equation}
\bigcup_{\alpha\in\mathcal{O}_{S}^{\times}}\alpha(A^{P}\cap C)=\bigcup
_{\alpha\in\mathcal{O}_{S}^{\times}}\alpha T=A^{P} \label{lu1}%
\end{equation}
of Eq. \ref{lm3b} presents it as a countable union of the compact sets
$\alpha(A^{P}\cap C)$ (using that $\alpha\in\mathcal{O}_{S}^{\times}$ acts
through homeomorphisms and $A^{P}\cap C$ is compact). Since $A^{P}\subseteq A$
is closed, $A^{P}$ is therefore locally compact and $\sigma$%
-compact\footnote{Neither of these properties implies the other. The rationals
$\mathbf{Q}$ with the real metric topology fail to be locally compact, but are
countable, hence $\sigma$-compact. On the other hand, any uncountable discrete
group fails to be $\sigma$-compact.}. As $p$ is invertible, $(-\cdot p)\colon
A^{P}\twoheadrightarrow A^{P}$ is a continuous surjective map with $\sigma
$-compact source. By Pontryagin's Open Mapping Theorem \cite[p. 23, Theorem
3]{MR0442141}, it follows that this must be an open map and therefore a
homeomorphism. This settles that $A^{P}$ is a topological $\mathbf{Q}_{p}%
$-vector space. It remains to show that it is finite-dimensional. As $p$ acts
through a homeomorphism on $A^{P}$ and $T$ is open in $A^{P}$, $pT$ will be
open in $T$, too. Hence, $T/pT$ is discrete. However, $T$ is also compact, so
$T/pT$ is both compact and discrete and therefore finite. Thus, $T$ is a
compact $\widehat{\mathcal{O}_{F,P}}$-module, and $T/\mathfrak{m}%
_{\widehat{\mathcal{O}_{F,P}}}T$ (which is a quotient of $T/pT$) is finite.
Therefore, by the topological Nakayama Lemma (\cite[Ch. 5, \S 1]{MR1029028}),
$T$ is a finitely generated $\widehat{\mathcal{O}_{F,P}}$-module. By the right
side of Eq. \ref{lu1}, $A^{P}$ is merely the localization of $T$ by inverting
$p$, so $\dim_{\mathbf{Q}_{p}}A^{P}$ is at most the number of $\mathbf{Z}_{p}%
$-module generators of $T$, which was finite\footnote{In fact, as $T\subseteq
A^{P}$ cannot have $p$-torsion since $A^{P}$ is a $\mathbf{Q}_{p}$-vector
space, $\dim_{\mathbf{Q}_{p}}A^{P}=\operatorname*{rk}_{\mathbf{Z}_{p}}%
T\leq\dim_{\mathbf{F}_{p}}T/pT$, but we do not need this stronger fact.}.
\end{proof}

\begin{example}
\label{example_AdeleSequence}Define the $S$\emph{-integral ad\`{e}les} as the
restricted product%
\begin{equation}
\mathbb{A}_{S}:=\left.  \underset{v\in S\;}{\prod\nolimits^{\prime}}\right.
F_{v}\text{,} \label{lcvn1}%
\end{equation}
with the usual restricted product topology. Then $\mathbb{A}_{S}%
\in\mathsf{LCA}_{\mathcal{O}_{S},ad}$. As a fact from classical number theory,
there is an exact sequence%
\begin{equation}
\mathcal{O}_{S}\hookrightarrow\mathbb{A}_{S}\twoheadrightarrow\mathbb{A}%
_{S}/\mathcal{O}_{S} \label{l_adele_sequence}%
\end{equation}
in $\mathsf{LCA}_{\mathcal{O}_{S}}$, where the left arrow is the diagonal
embedding. Here $\mathcal{O}_{S}$ is discrete and torsion-free, while
$\mathbb{A}_{S}/\mathcal{O}_{S}$ is compact and connected. This is the
\emph{ad\`{e}le sequence}. Even though Eq. \ref{lcvn1} also makes sense if $S$
does not contain all infinite places, Sequence \ref{l_adele_sequence} really
rests on this running assumption.
\end{example}

\begin{remark}
\label{rmk_ClassicalNotationForRestrictedProducts}Eq. \ref{lcvn1} relies on a
common and useful shorthand. Nonetheless, it is a little imprecise. Really, we
should specify the ad\`{e}les as $\mathbb{A}_{S}:=\underset{v\in S_{\infty}%
}{\prod}F_{v}\times\underset{v\in S\setminus S_{\infty}\;}{\prod
\nolimits^{\prime}}(F_{v}:\mathcal{O}_{v})$, where the second factor is a
restricted product with respect to the compact clopen subgroups $\mathcal{O}%
_{v}$.
\end{remark}

\subsection{Restricted product of
categories\label{topsect_RestrictedProductsOfCategories}}

\subsubsection{Exact structure\label{subsect_RestrictedProductAsExactCategory}%
}

Suppose $(\mathsf{C}_{v})_{v\in I}$ are exact categories, indexed by some set
$I$. Then $\prod_{v\in I}\mathsf{C}_{v}$ is an exact category. Its objects are
arrays $X=(X_{v})_{v\in I}$ such that $X_{v}\in\mathsf{C}_{v}$, morphisms are
entrywise, and a kernel-cokernel sequence is called exact if it is exact
entrywise in each $\mathsf{C}_{v}$. Suppose there are exact functors $\xi
_{v}\colon\mathsf{C}_{v}\rightarrow\mathsf{D}_{v}$ with $(\mathsf{C}%
_{v})_{v\in I},(\mathsf{D}_{v})_{v\in I}$ having the same indexing set. For
any subset $I^{\prime}\subseteq I$, define $\mathsf{J}^{(I^{\prime})}%
:=\prod_{v\in I\setminus I^{\prime}}\mathsf{C}_{v}\times\prod_{v\in I^{\prime
}}\mathsf{D}_{v}$. Then whenever $I^{\prime}\subseteq I^{\prime\prime}$, there
is an exact functor $\xi_{I^{\prime},I^{\prime\prime}}\colon\mathsf{J}%
^{(I^{\prime})}\rightarrow\mathsf{J}^{(I^{\prime\prime})}$, defined by taking
$\xi_{v}$ entrywise for all $v\in I^{\prime\prime}\setminus I^{\prime}$, and
being the identity elsewhere. We define the \emph{restricted product} of exact
categories as%
\begin{equation}
\underset{v\in I\;}{\prod\nolimits^{\prime}}(\mathsf{D}_{v}:\mathsf{C}%
_{v}):=\left.  \underset{I^{\prime}\subseteq I\text{, }\#I^{\prime}<\infty
}{\operatorname*{colim}}\left.  \mathsf{J}^{(I^{\prime})}\right.  \right.
\text{,} \label{l_RestrictedProductAsColimitJFormula}%
\end{equation}
where the $\xi_{v}$ are implicit, and the notation is (although perhaps deemed
unnaturally ordered), inspired from the traditional notation in topological
group theory. See Remark \ref{rmk_ClassicalNotationForRestrictedProducts}.

We note that if $\mathsf{C}_{v},\mathsf{D}_{v}$ are exact categories with
duality (\cite[\S 2]{MR2600285}), and the $\xi_{v}$ are duality preserving
exact functors, then the restricted product is an exact category with duality.

\subsubsection{Monoidal
structure\label{subsect_MonoidalStructureOnRestrictedProductsOfCategories}}

Suppose the categories $(\mathsf{C}_{v})_{v\in I}$ of
\S \ref{subsect_RestrictedProductAsExactCategory} additionally carry a
bi-exact (symmetric) monoidal structure $\otimes_{v}\colon\mathsf{C}_{v}%
\times\mathsf{C}_{v}\rightarrow\mathsf{C}_{v}$. This means that for any fixed
object $C\in\mathsf{C}_{v}$ the functor $\mathsf{C}_{v}\rightarrow
\mathsf{C}_{v}$ given by $X\mapsto C\otimes_{v}X$ is exact, and analogously
for $X\mapsto X\otimes_{v}C$. Moreover, suppose the $(\mathsf{D}_{v})_{v\in
I}$ are also equipped with bi-exact (symmetric) monoidal structures. If we
choose the functors $\xi_{v}\colon\mathsf{C}_{v}\rightarrow\mathsf{D}_{v}$ of
\S \ref{subsect_RestrictedProductAsExactCategory} to be (symmetric) monoidal
functors, then this renders $\prod\nolimits^{\prime}(\mathsf{D}_{v}%
:\mathsf{C}_{v})$ a bi-exact (symmetric) monoidal category. The monoidal
structure is taken elementwise and the functors $\xi_{I^{\prime}%
,I^{\prime\prime}}$ of \S \ref{subsect_RestrictedProductAsExactCategory} can
(induced from the $\xi_{v}$) be given a natural (symmetric) monoidal structure.

\subsection{Structure of adelic modules}

For formulating the following theorem, we use the convention that
$\mathcal{O}_{v}:=F_{v}$ if $v$ is an infinite place.

\begin{proposition}
\label{prop_IdentifyAdelicBlocks}Suppose $S\supseteq S_{\infty}$ is a set of
places of $F$, possibly infinite. Then there is an exact and
duality-preserving equivalence of exact categories with duality%
\begin{equation}
\Psi\colon\underset{v\in S\;}{\prod\nolimits^{\prime}}(\mathsf{Proj}%
_{F_{v},fg}:\mathsf{Proj}_{\mathcal{O}_{v},fg})\overset{\sim}{\rightarrow
}\mathsf{LCA}_{\mathcal{O}_{S},ad}\text{.} \label{lmx1}%
\end{equation}
Here the duality on the right side is Pontryagin duality $(X\mapsto X^{\vee}%
)$, on the left it is induced from%
\[
M\mapsto\operatorname*{Hom}\nolimits_{\mathcal{O}_{v}}(M,\mathcal{O}%
_{v})\qquad\text{resp.}\qquad M\mapsto\operatorname*{Hom}\nolimits_{F_{v}%
}(M,F_{v})\text{,}%
\]
and $\xi_{v}(M):=F_{v}\otimes_{\mathcal{O}_{v}}M$. The functor $\Psi$ sends an
array of modules $(M_{v})_{v\in S}$ to itself, but equipped with the real
topology for $v\in S_{\infty}$, and with the adic topology for $v\in
S\setminus S_{\infty}$.
\end{proposition}

\begin{proof}
[Proof of Prop. \ref{prop_IdentifyAdelicBlocks}]We refer the reader to
\cite[\S 4, esp. Prop. 4.8]{kthyartin}. This discusses the proof for $S$ being
all places, but is easy to adapt to the above (more general) version. Just
replace the Braconnier--Vilenkin type theorem \cite[\S 3]{kthyartin} by Thm.
\ref{thm_BraconnierVilenkinTypeTheorem} above.
\end{proof}

\begin{corollary}
\label{cor_LCAAdQuasiAbelian}$\mathsf{LCA}_{\mathcal{O}_{S},ad}$ is a
split-exact quasi-abelian category. If $S$ is finite, there is an equivalence
$\mathsf{LCA}_{\mathcal{O}_{S},ad}\overset{\sim}{\rightarrow}\mathsf{Mod}%
_{R,fg}$, where $R:=\prod_{v\in S}F_{v}$ is the product ring.
\end{corollary}

\begin{proof}
The category on the left side in Eq. \ref{lmx1} has kernels (since each
$\mathcal{O}_{v}$ is hereditary, i.e., any kernel in $\mathsf{Proj}%
_{\mathcal{O}_{v},fg}$ must again be projective). As the category is self-dual
to its opposite, also all cokernels exist. Hence, it is quasi-abelian. Next,
if each $\mathsf{C}_{v}$, $\mathsf{D}_{v}$ is split-exact, so is the
restricted product. If $S$ is finite, the restricted product simplifies to
$\prod_{v\in S}\mathsf{Mod}_{F_{v},fg}$, giving the second claim.
\end{proof}

If $S$ is infinite, $\mathsf{LCA}_{\mathcal{O}_{S},ad}$ will fail to agree
with the module category over $\prod_{v\in S}^{\prime}F_{v}$. See
\cite[Example 4.11]{kthyartin}. Even worse, for $S$ infinite, $\mathsf{LCA}%
_{\mathcal{O}_{S},ad}$ will not be an abelian category.

\begin{example}
If $S$ is infinite, there exists no object $X\in\mathsf{LCA}_{\mathcal{O}%
_{S},ad}$ such that any object is a quotient of $X^{\oplus n}$ for $n$ big enough.
\end{example}

\begin{remark}
\label{rmk_SymmetricMonoidalStructureOnAdelicModules}Prop.
\ref{prop_IdentifyAdelicBlocks} can be used to equip $\mathsf{LCA}%
_{\mathcal{O}_{S},ad}$ with a bi-exact symmetric monoidal structure. The
functors $\xi_{v}$ loc. cit. extend to symmetric monoidal functors $\xi
_{v}\colon(\mathsf{Proj}_{\mathcal{O}_{v},fg},\otimes_{\mathcal{O}_{v}%
})\rightarrow(\mathsf{Proj}_{F_{v},fg},\otimes_{F_{v}})$. As projective
modules are flat, the involved monoidal structures are all bi-exact. The
general construction indicated in
\S \ref{subsect_MonoidalStructureOnRestrictedProductsOfCategories} therefore
equips the left side in Eq. \ref{lmx1} with a bi-exact symmetric monoidal
structure. Now transport this to $\mathsf{LCA}_{\mathcal{O}_{S},ad}$ via the
equivalence of categories $\Psi$ of Prop. \ref{prop_IdentifyAdelicBlocks}. We
point out that this monoidal structure is compatible with the one in
\cite{MR2329311}.
\end{remark}

\subsection{Structure of locally compact modules}

\begin{theorem}
\label{thm_MainDecomp}Suppose $X\in\mathsf{LCA}_{\mathcal{O}_{S}}$ is arbitrary.

\begin{enumerate}
\item There exists a (non-canonical) exact sequence%
\[
Q\hookrightarrow X\twoheadrightarrow D\text{,}%
\]
in $\mathsf{LCA}_{\mathcal{O}_{S}}$, where $Q$ is quasi-adelic and $D$ discrete.

\item If $X$ is quasi-adelic, there exists a (canonical) exact sequence%
\[
K\hookrightarrow X\twoheadrightarrow A
\]
in $\mathsf{LCA}_{\mathcal{O}_{S}}$, where $K$ is compact and $A$ adelic. If
$X$ is vector-free, so are $K$ and $A$.

\item If $P$ is a discrete $\mathcal{O}_{S}$-module and projective in
$\mathsf{Mod}_{\mathcal{O}_{S}}$, then it is a projective object in
$\mathsf{LCA}_{\mathcal{O}_{S}}$.

\item If $I$ is a compact $\mathcal{O}_{S}$-module and its dual $I^{\vee}$ is
projective in $\mathsf{Mod}_{\mathcal{O}_{S}}$, then it is an injective object
in $\mathsf{LCA}_{\mathcal{O}_{S}}$.

\item Vector modules are injective and projective objects in $\mathsf{LCA}%
_{\mathcal{O}_{S}}$.
\end{enumerate}
\end{theorem}

\begin{proof}
(5) The ring $\mathcal{O}_{S}$ is a localization of $\mathcal{O}_{F}$. The
forgetful functor $\mathsf{LCA}_{\mathcal{O}_{S}}\rightarrow\mathsf{LCA}%
_{\mathcal{O}_{F}}$ is fully faithful, exact and reflects exactness. If $X$ is
a vector module, it is injective and projective in $\mathsf{LCA}%
_{\mathcal{O}_{F}}$ by \cite[Theorem 5.13]{MR4028830}. By fullness, our claim
follows. (3) As $P$ is projective in $\mathsf{Mod}_{\mathcal{O}_{S}}$, there
exists an algebraic lift to an exact sequence in $\mathsf{LCA}_{\mathcal{O}%
_{S}}$. As $P$ is discrete, any preimage under this map is necessarily open,
so the lift is continuous and thus a map in $\mathsf{LCA}_{\mathcal{O}_{S}}$.
(4)\ Follows from (3) under Pontryagin duality, as it exchanges injective and
projective objects (any categorical equivalence to the opposite category does
this). (1) By \cite[Lemma 6.5]{MR4028830} there is an exact sequence%
\begin{equation}
V\oplus C\hookrightarrow X\overset{q}{\twoheadrightarrow}D \label{lm1}%
\end{equation}
in $\mathsf{LCA}_{\mathcal{O}_{F}}$, where $V$ is a vector module,
$C\in\mathsf{LCA}_{\mathcal{O}_{F}}$ is compact and $D\in\mathsf{LCA}%
_{\mathcal{O}_{F}}$ discrete. As $V$ is an injective object in $\mathsf{LCA}%
_{\mathcal{O}_{F}}$ (using \cite[Theorem 5.13]{MR4028830} once more), we may
write $X\simeq X^{\prime}\oplus V$ in $\mathsf{LCA}_{\mathcal{O}_{F}}$.
Naturally, $V$ is automatically an $\mathcal{O}_{S}$-module (even an
$F$-vector space). Even though $X^{\prime}$ is a priori only an $\mathcal{O}%
_{F}$-module, it must be an (algebraic) $\mathcal{O}_{S}$-module (suppose for
some $\alpha\in\mathcal{O}_{F}$ with $\alpha^{-1}\in\mathcal{O}_{S}$ and
$y^{\prime}\in X^{\prime}$ we have $\alpha^{-1}y^{\prime}=x^{\prime}+v$ with
respect to the direct sum. Then $y^{\prime}=\alpha x^{\prime}+\alpha v$ forces
$\alpha v=0$ in $V$. But as a vector module, $V$ is also an $F$-module, so no
$\mathcal{O}_{F}$-torsion is possible, and we must have had $v=0$). As $X$ was
a topological $\mathcal{O}_{S}$-module, the scalar action of $\mathcal{O}_{S}$
restricted to $X^{\prime}$ is also continuous. Having split off $V$ as a
direct summand, it suffices to prove our claim for vector-free $X$. We return
to Sequence \ref{lm1}, this time with $V=0$. As $D$ is discrete, $C=q^{-1}(D)$
is open. Define $Q=\mathcal{O}_{S}\cdot C$ inside $X$. As $C$ is an
$\mathcal{O}_{F}$-module and $\mathcal{O}_{S}$ a localization of
$\mathcal{O}_{F}$, we have%
\[
Q=\bigcup_{x\in\mathcal{O}_{S}}xC=\bigcup_{x\in\mathcal{O}_{S}^{\times}%
}xC\text{.}%
\]
As $X$ is a topological $\mathcal{O}_{S}$-module, the multiplication with
$x\in\mathcal{O}_{S}^{\times}$ is a homeomorphism, and since $C$ is open, so
is each $xC$. Thus, $Q$ is (cl)open in $X$, and an $\mathcal{O}_{S}$-module by
construction. Hence, $Q\in\mathsf{LCA}_{\mathcal{O}_{S}}$ and we may form an
exact sequence $Q\hookrightarrow X\twoheadrightarrow D^{\prime}$ in
$\mathsf{LCA}_{\mathcal{O}_{S}}$. As $Q$ is open, $D^{\prime}$ must be a
discrete object in $\mathsf{LCA}_{\mathcal{O}_{S}}$. This settles (1). (2) If
$Q$ is vector-free quasi-adelic, we may write%
\[
Q=\bigcup_{x\in\mathcal{O}_{S}^{\times}}xC
\]
for $C\subseteq Q$ a compact clopen $\mathcal{O}_{F}$-submodule. Define
$K:=\bigcap_{x\in\mathcal{O}_{S}^{\times}}xC$. As $C$ is compact, each $xC$ is
compact and thus closed in $Q$. Thus, as an intersection of closed sets, $K$
is a closed $\mathcal{O}_{F}$-submodule of $C$ (as $1\cdot C$ is in the
intersection). Since $C$ is compact, so is $K$. Next, note that $K$ is
actually an $\mathcal{O}_{S}$-module: As $\mathcal{O}_{F}\rightarrow
\mathcal{O}_{S}$ is a localization, for every $k\in K$ and $\alpha
\in\mathcal{O}_{S}$ we find $\alpha=\frac{f}{s}$ with $f\in\mathcal{O}_{F}$
and $s\in\mathcal{O}_{S}^{\times}$ and then $\alpha k=s^{-1}(fk)$ with $fk\in
K$ ($\subseteq C$), so $s^{-1}(fk)\in s^{-1}C\subseteq K$. Defining $A:=Q/K$,
we obtain an exact sequence $K\hookrightarrow Q\overset{\pi}%
{\twoheadrightarrow}A$ in $\mathsf{LCA}_{\mathcal{O}_{S}}$. The object $A$ is
vector-free, for otherwise a vector module summand would lift to a vector
module summand of $Q$ since vector modules are projective objects (as we had
shown above). The image of $C$ under $\pi$ is again compact and an
$\mathcal{O}_{F}$-submodule. As $\pi$ is an admissible epic, it is an open
map, so the open $C\subseteq Q$ gets mapped to the open $\pi(C)$ in $A$. Thus,
$\pi(C)$ is a compact clopen $\mathcal{O}_{F}$-submodule in $A$. As $\pi$ is
surjective,%
\[
A=\mathcal{O}_{S}\cdot\pi(C)=\left\{  \alpha c\mid\alpha\in\mathcal{O}%
_{S},c\in\pi(C)\right\}  \text{.}%
\]
We deduce that $A$ is vector-free quasi-adelic. Next, we show that $A$ is
adelic: Suppose $\overline{t}\in\bigcap_{x\in\mathcal{O}_{S}^{\times}}x\pi
(C)$. Then for all $x\in\mathcal{O}_{S}^{\times}$, there is $c_{x}\in C$ such
that $\overline{t}=x\pi(c_{x})$. Let $t\in Q$ be a preimage of $\overline{t}$.
Then for all $x\in\mathcal{O}_{S}^{\times}$, we get $t-xc_{x}\in K$. Write
$t=xc_{x}+k_{x}$ with $k_{x}\in K$. As $K$ is an $\mathcal{O}_{S}$-module and
$x^{-1}\in\mathcal{O}_{S}^{\times}$, we can write $k_{x}=x\tilde{c}_{x}$ for
$\tilde{c}_{x}:=x^{-1}k_{x}\in K\subseteq C$. We obtain%
\[
t=xc_{x}+x\tilde{c}_{x}=x(c_{x}+\tilde{c}_{x})\qquad\text{for all}\qquad
x\in\mathcal{O}_{S}^{\times}\text{.}%
\]
Thus, $t\in K$ and therefore $\overline{t}=0$. It follows that $A$ is adelic.
Hence, letting $X:=Q$, we have shown (2) if $Q$ is vector-free. If $Q$ has a
non-trivial vector module summand, apply (2) to the vector-free complementary
direct summand. This shows (2) in the general case.
\end{proof}

\begin{example}
\label{example_SumOfClosedSubgroups}If $G$ is a topological group and
$A,B\subseteq G$ closed subgroups, $A+B$ need not be closed. For example,
$\mathbf{Z}\mathbb{+}\sqrt{2}\mathbf{Z}$ inside $\mathbf{R}$ demonstrates
this. However, if one of the groups is additionally assumed to be compact,
$A+B$ is also closed.
\end{example}

\begin{proposition}
\label{prop_icg_and_qa}We record the following relations among these categories:

\begin{enumerate}
\item $\mathsf{LCA}_{\mathcal{O}_{S},qa}\subseteq\mathsf{LCA}_{\mathcal{O}%
_{S},icg}$.

\item Any morphism in $\mathsf{LCA}_{\mathcal{O}_{S}}$ from an ind-c.g. object
to a discrete $\mathcal{O}_{S}$-module has finitely generated image.

\item Any morphism in $\mathsf{LCA}_{\mathcal{O}_{S}}$ from a quasi-adelic
object to a discrete $\mathcal{O}_{S}$-module has finite image.

\item The full subcategory inclusion $\mathsf{LCA}_{\mathcal{O}_{S}%
,icg}\subseteq\mathsf{LCA}_{\mathcal{O}_{S}}$ is left $s$-filtering.

\item The full subcategory $\mathsf{LCA}_{\mathcal{O}_{S},icg}\subseteq
\mathsf{LCA}_{\mathcal{O}_{S}}$ is extension-closed.

\item The full subcategory $\mathsf{LCA}_{\mathcal{O}_{S},qa}\subseteq
\mathsf{LCA}_{\mathcal{O}_{S}}$ is extension-closed.

\item The full subcategory $\mathsf{LCA}_{\mathcal{O}_{S},ad}\subseteq
\mathsf{LCA}_{\mathcal{O}_{S}}$ is extension-closed.
\end{enumerate}
\end{proposition}

\begin{proof}
(1) Suppose $X$ is quasi-adelic. Write $X\simeq V\oplus Q$ with $V$ a vector
module and $Q=\mathcal{O}_{S}\cdot C$ with $C$ a compact clopen $\mathcal{O}%
_{F}$-submodule of $Q$. Define $H:=V+C$. This is a closed $\mathcal{O}_{F}%
$-submodule in $X$ as $V$ is closed and $C$ compact in $X$ (see Ex.
\ref{example_SumOfClosedSubgroups}). We find $X/H\cong Q/C$, which is discrete
as $C$ was open in $Q$, and thus $H$ is (cl)open in $X$. The sum
$K:=\overline{B_{V,1}(0)}+C$ of a closed unit ball in $V$ (with respect to any
real vector space norm on $V$ to our liking) is a witness that $H$ is c.g. (2)
Suppose $X$ is ind-c.g. and we are given $g\colon X\rightarrow D$ with $D$
discrete. Pick $H\subseteq X$ as in Def. \ref{def_indcg}. Then as an LCA group
$H\simeq\mathbf{R}^{n}\oplus\mathbf{Z}^{m}\oplus C$ with $C$ a compact
LCA\ group \cite[Thm. 2.5]{MR0215016}. Clearly, $g(\mathbf{R}^{n})=0$ and
$g(C)$ must be finite in $D$ as $\mathbf{R}^{n}$ is connected and $C$ compact.
Thus, $g(H)$ is a finitely generated abelian group. Write $\mathcal{O}%
_{S}\left\langle h(H)\right\rangle $ for the $\mathcal{O}_{S}$-module span of
$g(H)$. Then $g(X)=$ $\mathcal{O}_{S}\left\langle g(H)\right\rangle $ and this
is finitely generated as an $\mathcal{O}_{S}$-module. (3) Exactly as the proof
for (2), but this time $m=0$. (4) We show left filtering: Suppose $f\colon
X\rightarrow Y$ is given with $X\in\mathsf{LCA}_{\mathcal{O}_{S},icg}$. Use
Theorem \ref{thm_MainDecomp} to get the left diagram in%
\begin{equation}%
\xymatrix{
& Q \ar@{^{(}->}[d] \\
X \ar[r]_{f} \ar[dr]_{h} & Y \ar@{->>}[d] \\
& D
}%
\qquad\qquad\qquad%
\xymatrix{
& Q^{\prime} \ar@{^{(}->}[d] \\
X \ar[r] \ar[dr]_{0} \ar@{-->}[ur] & Y \ar@{->>}[d] \\
& D/{\mathcal{O}_{S} \left\langle h(H)\right\rangle}
}
\label{lwca1}%
\end{equation}
with $Q$ quasi-adelic and $D$ discrete. Pick $H\subseteq X$ as in Def.
\ref{def_indcg}. By part (2) of the proposition, we deduce that $\mathcal{O}%
_{S}\left\langle h(H)\right\rangle $ is finitely generated. Defining
$Q^{\prime}:=\ker(Y\twoheadrightarrow D/\mathcal{O}_{S}\left\langle
h(H)\right\rangle )$, we get the right diagram in Eq. \ref{lwca1}. We claim
that $Q^{\prime}$ is ind-c.g. To see this, note that the filtration
$Q\hookrightarrow Q^{\prime}\hookrightarrow Y$ induces the exact sequence
$Q^{\prime}/Q\hookrightarrow Y/Q\overset{q}{\twoheadrightarrow}Y/Q^{\prime}$,
but in view of Diagram \ref{lwca1} the arrow $q$ agrees with
$D\twoheadrightarrow D/\mathcal{O}_{S}\left\langle h(H)\right\rangle $, so we
obtain $Q^{\prime}/Q\cong\mathcal{O}_{S}\left\langle h(H)\right\rangle $. We
obtain the exact sequence%
\begin{equation}
Q\overset{t}{\hookrightarrow}Q^{\prime}\twoheadrightarrow\mathcal{O}%
_{S}\left\langle h(H)\right\rangle \label{lwca2}%
\end{equation}
with discrete quotient on the right. As $Q$ is quasi-adelic, it is ind-c.g. by
part (1) of the proposition, and we pick $\widehat{H}\subseteq Q$ as in Def.
\ref{def_indcg}, $Q=\bigcup_{\alpha\in\mathcal{O}_{S}^{\times}}\alpha
\widehat{H}$. Pick elements $b_{1},\ldots,b_{r}\in Q^{\prime}$ mapping to a
set of $\mathcal{O}_{S}$-module generators of $\mathcal{O}_{S}\left\langle
h(H)\right\rangle $. Define $H^{\prime}:=\widehat{H}+\mathcal{O}%
_{F}\left\langle b_{1},\ldots,b_{r}\right\rangle \subseteq Q^{\prime}$. As
$\widehat{H}$ is open in $Q$, and $Q$ is open in $Q^{\prime}$ (as the cokernel
of $t$ is discrete in Eq. \ref{lwca2}), $\widehat{H}$ is open in $Q^{\prime}$
and therefore so is $H^{\prime}$. As $\widehat{H}$ is c.g., so is $H^{\prime}$
(\cite[Thm. 2.5]{MR0215016}). We find $Q^{\prime}=\bigcup_{\alpha
\in\mathcal{O}_{S}^{\times}}\alpha H^{\prime}$, so $Q^{\prime}$ turns out to
be ind-c.g. Hence, the right diagram in Eq. \ref{lwca1} proves the left
filtering property. Next, we show left specialness: Suppose we are given an
exact sequence $X^{\prime}\hookrightarrow X\overset{q}{\twoheadrightarrow
}X^{\prime\prime}$ with $X^{\prime\prime}$ ind-c.g. Use Theorem
\ref{thm_MainDecomp} to obtain the middle column in the left diagram in%
\begin{equation}%
\xymatrix{
& Q \ar@{^{(}->}[d]_{i} \ar[dr]^{h} \\
X^{\prime} \ar@{^{(}->}[r] & X \ar@{->>}[d] \ar@{->>}[r]_{q} & X^{\prime
\prime} \\
& D
}%
\qquad%
\xymatrix{
& Q \ar@{^{(}->}[d]_{i} \ar[drr]^{0} \\
X^{\prime} \ar@{^{(}->}[r] & X \ar@{->>}[d] \ar@{->>}[r]_{q} & X^{\prime
\prime} \ar@{->>}[r]_-{\ell} & X^{\prime\prime} / W \\
& D \ar[urr]_{c}
}
\label{lwca3}%
\end{equation}
with $h:=q\circ i$. Pick $H\subseteq Q$ as in Def. \ref{def_indcg}. As $H$ is
open in $Q$, and $i$ is open, and $q$ is open, the set $h(H)\subseteq
X^{\prime\prime}$ is (cl)open. Thus, $h(H)=\overline{h(H)}$ is c.g. open and
an $\mathcal{O}_{F}$-module. Define $W:=\bigcup_{\alpha\in\mathcal{O}%
_{S}^{\times}}\alpha h(H)$. As each $\alpha$ acts through a homeomorphism,
each $\alpha h(H)$ is also open, so their union $W$ is also (cl)open in
$X^{\prime\prime}$. By construction, $W$ is an open $\mathcal{O}_{S}%
$-submodule of $X^{\prime\prime}$. We obtain the right diagram in Eq.
\ref{lwca3}, the lift $c$ exists and must be an admissible epic, and therefore
$X^{\prime\prime}/W$ is discrete. As $X^{\prime\prime}$ is ind-c.g., the image
of $\ell$ is a finitely generated $\mathcal{O}_{S}$-module, but $\ell$ is
surjective, so $X^{\prime\prime}/W$ is finitely generated. Define%
\[
\tilde{Q}:=Q+\mathcal{O}_{S}\left\langle g_{1},\ldots,g_{r}\right\rangle
\subseteq X\text{,}%
\]
where $g_{1},\ldots,g_{r}\in X$ are elements mapping to a finite set of
$\mathcal{O}_{S}$-module generators of $X^{\prime\prime}/W$. As $Q$ is open in
$X$, so is $\tilde{Q}$. Now use $\tilde{H}:=H+\mathcal{O}_{F}\left\langle
g_{1},\ldots,g_{r}\right\rangle \subseteq\tilde{Q}$ as a witness that
$\tilde{Q}$ is ind-c.g. We arrive at%
\begin{equation}%
\xymatrix{
K \ar@{^{(}->}[r] & \tilde{Q} \ar[d] \ar@{->>}[r] & X^{\prime\prime} \ar@
{=}[d] \\
X^{\prime} \ar@{^{(}->}[r] & X  \ar@{->>}[r] & X^{\prime\prime}
}
\label{lwca3b}%
\end{equation}
where the middle downward arrow is the (cl)open inclusion $\tilde
{Q}\hookrightarrow X$. The map $\tilde{Q}\rightarrow X^{\prime\prime}$ is
surjective by construction, and open as the composition of the open maps
$\tilde{Q}\hookrightarrow X\underset{q}{\twoheadrightarrow}X^{\prime\prime}$.
The object $K$ is defined as the kernel of the top epic on the right. It is
also ind-c.g. One may use $H_{K}:=H\cap K\subseteq K$ as a witness. As $H$ is
open in $Q$, $H_{K}$ is open in $K$ (by the definition of the subspace
topology) and $\mathcal{O}_{S}\cdot H_{K}=\mathcal{O}_{S}H\cap\mathcal{O}%
_{S}K=\tilde{Q}\cap K=K$. (5) Use \cite[Lemma 2.14]{MR3510209}, relying on
part (4) of the proposition. (6) Suppose $X\overset{f}{\hookrightarrow
}Y\twoheadrightarrow X^{\prime}$ is an exact sequence with $X,X^{\prime}$
quasi-adelic. Now repeat the first steps of the proof of claim (4). It yields
the diagram%
\[%
\xymatrix{
& Q^{\prime} \ar@{^{(}->}[d]_{i}  \\
X \ar@{^{(}->}[r]_{f} \ar@{^{(}-->}[ur] \ar[dr]_{0} & Y \ar@{->>}[d] \ar@
{->>}[r]_{q} & X^{\prime} \\
& D
}%
\]
with $Q^{\prime}$ quasi-adelic and $D$ discrete, and such that the diagonal
dashed lift exists (namely: By literally repeating the argument, we arrive at
such a diagram, but with $Q^{\prime}$ ind-c.g., but as $X$ is quasi-adelic,
$\mathcal{O}_{S}\left\langle h(H)\right\rangle $ is a finite torsion
$\mathcal{O}_{S}$-module, and taking this into consideration, we can arrange
$Q^{\prime}$ to be quasi-adelic). Moreover, unlike for (3), the map $f$ is an
admissible monic. As $\mathsf{LCA}_{\mathcal{O}_{S}}$ is idempotent complete
(it is quasi-abelian), it follows that $X\hookrightarrow Q^{\prime}$ must be
an admissible monic. The filtration $X\hookrightarrow Q^{\prime}%
\hookrightarrow Y$ gives rise to the exact sequence $Q^{\prime}%
/X\hookrightarrow Y/X\twoheadrightarrow Y/Q^{\prime}$, but $Y/Q^{\prime}\cong
D$ is discrete and therefore get the exact sequence $Q^{\prime}%
/X\hookrightarrow X^{\prime}\twoheadrightarrow D$. As $X^{\prime}$ is
quasi-adelic and $D$ discrete, part (3) of the proposition implies that $D$ is
finite. However, as we know that $Q^{\prime}$ is quasi-adelic, the sequence
$Q^{\prime}\hookrightarrow Y\twoheadrightarrow D$ now implies that $Y$ is also
quasi-adelic (e.g., take $s_{1},\ldots,s_{l}\in Y$ preimages of the elements
in $D$, and if $Q^{\prime}=\bigcup_{\alpha\in\mathcal{O}_{S}^{\times}}%
\alpha\tilde{C}$ with $\tilde{C}$ a compact clopen $\mathcal{O}_{F}%
$-submodule, then for $\widehat{C}:=\sum s_{i}C$ we get $Y=\bigcup_{\alpha
\in\mathcal{O}_{S}^{\times}}\alpha\widehat{C}$). (7) Write $X^{\prime
}\hookrightarrow X\twoheadrightarrow X^{\prime\prime}$ with $X^{\prime
},X^{\prime\prime}\in\mathsf{LCA}_{\mathcal{O}_{S},ad}$. As adelic objects are
quasi-adelic, and these are closed under extensions by part (6), $X$ is
quasi-adelic. As duals of adelic objects are again adelic, the same argument
applied to $X^{\prime\prime\vee}\hookrightarrow X^{\vee}\twoheadrightarrow
X^{\prime\vee}$ shows that $X^{\vee}$ is quasi-adelic. By Thm.
\ref{thm_MainDecomp} write $K\hookrightarrow X\twoheadrightarrow A$ with $K$
compact and $A$ adelic. Analogously, $K^{\prime}\hookrightarrow X^{\vee
}\twoheadrightarrow A^{\prime}$ for the dual. Thus, $A^{\prime\vee
}\hookrightarrow X\twoheadrightarrow K^{\prime\vee}$ is an exact sequence with
$A^{\prime\vee}$ adelic and $K^{\prime\vee}$ discrete. Maps from a
quasi-adelic object to a discrete object must have finite image by part (3).
Thus, $K^{\prime\vee}$ is finite, and therefore $K$ is finite. Now recall that
$K\subseteq X$. Since $X^{\prime},X^{\prime\prime}$ are adelic (and therefore
torsion-free), the sequence $X^{\prime}\hookrightarrow X\twoheadrightarrow
X^{\prime\prime}$, after forgetting topology, shows that $X$ is an extension
of torsion-free abelian groups, and therefore itself torsion-free\footnote{It
is a very elementary classical exercise to show that torsion-free abelian
groups are extension-closed in all abelian groups.}. As $K\subseteq X$
consists exclusively of torsion elements, we must have $K=0$. But then $X\cong
A$, showing that $X$ is adelic.
\end{proof}

\subsection{Connected modules}

We collect a few observations regarding connectedness and injective objects in
$\mathsf{LCA}_{\mathcal{O}_{S}}$. These generalize the characterizations from
\cite{MR4028830,kthyartin}. All in all, the situation is a little convoluted.
Suppose $X\in\mathsf{LCA}_{\mathcal{O}_{S}}$. We call%
\begin{equation}
\operatorname*{aco}(X):=\bigcup_{f\in\operatorname*{Hom}_{\mathsf{LCA}%
_{\mathcal{O}_{S}}}(\mathbb{A}_{S},X)}\operatorname*{im}%
\nolimits_{\mathsf{Set}}(f\colon\mathbb{A}_{S}\longrightarrow X)\text{,}
\label{lv1}%
\end{equation}
i.e., the union of the set-theoretic images of all $\mathcal{O}_{S}$-module
homomorphisms from the $S$-ad\`{e}les to $X$, the \emph{adelic core} of $X$.

Note that $\mathbb{A}_{S}$ is, besides being an $\mathcal{O}_{S}$-module, also
a topological ring. Suppose $f(x),g(y)\in\operatorname*{aco}(X)$ and
$\alpha,\beta\in\mathcal{O}_{S}$. By the continuous $\mathcal{O}_{S}$-module
map
\[
\mathbb{A}_{S}\overset{\triangle}{\longrightarrow}\mathbb{A}_{S}%
\oplus\mathbb{A}_{S}\overset{(\cdot x,\cdot y)}{\longrightarrow}\mathbb{A}%
_{S}\oplus\mathbb{A}_{S}\overset{(\alpha f,\beta g)}{\longrightarrow}X\oplus
X\overset{+}{\longrightarrow}X
\]%
\[
s\longmapsto(s,s)\longmapsto(sx,sy)\longmapsto(\alpha f(sx),\beta
g(sy))\longmapsto\alpha f(sx)+\beta g(sy)\text{,}%
\]
any element $\alpha f(x)+\beta g(y)$ can be the image of $1_{\mathbb{A}_{S}}$,
so the set $\operatorname*{aco}(X)$ is an algebraic $\mathcal{O}_{S}%
$-submodule of $X$.

\begin{example}
In $\mathsf{LCA}_{\mathbf{Z}}$ the condition $\operatorname*{aco}(X)=X$ is
equivalent to being injective and the condition $\overline{\operatorname*{aco}%
(X)}=X$ is equivalent to being connected (both follows from $\mathbb{A}%
_{\{\infty\}}=\mathbf{R}$ and \cite[Thm. 3.2 and its Cor. 2]{MR0215016}).
\end{example}

\begin{example}
The importance of running over all $f$ in Eq. \ref{lv1} is seen from the
following fact: Let $C\in\mathsf{LCA}_{\mathbf{Z}}$ be compact. Then a
continuous group homomorphism $f\colon\mathbf{R}\rightarrow C$ with dense
image exists if and only if $C$ is connected \emph{and} admits an open
neighbourhood basis of cardinality $\leq\mathfrak{c}$, the cardinality of the
continuum \cite[Prop. 5.16 (a),(c)]{MR637201}. We expect that the condition
$\overline{\operatorname*{aco}(X)}=X$ scales better with size. See
\S \ref{sect_Appendix_SetTheory} for more on such matters.
\end{example}

\begin{lemma}
Every connected module in $\mathsf{LCA}_{\mathcal{O}_{S}}$ is isomorphic to
$V\oplus K$, where $V$ is a vector $\mathcal{O}_{S}$-module and $K$ a compact
connected $\mathcal{O}_{S}$-module.

\begin{enumerate}
\item Pontryagin duality sends compact connected modules to discrete
$\mathcal{O}_{S}$-flat modules and reversely.

\item If $S$ contains only the infinite places, every injective object in
$\mathsf{LCA}_{\mathcal{O}_{S}}$ is connected.

\item Any compact connected module $K$ such that $K^{\vee}$ is a finitely
generated $\mathcal{O}_{S}$-module is an injective object in $\mathsf{LCA}%
_{\mathcal{O}_{S}}$. Unless $S$ contains all places, there is always some
compact connected $K$ which is not an injective object.

\item If $S$ contains all places, every compact module is connected and injective.

\item If $I\in\mathsf{LCA}_{\mathcal{O}_{S}}$ is an injective object, then
$\operatorname*{aco}(X)=X$.
\end{enumerate}
\end{lemma}

One can directly prove (2) by using (5):\ If $S$ contains only the infinite
places, $\mathbb{A}_{S}$ is a real vector space, hence connected, so
$\operatorname*{aco}(X)$ must lie in the connected component of the identity
in $X$.

\begin{proof}
Suppose $X\in\mathsf{LCA}_{\mathcal{O}_{S}}$ is connected. By Thm.
\ref{thm_MainDecomp} (1) there is an exact sequence $Q\hookrightarrow
X\twoheadrightarrow D$ with $D$ discrete. Since $X$ is connected, so must be
$D$, hence $D=0$. By Thm. \ref{thm_MainDecomp} (2) there is an exact sequence
$K\hookrightarrow X\twoheadrightarrow A$ with $K$ a compact $\mathcal{O}_{S}%
$-module and $A$ adelic. Write $A\simeq V\oplus A^{\prime}$ with $V$ a vector
$\mathcal{O}_{S}$-module and $A^{\prime}$ vector-free adelic. By Thm.
\ref{thm_BraconnierVilenkinTypeTheorem} (2) $A^{\prime}$ is totally
disconnected. Since $X\twoheadrightarrow A\twoheadrightarrow A^{\prime}$ is
surjective, $A^{\prime}$ must also be connected, so $A^{\prime}=0$. We obtain
an exact sequence $K\hookrightarrow X\twoheadrightarrow V$. Since $V$ is a
projective object by Thm. \ref{thm_MainDecomp} (5), this sequence splits. The
splitting allows us to find an admissible epic $X\twoheadrightarrow K$,
proving that $K$ is connected. By \cite[Cor. 2 to Thm. 30]{MR0442141} the dual
$K^{\vee}$ is discrete and has no non-trivial compact subgroups, i.e.,
$K^{\vee}$ is torsion-free. (1) As $\mathcal{O}_{S}$ is a Dedekind domain, it
follows that $K^{\vee}$ is flat over $\mathcal{O}_{S}$. Conversely, the
(discrete) flat modules have no torsion and the cited result can be applied in
reverse to deduce that the Pontryagin dual is compact and connected. (2)
\cite[Prop. 8.1 (1)]{MR4028830} (3) As $K^{\vee}$ is torsion-free, use that
finitely generated flat modules over a Dedekind domain are projective. Then
Thm. \ref{thm_MainDecomp} (4) shows that $K$ is an injective object.
Conversely, assume that every compact connected module $K$ is injective in
$\mathsf{LCA}_{\mathcal{O}_{S}}$. By (1) this means that every algebraic flat
$\mathcal{O}_{S}$-module is projective. Thus, $\mathcal{O}_{S}$ is a
perfect\footnote{This term has no relation to the Frobenius.} ring (in the
sense of non-commutative ring theory/in the sense of Bass) by \cite[Thm
24.25]{MR1125071}. Since $\mathcal{O}_{S}$ is a domain, \cite[Thm
23.24]{MR1125071} implies that $\mathcal{O}_{S}$ is a local ring such that all
elements in its maximal ideal $\mathfrak{m}$ are nilpotent. It follows that
$\mathfrak{m}=(0)$ and $\mathcal{O}_{S}$ must be a field. As $S$ contains all
infinite places by assumption, Eq. \ref{l_Def_RingOS} implies that $S$ must
contain \textit{all} places. (4) If $S$ contains all places, $\mathcal{O}%
_{S}=F$ and thus all $\mathcal{O}_{S}$-modules are flat and the reverse
direction of (1) shows that any compact module is connected. Moreover,
$K^{\vee}$ is always a projective $F$-module, even if it is not finitely
generated. Again, Thm. \ref{thm_MainDecomp} (4) applies. (5) Suppose $y\in I$.
Consider the solid arrows in the diagram%
\[%
\xymatrix{
\mathcal{O}_{S} \ar[dr]_{\cdot y} \ar@{^{(}->}[r] & \mathbb{A}_{S} \ar@
{->>}[r] \ar@{-->}[d] & {\mathbb{A}_{S}}/{\mathcal{O}_{S}} \\
& I,
}%
\]
where the diagonal arrow sends $x\in\mathcal{O}_{S}$ to $xy\in I$ with respect
to its $\mathcal{O}_{S}$-module structure. The dashed arrow exists by the
injectivity of $I$. It directly follows from Eq. \ref{lv1} that $y\in
\operatorname*{aco}(X)$. As $y$ was arbitrary, $\operatorname*{aco}(X)=X$.
\end{proof}

\begin{problem}
What are the objects $X\in\mathsf{LCA}_{\mathcal{O}_{S}}$ such that
$\operatorname*{aco}(X)=X$?
\end{problem}

\subsection{Ind-c.g. Modules\label{subsect_IndCGModules}}

\begin{proposition}
\label{prop_icg_resolvable_in_qa}The inclusion $\mathsf{LCA}_{\mathcal{O}%
_{S},qa}\subseteq\mathsf{LCA}_{\mathcal{O}_{S},icg}$ induces an equivalence of
stable $\infty$-categories $\mathsf{D}_{\infty}^{b}\left(  \mathsf{LCA}%
_{\mathcal{O}_{S},qa}\right)  \overset{\sim}{\rightarrow}\mathsf{D}_{\infty
}^{b}\left(  \mathsf{LCA}_{\mathcal{O}_{S},icg}\right)  $.
\end{proposition}

\begin{proof}
The category $\mathsf{LCA}_{\mathcal{O}_{S},qa}$ is a fully exact subcategory
of $\mathsf{LCA}_{\mathcal{O}_{S}}$ since it is extension-closed (Prop.
\ref{prop_icg_and_qa}). We check the conditions of \cite[Thm. 12.1]{MR1421815}
to obtain that the inclusion induces a triangulated equivalence on the
homotopy categories $Ho(\mathsf{D}_{\infty}^{b}\left(  \mathsf{LCA}%
_{\mathcal{O}_{S},qa}\right)  )\overset{\sim}{\rightarrow}Ho(\mathsf{D}%
_{\infty}^{b}\left(  \mathsf{LCA}_{\mathcal{O}_{S},icg}\right)  )$. Condition
\textit{C1} of loc. cit. amounts to showing that for each $X\in\mathsf{LCA}%
_{\mathcal{O}_{S},icg}$ there is an exact sequence $X\hookrightarrow
N\twoheadrightarrow N^{\prime}$ such that $N\in\mathsf{LCA}_{\mathcal{O}%
_{S},qa}$. Pick $H\subseteq X$ as in Def. \ref{def_indcg}. Then $H\in
\mathsf{LCA}_{\mathcal{O}_{F},cg}$ and%
\begin{equation}
X=\bigcup_{\alpha\in\mathcal{O}_{S}^{\times}}\alpha H=\underset{\alpha
\in\mathcal{O}_{S}^{\times}}{\underrightarrow{\operatorname*{colim}}%
}\,H\text{.} \label{lw1}%
\end{equation}
(i.e., all modules in the colimit are $H$, but the transition morphisms are
all of the shape $H\hookrightarrow\alpha H$). There exists an isomorphism
$H\simeq V\oplus G\oplus C$ with $V$ a vector $\mathcal{O}_{F}$-module, $G$ a
discrete $\mathcal{O}_{F}$-module with underlying group $\mathbf{Z}^{n}$ for
some $n<\infty$, and $C$ a compact $\mathcal{O}_{F}$-module (use \cite[Thm.
6.4 (2)]{MR4028830}, exploiting that $\mathcal{O}_{F}$ is a Dedekind domain).
First, define%
\[
\tilde{G}:=\underset{\alpha\in\mathcal{O}_{S}^{\times}}{\underrightarrow
{\operatorname*{colim}}}\,G\text{,}%
\]
where the transition maps are as in Eq. \ref{lw1}. As each $G$ is discrete,
$\tilde{G}$ is discrete and agrees with $\mathcal{O}_{S}\otimes_{\mathcal{O}%
_{F}}G$ algebraically. We note that $\tilde{G}\in\mathsf{LCA}_{\mathcal{O}%
_{S}}$ is discrete. Note that since $G$ is a finitely generated (torsion-free)
$\mathcal{O}_{F}$-module, $\tilde{G}$ is a finitely generated (torsion-free)
$\mathcal{O}_{S}$-module. Next, $V$ is necessarily already a locally compact
$\mathcal{O}_{S}$-module. To see this, use that all $\alpha\in\mathcal{O}_{S}$
necessarily act through bijectively on $V$ since it is a $F$-vector space.
Moreover, $V$ is $\sigma$-compact and therefore Pontryagin's Open Mapping
Theorem (\cite[p. 23, Theorem 3]{MR0442141}) forces multiplication by any
$\alpha$ to be a homeomorphism. In particular, the colimit $\tilde
{V}:=\underrightarrow{\operatorname*{colim}}_{\alpha\in\mathcal{O}_{S}%
^{\times}}\,G$ is itself $V$ as all transition maps are isomorphisms. We
conclude that $X\cong\tilde{V}\oplus\tilde{G}\oplus\tilde{C}$. As $\tilde
{V},\tilde{C}\in\mathsf{LCA}_{\mathcal{O}_{S},qa}$, it remains to find a
resolution for $\tilde{G}$. As $\tilde{G}$ is a finitely generated
torsion-free $\mathcal{O}_{S}$-module, it is actually projective (as
$\mathcal{O}_{S}$ is a Dedekind domain). Now (topologically)
tensor\footnote{Use Moskowitz's theory of tensor products to equip this with a
topology \cite{MR0215016}. The key point is that $\tilde{G}$ is finitely
generated as an $\mathcal{O}_{S}$-module.} the ad\`{e}le sequence
$\mathcal{O}_{S}\hookrightarrow\mathbb{A}_{S}\twoheadrightarrow\mathbb{A}%
_{S}/\mathcal{O}_{S}$ over $\mathcal{O}_{S}$ with $\tilde{G}$, see Eq.
\ref{l_adele_sequence}. As discussed loc. cit., $\mathbb{A}_{S}$ is adelic and
$\mathbb{A}_{S}/\mathcal{O}_{S}$ is quasi-adelic (since it is a compact
$\mathcal{O}_{S}$-module\footnote{Actually, a more na\"{\i}ve, even if less
conceptual, argument is to note that quotients of adelic objects must always
be quasi-adelic. This is exactly the argument we use to show condition
\textit{C2} in the same proof, so the reader may just copy this argument.}).
Thus, the sequence%
\[
X\hookrightarrow\tilde{V}\oplus(\tilde{G}\otimes_{\mathcal{O}_{S}}%
\mathbb{A}_{S})\oplus\tilde{C}\twoheadrightarrow\tilde{G}\otimes
_{\mathcal{O}_{S}}\mathbb{A}_{S}/\mathcal{O}_{S}%
\]
exhibits all properties we need for Condition \textit{C1}. For Condition
\textit{C2} we need to check that whenever $Q^{\prime}\hookrightarrow
Q\overset{q}{\twoheadrightarrow}X$ in $\mathsf{LCA}_{\mathcal{O}_{S},icg}$
with $Q^{\prime},Q\in\mathsf{LCA}_{\mathcal{O}_{S},qa}$, we must have
$X\in\mathsf{LCA}_{\mathcal{O}_{S},qa}$. This is immediate: Write
$Q=\bigcup_{\alpha\in\mathcal{O}_{S}^{\times}}\alpha C$ with $C$ a compact
$\mathcal{O}_{F}$-module. Then $X=\bigcup_{\alpha\in\mathcal{O}_{S}^{\times}%
}\alpha q(C)$, where $q(C)$ is also open since $q$ is an open map, and it is
compact since $C$ is compact.
\end{proof}

\begin{lemma}
\label{lemma_ad_same_motive_as_qa}The inclusion $\mathsf{LCA}_{\mathcal{O}%
_{S},ad}\subseteq\mathsf{LCA}_{\mathcal{O}_{S},qa}$ induces an equivalence of
localizing non-commutative motives $\mathcal{M}\mathsf{D}_{\infty}^{b}\left(
\mathsf{LCA}_{\mathcal{O}_{S},ad}\right)  \overset{\sim}{\longrightarrow
}\mathcal{M}\mathsf{D}_{\infty}^{b}\left(  \mathsf{LCA}_{\mathcal{O}_{S}%
,qa}\right)  $.
\end{lemma}

\begin{proof}
Let $\mathsf{C}\subseteq\mathsf{LCA}_{\mathcal{O}_{S},qa}$ denote the full
subcategory of compact $\mathcal{O}_{S}$-modules\footnote{As Pontryagin
duality provides an equivalence to $\mathsf{Mod}_{\mathcal{O}_{S}}^{op}$, we
see that $\mathsf{C}$ is an abelian category. This can be used to give
alternative proofs, however, we will not go into this any further.}. It is
easy to see that $\mathsf{C}\subseteq\mathsf{LCA}_{\mathcal{O}_{S},qa}$ is
left $s$-filtering, so there is a Verdier localization sequence $\mathsf{D}%
_{\infty}^{b}\left(  \mathsf{C}\right)  \rightarrow\mathsf{D}_{\infty}%
^{b}\left(  \mathsf{LCA}_{\mathcal{O}_{S},qa}\right)  \rightarrow
\mathsf{D}_{\infty}^{b}\left(  \mathsf{LCA}_{\mathcal{O}_{S},qa}%
/\mathsf{C}\right)  $. By the Eilenberg swindle for countable products,
$\mathcal{M}\mathsf{D}_{\infty}^{b}\left(  \mathsf{C}\right)  =0$ since
products of compact modules are compact (Tychonov's theorem). Finally, the
functor $\mathsf{LCA}_{\mathcal{O}_{S},ad}\rightarrow\mathsf{LCA}%
_{\mathcal{O}_{S},qa}/\mathsf{C}$ is exact, essentially surjective by Theorem
\ref{thm_MainDecomp}, and fully faithful. Hence, an exact equivalence.
\end{proof}

\subsection{The main fiber sequence\label{sect_TheMainFiberSequence}}

For the purposes of this text, we are mainly interested in non-connective
$K$-theory. However, some of the following results are valid for all
localizing invariants \cite[Def. 8.1]{MR3070515}. We nonetheless denote them
by $K$ to focus on our principal application.

\begin{theorem}
\label{thm_ComputeLCAAdelic}Let $F$ be a number field and $S$ be a set of
places of $F$, possibly infinite, which contains all infinite places. Suppose
$\mathsf{A}$ is a stable presentable $\infty$-category. Suppose $K\colon
\operatorname*{Cat}_{\infty}^{\operatorname*{ex}}\rightarrow\mathsf{A}$ is a
localizing invariant which commutes with filtering colimits. If $\#S=\infty$,
we also need to assume that $K$ preserves countable products. Then%
\begin{align}
K(\mathsf{LCA}_{\mathcal{O}_{S},ad})  &  \cong\underset{S^{\prime}\subseteq
S\text{, }\#S^{\prime}<\infty}{\operatorname*{colim}}K\left(
{\textstyle\prod\nolimits_{v\in S\setminus S^{\prime}}}
\mathsf{Proj}_{\mathcal{O}_{v},fg}\times%
{\textstyle\prod\nolimits_{v\in S^{\prime}}}
\mathsf{Mod}_{F_{v},fg}\right) \label{lmx3}\\
&  \cong\left.
{\textstyle\prod\nolimits^{\prime}}
\right.  _{v\in S}K(F_{v}):K(\mathcal{O}_{v})\text{,}\nonumber
\end{align}

\end{theorem}

\begin{proof}
This follows from Prop. \ref{prop_IdentifyAdelicBlocks} and Eq.
\ref{l_RestrictedProductAsColimitJFormula} since an exact equivalence of exact
categories induces an equivalence of its derived stable $\infty$-categories in
$\operatorname*{Cat}_{\infty}^{\operatorname*{ex}}$, and therefore the claim
follows immediately from the assumptions made for the invariant $K$.
\end{proof}

\begin{theorem}
\label{thm_main_weak}Let $F$ be a number field and $S$ be a set of places of
$F$, possibly infinite, which contains all infinite places. Suppose
$\mathsf{A}$ is a stable presentable $\infty$-category. Suppose $K\colon
\operatorname*{Cat}_{\infty}^{\operatorname*{ex}}\rightarrow\mathsf{A}$ is a
localizing invariant (which need not commute with filtering colimits). Then
the sequence%
\[
K(\mathcal{O}_{S})\longrightarrow K\left(  \mathsf{LCA}_{\mathcal{O}_{S}%
,ad}\right)  \longrightarrow K(\mathsf{LCA}_{\mathcal{O}_{S}})
\]
induced by the exact functor $(-)\mapsto(-)\otimes_{\mathcal{O}_{S}}%
\mathbb{A}_{S}$ as the first arrow, is a fiber sequence in $\mathsf{A}$.
\end{theorem}

This generalizes the main fiber sequence of \cite{obloc} (corresponding the
the special case where $S=S_{\infty}$ are just the infinite places) and
\cite{kthyartin} (where $S$ are all places).

We introduce the shorthand%
\begin{equation}
\underset{v\in S\;}{\prod\nolimits^{\prime}}K(F_{v}):K(\mathcal{O}%
_{v}):=\underset{S^{\prime}\subseteq S\text{, }\#S^{\prime}<\infty
}{\operatorname*{colim}}%
{\textstyle\prod\nolimits_{v\in S\setminus S^{\prime}}}
K(\mathcal{O}_{v})\times%
{\textstyle\prod\nolimits_{v\in S^{\prime}}}
K(F_{v})\text{,} \label{lct4a}%
\end{equation}
where the transition maps $K(\mathcal{O}_{v})\rightarrow K(F_{v})$ in the
colimit are induced from the exact functors $(-)\mapsto(-)\otimes
_{\mathcal{O}_{v}}F_{v}$. The following version of the above theorem requires
stronger assumptions, but leads to a more explicit shape of the middle term.

\begin{theorem}
\label{thm_main_strong}Let $F$ be a number field and $S$ be a set of places of
$F$, possibly infinite, which contains all infinite places. Suppose
$\mathsf{A}$ is a stable presentable $\infty$-category. Suppose $K\colon
\operatorname*{Cat}_{\infty}^{\operatorname*{ex}}\rightarrow\mathsf{A}$ is a
localizing invariant which commutes with filtering colimits. If $\#S=\infty$,
we also need to assume that $K$ preserves countable products. Then the
sequence%
\begin{equation}
K(\mathcal{O}_{S})\longrightarrow\left.  \underset{v\in S\;}{\prod
\nolimits^{\prime}}\right.  \left(  K(F_{v}):K(\mathcal{O}_{v})\right)
\longrightarrow K(\mathsf{LCA}_{\mathcal{O}_{S}}) \label{lct4}%
\end{equation}
induced by the exact functor $(-)\mapsto(-)\otimes_{\mathcal{O}_{S}}%
\mathbb{A}_{S}$ as the first arrow, is a fiber sequence in $\mathsf{A}$.
\end{theorem}

\begin{proof}
[Proof of Thm. \ref{thm_main_weak} and Thm. \ref{thm_main_strong}](Step 1) We
obtain the diagram%
\begin{equation}%
\xymatrix{
K(\mathsf{Mod}_{\mathcal{O}_{S},fg}) \ar[r] \ar[d]_{\eta} & K(\mathsf
{Mod}_{\mathcal{O}_{S}}) \ar[r] \ar[d] & K(\mathsf{Mod}_{\mathcal{O}_{S}%
}/\mathsf{Mod}_{\mathcal{O}_{S},fg}) \ar[d]^{\tau} \\
K(\mathsf{LCA}_{\mathcal{O}_{S},icg}) \ar[r] & K(\mathsf{LCA}_{\mathcal{O}%
_{S}}) \ar[r] & K(\mathsf{LCA}_{\mathcal{O}_{S}}/\mathsf{LCA}_{\mathcal{O}%
_{S},icg})
}
\label{lxga1}%
\end{equation}
as follows: (1) The inclusions $\mathsf{LCA}_{\mathcal{O}_{S},icg}%
\subseteq\mathsf{LCA}_{\mathcal{O}_{S}}$ (resp. $\mathsf{Mod}_{\mathcal{O}%
_{S},fg}\subseteq\mathsf{Mod}_{\mathcal{O}_{S}}$) induce Verdier localization
sequences for the respective derived $\infty$-categories because they are left
$s$-filtering by Prop. \ref{prop_icg_and_qa} (resp. a Serre subcategory in an
abelian category). See \cite[\S 8.3]{hr,hr2}\footnote{For the top row the
Localization Theorem of Schlichting in \cite{MR2079996} also shows this. For
the bottom row the original Quillen\ Localization Theorem \cite{MR0338129} is
enough. However, in order to get the commutativity of the diagram, it is
better to use the same Localization Theorem (for otherwise one would have to
prove that they give compatible fiber sequences). The big advantage of the
Henrard--van Roosmalen Localization Theorem is that it is the most general
known localization theorem for (even just one-sided) exact categories.}. (2)
The downward arrows just refer to regarding the top row modules as equipped
with the discrete topology. For a finitely generated module, picking $K:=\{\pm
b_{1},\ldots,\pm b_{r}\}$ for $b_{1},\ldots,b_{r}$ a set of generators, shows
that it is ind-c.g. (Step 2) It is easy to check that the functor
$\mathsf{Mod}_{\mathcal{O}_{S}}/\mathsf{Mod}_{\mathcal{O}_{S},fg}%
\rightarrow\mathsf{LCA}_{\mathcal{O}_{S}}/\mathsf{LCA}_{\mathcal{O}_{S},icg}$
is an exact equivalence of exact categories:\ It is exact. By Thm.
\ref{thm_MainDecomp} any $X\in\mathsf{LCA}_{\mathcal{O}_{S}}$ has an exact
sequence $Q\hookrightarrow X\twoheadrightarrow D$ with $Q$ quasi-adelic (hence
ind-c.g. by Prop. \ref{prop_icg_and_qa}, (1)) and $D$ discrete. Thus, it is
essentially surjective. Next, note that a discrete $\mathcal{O}_{S}$-module is
ind-c.g. if and only if it is finitely generated\footnote{Use that a discrete
LCA group is c.g. if and only if it is isomorphic to $\mathbf{Z}^{n}$ for some
$n$ \cite[Thm. 2.5]{MR0215016}. This forces $H$ in Def. \ref{def_indcg} to be
a finitely generated $\mathcal{O}_{F}$-module.}. Thus, the arrow $\tau$ is an
equivalence in $\mathsf{A}$. It follows that the left square in Diagram
\ref{lxga1} is bi-Cartesian. However, $K(\mathsf{Mod}_{\mathcal{O}_{S}})=0$ by
the Eilenberg swindle as $\mathsf{Mod}_{\mathcal{O}_{S}}$ posesses countable
coproducts. Moreover, the arrow $K(\mathcal{O}_{S})\rightarrow K(\mathsf{Mod}%
_{\mathcal{O}_{S},fg})$ induced from the inclusion of the full subcategory of
projective objects is an equivalence since $\mathcal{O}_{S}$ is regular and
therefore $\mathsf{D}_{\infty}^{b}(\mathsf{Proj}_{\mathcal{O}_{S}%
,fg})\rightarrow\mathsf{D}_{\infty}^{b}(\mathsf{Mod}_{\mathcal{O}_{S},fg})$
induces an equivalence. We arrive at the fiber sequence $K(\mathcal{O}%
_{S})\rightarrow K(\mathsf{LCA}_{\mathcal{O}_{S},icg})\rightarrow
K(\mathsf{LCA}_{\mathcal{O}_{S}})$. (Step 3) Following this construction, the
arrow $\eta$ is induced from the functor sending a finitely generated
projective $\mathcal{O}_{S}$-module $X$ to itself, with the discrete topology.
Write $\mathcal{E(-)}$ for the exact category of exact sequences
(\cite[Exercise 3.9]{MR2606234}). We claim that the arrow $\eta$ is homotopic
to the one induced from the functor $(-)\mapsto(-)\otimes_{\mathcal{O}_{S}%
}\mathbb{A}_{S}$. To see this, we note that the functor%
\[
\mathsf{Proj}_{\mathcal{O}_{S},fg}\longrightarrow\mathcal{E}\mathsf{LCA}%
_{\mathcal{O}_{S},icg}\text{,}\qquad X\mapsto\left(  X\hookrightarrow
X\otimes_{\mathcal{O}_{S}}\mathbb{A}_{S}\twoheadrightarrow X\otimes
_{\mathcal{O}_{S}}\mathbb{A}_{S}/X\right)
\]
is exact (where $X\otimes_{\mathcal{O}_{S}}\mathbb{A}_{S}$ is given the
topology of the ad\`{e}les. It is then adelic, whereas the quotient
$X\otimes_{\mathcal{O}_{S}}\mathbb{A}_{S}/X$ by the discrete $X$ is compact).
Hence, as $K$ is also an additive invariant, it follows that if we write
$f_{i\ast}$ ($i=1,2,3$) for the induced projections for the three entries
\textit{left, middle, right} of the exact sequences that $f_{2\ast}=f_{1\ast
}+f_{3\ast}\colon K(\mathcal{O}_{S})\rightarrow K(\mathsf{LCA}_{\mathcal{O}%
_{S},icg})$. Now note that $f_{3\ast}$ factors as $K(\mathcal{O}%
_{S})\rightarrow K(\mathsf{LCA}_{\mathcal{O}_{S},\operatorname*{compact}%
})\rightarrow K(\mathsf{LCA}_{\mathcal{O}_{S},icg})$, noting that the compact
modules, referred to by $\mathsf{LCA}_{\mathcal{O}_{S},\operatorname*{compact}%
}$ here, are a full subcategory of the ind-c.g. modules. As arbitrary products
of compact modules are again compact by Tychonov's theorem, $K(\mathsf{LCA}%
_{\mathcal{O}_{S},\operatorname*{compact}})=0$ by the Eilenberg swindle.
Therefore, $f_{3\ast}=0$. It follows that $f_{1\ast}=f_{2\ast}$.\footnote{Let
us stress that there is no corresponding \textquotedblleft
dual\textquotedblright\ factorization $K(\mathcal{O}_{S})\rightarrow
K(\mathsf{LCA}_{\mathcal{O}_{S},\operatorname*{discrete}})\overset
{a}{\rightarrow}K(\mathsf{LCA}_{\mathcal{O}_{S},icg})$ since discrete modules
are ind-c.g. only if they are finitely generated. Thus, such a putative arrow
$a$ does not exist and thus, there is no Eilenberg swindle here which could
show that $f_{1\ast}=0$.} (Step 4) By Prop. \ref{prop_icg_resolvable_in_qa}
and Lemma \ref{lemma_ad_same_motive_as_qa} we get $K(\mathsf{LCA}%
_{\mathcal{O}_{S},icg})\cong K(\mathsf{LCA}_{\mathcal{O}_{S},qa})\cong
K(\mathsf{LCA}_{\mathcal{O}_{S},ad})$. Now, using Prop.
\ref{prop_IdentifyAdelicBlocks} there is an exact equivalence of categories
$\mathsf{LCA}_{\mathcal{O}_{S},ad}\cong\left.  \prod_{v\in S}^{\prime}\right.
\mathsf{Mod}_{F_{v},fg}$. This proves Thm. \ref{thm_main_weak} and by
combining it with Thm. \ref{thm_ComputeLCAAdelic}, we also get Thm.
\ref{thm_main_strong}.
\end{proof}

\begin{example}
In \cite{kthyartin} it was proven that $K_{1}(\mathsf{LCA}_{F})$ agrees with
the id\`{e}le class group. In the present setting, this corresponds to letting
$S$ be the set of all places. The above computations show that this
observation does \emph{not} generalize to smaller $S$. From Thm.
\ref{thm_main_strong} we get the long exact sequence%
\[
K_{1}(\mathcal{O}_{S})\rightarrow\left.  \underset{v\in S\;}{%
{\textstyle\prod\nolimits^{\prime}}
}\left(  F_{v}^{\times}:\mathcal{O}_{v}^{\times}\right)  \right.  \rightarrow
K_{1}(\mathsf{LCA}_{\mathcal{O}_{S}})\rightarrow K_{0}(\mathcal{O}%
_{S})\overset{\beta}{\rightarrow}\left.  \underset{v\in S}{%
{\textstyle\prod}
}\right.  \mathbf{Z}\text{.}%
\]
Since $\mathcal{O}_{S}$ is a Dedekind domain\footnote{In this text we use the
convention that the Krull dimension of Dedekind domains has to be $\leq1$, so
that fields are included in the concept.} $K_{0}(\mathcal{O}_{S}%
)\cong\mathbf{Z}\oplus\operatorname*{Cl}(\mathcal{O}_{S})$ and it is easy to
see that $\beta$ is injective on the rank summand\footnote{under the running
assumptions of this section, $S$ contains all infinite places, so $S$ is
non-empty.} and sends $\operatorname*{Cl}(\mathcal{O}_{S})$ to zero
(necessarily so because the class group is torsion, but the target of the map
is torsion-free). By Bass--Milnor--Serre \cite{MR244257} and localization, we
have $K_{1}(\mathcal{O}_{S})\cong\mathcal{O}_{S}^{\times}$. Thus, we get an
exact sequence%
\[
C_{F,S}\hookrightarrow K_{1}(\mathsf{LCA}_{\mathcal{O}_{S}})\twoheadrightarrow
\operatorname*{Cl}(\mathcal{O}_{S})\text{,}%
\]
where $C_{F,S}$ is the $S$-id\`{e}le class group of $F$. Thus, $K_{1}%
(\mathsf{LCA}_{\mathcal{O}_{S}})$ agrees with the id\`{e}le class group as
soon as the ideal class group of $\mathcal{O}_{S}$ is trivial.
\end{example}

\subsection{Ring structures\label{subsect_RingStructures}}

In this section $K$ stands for non-connective $K$-theory. But more generally,
any lax monoidal localizing invariant between symmetric monoidal $\infty
$-categories would admit the following constructions as well.

The bi-exact symmetric monoidal structure%
\[
\mathsf{Proj}_{\mathcal{O}_{S},fg}\times\mathsf{Proj}_{\mathcal{O}_{S}%
,fg}\longrightarrow\mathsf{Proj}_{\mathcal{O}_{S},fg}\text{,}%
\]
sending $(X,Y)\mapsto X\otimes_{\mathcal{O}_{S}}Y$, induces an $E_{\infty}%
$-ring structure on $K$-theory%
\[
K(\mathcal{O}_{S})\otimes_{\mathsf{Sp}}K(\mathcal{O}_{S})\longrightarrow
K(\mathcal{O}_{S})\text{.}%
\]
There is a similar bi-exact pairing%
\begin{equation}
\mathsf{LCA}_{\mathcal{O}_{S}}\times\mathsf{Proj}_{\mathcal{O}_{S}%
,fg}\longrightarrow\mathsf{LCA}_{\mathcal{O}_{S}}\text{.} \label{lmx4}%
\end{equation}
It can be constructed as follows: Let $\mathsf{Free}_{\mathcal{O}_{S},fg}$
denote the exact category of finitely generated free $\mathcal{O}_{S}%
$-modules. It is additively generated by $\mathcal{O}_{X}$. It follows that
there is a unique pairing%
\[
\mathsf{LCA}_{\mathcal{O}_{S}}\times\mathsf{Free}_{\mathcal{O}_{S}%
,fg}\longrightarrow\mathsf{LCA}_{\mathcal{O}_{S}}%
\]
such that $(-)\otimes\mathcal{O}_{S}\colon X\otimes\mathcal{O}_{S}:=X$ is the
identity functor. This extends uniquely to the idempotent completion
$(-)^{\operatorname*{ic}}$. As $\mathsf{LCA}_{\mathcal{O}_{S}}$ is already
idempotent complete (it has all kernels), and $\mathsf{Free}_{\mathcal{O}%
_{S},fg}^{\operatorname*{ic}}\cong\mathsf{Proj}_{\mathcal{O}_{S},fg}$, we
arrive at Eq. \ref{lmx4}. We observe that adelic modules get sent to adelic
modules:%
\[
\mathsf{LCA}_{\mathcal{O}_{S},ad}\times\mathsf{Proj}_{\mathcal{O}_{S}%
,fg}\longrightarrow\mathsf{LCA}_{\mathcal{O}_{S},ad}\text{.}%
\]
In Rmk. \ref{rmk_SymmetricMonoidalStructureOnAdelicModules} we have set up a
bi-exact symmetric monoidal structure%
\[
\otimes\colon\mathsf{LCA}_{\mathcal{O}_{S},ad}\times\mathsf{LCA}%
_{\mathcal{O}_{S},ad}\longrightarrow\mathsf{LCA}_{\mathcal{O}_{S},ad}%
\]
on $\mathsf{LCA}_{\mathcal{O}_{S},ad}$. This renders the $K$-theory
$K(\mathsf{LCA}_{\mathcal{O}_{S},ad})$ an $E_{\infty}$-ring spectrum,
suggestively written as%
\begin{equation}
K(\mathsf{LCA}_{\mathcal{O}_{S},ad})\otimes_{\mathsf{Sp}}K(\mathsf{LCA}%
_{\mathcal{O}_{S},ad})\longrightarrow K(\mathsf{LCA}_{\mathcal{O}_{S}%
,ad})\text{.} \label{lmx2}%
\end{equation}
We can summarize this discussion as follows.

\begin{proposition}
\label{prop_RingStructuresOnSpectra}Suppose $K$ refers to non-connective
$K$-theory. Using the above setup of \S \ref{subsect_RingStructures} as
constructions, $K(\mathcal{O}_{S})$ is an $E_{\infty}$-ring spectrum,
$K(\mathsf{LCA}_{\mathcal{O}_{S},ad})$ is an $E_{\infty}$-ring spectrum,
$K(\mathsf{LCA}_{\mathcal{O}_{S}})$ is a $K(\mathcal{O}_{S})$-module, and
$K(\mathcal{O}_{S})\rightarrow K(\mathsf{LCA}_{\mathcal{O}_{S},ad})$, induced
from $X\mapsto\mathbb{A}_{S}\otimes_{\mathcal{O}_{S}}X$, makes $K(\mathsf{LCA}%
_{\mathcal{O}_{S},ad})$ a $K(\mathcal{O}_{S})$-algebra.
\end{proposition}

\begin{proposition}
\label{prop_MapOfEInftyRingSpectra}If we take the localizing invariant
$K\colon\operatorname*{Cat}\nolimits_{\infty}^{\operatorname*{ex}}%
\rightarrow\mathsf{Sp}$ to have values in spectra, commute with filtering
colimits and be lax symmetric monoidal once the input is a symmetric monoidal
stable $\infty$-category, then the above results can be augmented by the following:

\begin{enumerate}
\item The equivalence of spectra in Eq. \ref{lmx3} can be promoted to an
equivalence of $E_{\infty}$-ring spectra.

\item The first arrow in Eq. \ref{lct4} is a map of $E_{\infty}$-ring spectra.
\end{enumerate}
\end{proposition}

\begin{proof}
(1) The equivalence is induced by the exact functor $\Psi$ provided by Prop.
\ref{prop_IdentifyAdelicBlocks}. By Rmk.
\ref{rmk_SymmetricMonoidalStructureOnAdelicModules} the monoidal structure on
$\mathsf{LCA}_{\mathcal{O}_{S},ad}$ is constructed so that $\Psi$ is a
symmetric monoidal equivalence. Thus, by functoriality, the equivalence is an
equivalence of ring spectra. (2) This is true because the underlying exact
functor is lax symmetric monoidal.
\end{proof}

\section{\label{sect_LocalDuality}Restricted product of descent spectral
sequences}

In this section $K$ stands for non-connective $K$-theory. Let $F$ be a number
field. Let $S$ be a (possibly infinite) set of places of $F$ containing all
the infinite places and such that $\frac{1}{p}\in\mathcal{O}_{S}$. Recall from
Prop. \ref{prop_IdentifyAdelicBlocks} that%
\begin{equation}
K(\mathsf{LCA}_{\mathcal{O}_{S},ad})\cong\underset{S^{\prime}\subseteq
S\text{, }\#S^{\prime}<\infty}{\operatorname*{colim}}%
{\textstyle\prod\nolimits_{v\in S\setminus S^{\prime}}}
K(\mathcal{O}_{v})\times%
{\textstyle\prod\nolimits_{v\in S^{\prime}}}
K\left(  F_{v}\right)  \text{.} \label{l_m0}%
\end{equation}
In this section we want to prove a $K(1)$-local analogue of this result, i.e.,
we would like to replace each letter \textquotedblleft$K$\textquotedblright%
\ by \textquotedblleft$L_{K(1)}K$\textquotedblright.

We write $\mathsf{Sp}$ for the ordinary stable homotopy category and
$\mathsf{K}$ for the $K(1)$-local homotopy category. We shall use an
additional superscript to clarify where we take the (co)limit. For example,
$\left.  \operatorname*{colim}\nolimits^{\mathsf{K}}\right.  $ or $\left.
\bigoplus\nolimits^{\mathsf{K}}\right.  $ refer to the corresponding colimits
taken in $\mathsf{K}$, and then represented as a $K(1)$-local object in
$\mathsf{Sp}$. That is: We never work truly intrinsically to the category
$\mathsf{K}$, we always just consider its local object representatives in
$\mathsf{Sp}$.

This being said, let us return to finding a $K(1)$-local analogue of Eq.
\ref{l_m0}. First, observe that the universal properties of limits and
colimits provide natural maps%
\begin{align}
L_{K(1)}%
{\textstyle\prod\nolimits^{\mathsf{Sp}}}
\left(  \ldots\right)   &  \longrightarrow%
{\textstyle\prod\nolimits^{\mathsf{K}}}
L_{K(1)}(\ldots)\label{l_m1}\\
\left.  \operatorname*{colim}\nolimits^{\mathsf{K}}\right.  L_{K(1)}(\ldots)
&  \longrightarrow L_{K(1)}\left(  \left.  \operatorname*{colim}%
\nolimits^{\mathsf{Sp}}\right.  (\ldots)\right)  \text{.} \label{l_m2}%
\end{align}
A priori, neither of them has a reason to be an equivalence.

\begin{theorem}
\label{thm_compar1}Let $F$ be a number field and $S$ a (possibly infinite) set
of places of $F$ containing the infinite places and such that $\frac{1}{p}%
\in\mathcal{O}_{S}$. Then, starting from the left side of Eq. \ref{l_m0}, the
zig-zag formed from the morphisms in Eqs. \ref{l_m1} and \ref{l_m2}, induces
an equivalence of spectra%
\begin{equation}
L_{K(1)}K(\mathsf{LCA}_{\mathcal{O}_{S},ad})\cong\left.  \underset{S^{\prime
}\subseteq S\text{, }\#S^{\prime}<\infty}{\operatorname*{colim}%
\nolimits^{\mathsf{K}}}\right.
{\textstyle\prod\nolimits_{v\in S\setminus S^{\prime}}^{\mathsf{K}}}
L_{K(1)}K(\mathcal{O}_{v})\times%
{\textstyle\prod\nolimits_{v\in S^{\prime}}^{\mathsf{K}}}
L_{K(1)}K\left(  F_{v}\right)  \text{.} \label{l_m4}%
\end{equation}
Here $K$ denotes non-connective $K$-theory.
\end{theorem}

For the products, the superscript $%
{\textstyle\prod\nolimits^{\mathsf{K}}}
$ is not needed, because products are the same in $\mathsf{Sp}$ and
$\mathsf{K}$, but we added them to stress the perfect analogy to the
formulation for $\mathsf{Sp}$ in Eq. \ref{l_m0}.

\begin{proof}
[Proof of Thm. \ref{thm_compar1} for $S$ finite]If $S$ is finite, the colimit
in Eq. \ref{l_m0} runs over a finite directed system, and since $L_{K(1)}$
commutes with finite colimits, we may pick $S^{\prime}=S$ and get%
\[
L_{K(1)}K(\mathsf{LCA}_{\mathcal{O}_{S},ad})\cong L_{K(1)}\left(
{\textstyle\prod\nolimits_{v\in S^{\prime}}}
K\left(  F_{v}\right)  \right)  \cong%
{\textstyle\prod\nolimits_{v\in S^{\prime}}}
L_{K(1)}\left(  K\left(  F_{v}\right)  \right)
\]
since $L_{K(1)}$ commutes with finite products.
\end{proof}

The rest of the section is devoted to proving Thm. \ref{thm_compar1} in the
case of possibly infinite $S$.

\subsection{Individual descent spectral
sequences\label{sect_IndividualThomasonDescentSpectralSequences}}

We first recall the shape of the individual spectral sequences which shall
enter our restricted product construction. If only to set up notation, let us
quickly go through this, but only in the special case of those rings we will
later actually care about. Suppose $\frac{1}{p}\in\mathcal{O}_{S}$. Then each
of the rings $R=\mathcal{O}_{S},\mathcal{O}_{v},F_{v}$ contains $\frac{1}{p}$
and has \'{e}tale $p$-cohomological dimension $\leq2$. Indeed, precisely $2$
in all cases except for $\mathcal{O}_{v}=F_{v}$ for real and complex places.
Hence, for any $k\geq1$ the\ Thomason descent spectral sequence is convergent
and has the $E_{2}$-page%
\begin{equation}
E_{2}^{i,j}:=H^{i}(R,\mathbf{Z}/p^{k}(-\tfrac{j}{2}))\Rightarrow\pi
_{-i-j}L_{K(1)}K(R)\otimes_{\mathsf{Sp}}\mathbb{S}/p^{k}\text{,}
\label{lDescentSpectralSequence}%
\end{equation}
but at worst three columns may have non-zero terms. As a result, the $E_{2}%
$-page looks like%
\begin{equation}%
\begin{tabular}
[c]{rccccc}%
$\vdots\,\,$ & \multicolumn{1}{|c}{} & $\vdots$ & $\vdots$ & $\vdots$ & \\
$\mathsf{2}$ & \multicolumn{1}{|c}{$\quad0\quad$} & $H^{0}(R,\mathbf{Z}%
/p^{k}(-1))$ & $H^{1}(R,\mathbf{Z}/p^{k}(-1))$ & $H^{2}(R,\mathbf{Z}%
/p^{k}(-1))$ & $\quad0\quad$\\
$\mathsf{1}$ & \multicolumn{1}{|c}{$\quad0\quad$} & $0$ & $0$ & $0$ &
$\quad0\quad$\\
$\mathsf{0}$ & \multicolumn{1}{|c}{$\quad0\quad$} & $H^{0}(R,\mathbf{Z}%
/p^{k}(0))%
{\cellcolor{lg}}%
$ & $H^{1}(R,\mathbf{Z}/p^{k}(0))$ & $H^{2}(R,\mathbf{Z}/p^{k}(0))$ &
$\quad0\quad$\\
$\mathsf{-1}$ & \multicolumn{1}{|c}{$\quad0\quad$} & $0$ & $0%
{\cellcolor{lg}}%
$ & $0$ & $\quad0\quad$\\
$\mathsf{-2}$ & \multicolumn{1}{|c}{$\quad0\quad$} & $H^{0}(R,\mathbf{Z}%
/p^{k}(1))$ & $H^{1}(R,\mathbf{Z}/p^{k}(1))$ & $H^{2}(R,\mathbf{Z}/p^{k}(1))%
{\cellcolor{lg}}%
$ & $\quad0\quad$\\
$\vdots\,\,$ & \multicolumn{1}{|c}{} & $\vdots$ & $\vdots$ & $\vdots$ &
\\\cline{2-6}
& $\mathsf{-1}$ & $\mathsf{0}$ & $\mathsf{1}$ & $\mathsf{2}$ & $\mathsf{\cdots
}$%
\end{tabular}
\ \ \ \ \label{lwui1}%
\end{equation}
with differential $d_{2}\colon E_{2}^{i,j}\rightarrow E_{2}^{i+2,j-1}$ of
cohomological type. In any non-zero row, the differential necessarily points
to a zero entry and consequently the spectral sequence must collapse already
on the $E_{2}$-page. Thus, we obtain canonical isomorphisms $\pi
_{2j-1}L_{K(1)}K(R)\otimes\mathbb{S}/p^{k}\cong H^{1}(R,\mathbf{Z}%
/p^{k}\mathbf{Z}(j))$ and exact sequences%
\begin{equation}
H^{2}(R,\mathbf{Z}/p^{k}\mathbf{Z}(j+1))\hookrightarrow\pi_{2j}L_{K(1)}%
K(R)\otimes_{\mathsf{Sp}}\mathbb{S}/p^{k}\twoheadrightarrow H^{0}%
(R,\mathbf{Z}/p^{k}\mathbf{Z}(j)) \label{lct5b}%
\end{equation}
for all $j\in\mathbf{Z}$, respectively. The shaded regions highlight the
$\pi_{0}$-diagonal.

\subsection{$K(1)$ on the inside}

We first compute the right side in Eq. \ref{l_m4}. This is a computation
mostly intrinsic to the $K(1)$-local homotopy category.

\subsubsection{The $P$-groups}

\begin{definition}
\label{def_PGroups}Let $p$ be an odd prime. Let $F$ be a number field and $S$
a (possibly infinite) set of places of $F$ containing the infinite places and
such that $\frac{1}{p}\in\mathcal{O}_{S}$. Write $G_{S}:=\operatorname*{Gal}%
(F_{S}/F)$, where $F_{S}$ is the maximal Galois extension of $F$ which is
unramified outside $S$. If $M$ is a finite $p$-primary $G_{S}$-module, define%
\[
P_{S}^{i}(F,M):=\underset{v\in S\setminus S_{\infty}}{\left.  \prod
\nolimits^{\prime}\right.  }H^{i}(F_{v},M):H_{\operatorname*{ur}}^{i}%
(F_{v},M)
\]
as a restricted product of abelian groups. Note that $v$ runs only through the
finite places.\footnote{From the viewpoint of this text, it must appear rather
ad hoc and unnatural to skip the infinite places in this definition. We chose
to do this so that the definition is compatible with the literature. There
\textit{are} reasons to define the $P$-groups this way. The issue is that the
local duality at infinite places must be formulated with Tate cohomology
groups, and if one uses the usual Galois cohomology groups, this duality
collapses at all real and complex places. For more on this, see Pitfall
\ref{pitfall_l1}.}
\end{definition}

\begin{remark}
\label{rmk_TheStoryAboutTateCohomology1}This definition agrees with \cite[Ch.
17, Def. 17.5]{MR4174395}. Harari's $v$ runs through all places in $S$, but he
uses the convention (\cite[Def. 17.4]{MR4174395}) that for infinite places
$H^{0}$ tacitly is meant to refer to Tate cohomology $\widehat{H}^{0}$
instead. But this is just zero in our situation:\ For complex places,
$\widehat{H}^{0}(F_{v},M)$ is always zero and for real places, $2\widehat
{H}(F_{v},M)=0$ since $\operatorname*{Gal}(\mathbf{C}/\mathbf{R})$ has order
$2$. As $M$ has odd order, this group also vanishes. The definition is also
compatible with Milne \cite[Ch. I, \S 4, before Lemma 4.8]{MR2261462} or
\cite[Ch. VIII, \S 5]{MR2392026}.
\end{remark}

\subsubsection{Derived completion}

We first recall some foundations necessary for computing $K(1)$-local
colimits. Let $p$ be a prime. Write $\mathsf{Ab}$ for the abelian category of
abelian groups. A group $A\in\mathsf{Ab}$ is called $p$\emph{-complete} if the
natural map $A\rightarrow\widehat{A}:=\lim_{n\rightarrow\infty}A/p^{n}A$ is an
isomorphism. We write $\widehat{\mathsf{Ab}}_{p}$ for the category of derived
$p$-complete abelian groups\footnote{Other names are in active use.
Bousfield--Kan \cite{MR0365573}, Mitchell, Hahn \cite{MR2327028}, among
others, refer to them as $\operatorname*{Ext}$-$p$\emph{-complete groups}.
Barthel, Frankland, Hovey, Strickland, e.g. \cite[Appendix A]{MR1601906},
\cite[Appendix A]{MR3402336} refer to $L$\emph{-complete} groups. We favour
\textquotedblleft derived complete\textquotedblright\ because of its
descriptive nature.}, i.e., the smallest abelian full subcategory of
$\mathsf{Ab}$ which contains all $p$-complete groups. A great reference is
\cite[Appendix A]{MR1601906}. We only need a few properties: Define%
\[
L_{i}(A):=\operatorname*{Ext}\nolimits_{\mathsf{Ab}}^{1-i}(\mathbf{Q}%
_{p}/\mathbf{Z}_{p},A)\text{.}%
\]
These are functors $L_{i}\colon\mathsf{Ab}\rightarrow\widehat{\mathsf{Ab}}%
_{p}\subset\mathsf{Ab}$, and $L_{0}$ is right-exact and $L_{1}$ is its single
higher left-derived functor. That is, for any exact sequence $A^{\prime
}\hookrightarrow A\twoheadrightarrow A^{\prime\prime}$, the sequence%
\begin{equation}
0\rightarrow L_{1}A^{\prime}\rightarrow L_{1}A\rightarrow L_{1}A^{\prime
\prime}\rightarrow L_{0}A^{\prime}\rightarrow L_{0}A\rightarrow L_{0}%
A^{\prime\prime}\rightarrow0 \label{l_ExactSequenceOfLFunctors}%
\end{equation}
is exact. Recall that $\mathbf{Z}$ has global dimension one (a hereditary
ring), so $\operatorname*{Ext}\nolimits_{\mathsf{Ab}}^{i}=\operatorname*{Ext}%
\nolimits_{\mathsf{Mod}_{\mathbf{Z}}}^{i}$ is zero for $i\neq0,1$. The
principal tool is the following easy fact.

\begin{lemma}
\label{lemma_ComputeL0ViaClassicalCompletion}(\cite[Thm. A.2 (a)]{MR1601906})
For any abelian group $A$,%
\[
A_{\operatorname*{div}}\hookrightarrow L_{0}A\twoheadrightarrow\widehat{A}%
\]
is exact, where $A_{\operatorname*{div}}$ denotes the subgroup of
$p$-divisible elements and $\widehat{A}:=\lim_{n\rightarrow\infty}A/p^{n}A$ is
the classical $p$-completion.
\end{lemma}

\subsubsection{The main computation}

\begin{lemma}
\label{lemma_DerivedCompletionOfPGroups}Let $F$ be a number field and suppose
$S_{\infty}\subseteq S$. Let $p$ be an odd prime. Suppose $M$ is a finite
$p$-primary Galois module. Then

\begin{enumerate}
\item $L_{0}P_{S}^{i}(F,M)\cong P_{S}^{i}(F,M)$ for all $i\in\mathbf{Z}$.

\item $L_{1}P_{S}^{i}(F,M)=0$ for all $i\in\mathbf{Z}$.
\end{enumerate}
\end{lemma}

\begin{proof}
For each $i$ there exists an inclusion%
\begin{equation}
P_{S}^{i}(F,M)\subseteq\prod_{v\in S}H^{i}(F_{v},M)\text{.} \label{lhupp2}%
\end{equation}
More precisely, for $i=0$ this is indeed an equality in all cases. For $i=1,2$
it is an equality if $S$ is finite, and typically an inclusion of a strictly
smaller subgroup when $S$ is infinite. The product comes with projectors to
each $H^{i}(F_{v},M)$ and the latter groups are always finite (by local class
field theory). As the image of a $p$-divisible element must again be
$p$-divisible and finite groups have no non-zero $p$-divisible elements, we
conclude that $P_{S}^{i}(F,M)_{\operatorname*{div}}=0$. Thus,%
\[
L_{1}P_{S}^{i}(F,M)=\operatorname*{Hom}\nolimits_{\mathsf{Ab}}(\mathbf{Q}%
_{p}/\mathbf{Z}_{p},P_{S}^{i}(F,M))=0
\]
since every element in $\mathbf{Q}_{p}/\mathbf{Z}_{p}$ is $p$-divisible. This
proves (2). For (1), Lemma \ref{lemma_ComputeL0ViaClassicalCompletion} shows
that we only need to compute the classical $p$-completion. Since $M$ is
finite, $p^{k}\cdot M=0$ for $k$ sufficiently large. Thus, $p^{k}\cdot
P_{S}^{i}(F,M)=0$ and therefore the inverse limit is eventually the identity.
\end{proof}

\begin{lemma}
\label{lemma_specseq_for_restricted_product_inside}Let $F$ be a number field
and $S$ a (possibly infinite) set of places of $F$ containing the infinite
places and such that $\frac{1}{p}\in\mathcal{O}_{S}$. Suppose $k\geq1$. Define%
\[
\mathfrak{X}_{\infty}:=\left.  \underset{S^{\prime}\subseteq S\text{,
}\#S^{\prime}<\infty}{\operatorname*{colim}\nolimits^{\mathsf{K}}}\right.
\left(
{\textstyle\prod\nolimits_{v\in S\setminus S^{\prime}}}
L_{K(1)}(K/p^{k})(\mathcal{O}_{v})\times%
{\textstyle\prod\nolimits_{v\in S^{\prime}}}
L_{K(1)}(K/p^{k})\left(  F_{v}\right)  \right)
\]
This is the right side of Eq. \ref{l_m4}, but tensored with the Moore spectrum
$\mathbb{S}/p^{k}$. Then the restricted product of the Thomason descent
spectral sequences of Eq. \ref{lDescentSpectralSequence} (for $R=\mathcal{O}%
_{v}$ resp. $F_{v}$) yields a convergent spectral sequence%
\[
E_{2}^{i,j}:=\left.  \underset{v\in S\;}{%
{\displaystyle\prod\nolimits^{\prime}}
}\right.  H^{i}(F_{v},\mathbf{Z}/p^{k}\mathbf{Z}(-\tfrac{j}{2})):H^{i}%
(\mathcal{O}_{v},\mathbf{Z}/p^{k}\mathbf{Z}(-\tfrac{j}{2}))\Rightarrow
\pi_{-i-j}\mathfrak{X}_{\infty}\text{,}%
\]
and (just as for each Eq. \ref{lDescentSpectralSequence} individually) it
collapses to two columns. Notably,

\begin{enumerate}
\item $\pi_{2j-1}\mathfrak{X}_{\infty}\cong P_{S}^{1}(F,\mathbf{Z}%
/p^{k}\mathbf{Z}(j))$, and

\item $P_{S}^{2}(F,\mathbf{Z}/p^{k}\mathbf{Z}(j+1))\hookrightarrow\pi
_{2j}\mathfrak{X}_{\infty}\twoheadrightarrow P_{S}^{0}(F,\mathbf{Z}%
/p^{k}\mathbf{Z}(j))\oplus\bigoplus_{v\in S}H^{0}(F_{v},\mathbf{Z}%
/p^{k}\mathbf{Z}(j))$,

\item and the natural morphism $(-)\rightarrow L_{K(1)}(-)$ induces an
equivalence from
\[
\left.  \underset{S^{\prime}\subseteq S\text{, }\#S^{\prime}<\infty
}{\operatorname*{colim}\nolimits^{\mathsf{Sp}}}\right.  \left(
{\textstyle\prod\nolimits_{v\in S\setminus S^{\prime}}}
L_{K(1)}(K/p^{k})(\mathcal{O}_{v})\times%
{\textstyle\prod\nolimits_{v\in S^{\prime}}}
L_{K(1)}(K/p^{k})\left(  F_{v}\right)  \right)
\]
to $\mathfrak{X}_{\infty}$. In other words: There is no difference between the
colimit taken intrinsically to the $K(1)$-local homotopy category or not.
\end{enumerate}
\end{lemma}

\begin{remark}
\label{rmk_TheStoryAboutTateCohomology2}Note that the even-indexed homotopy
groups have an additional summand corresponding to the infinite places which
does not come from the $P^{0}$-group. This summand occurs for the following
reason: In the usual literature, the $H^{0}$-summands in the $P^{0}$-group for
infinite places use Tate cohomology. As we had recalled in Rmk.
\ref{rmk_TheStoryAboutTateCohomology1}, these only contribute $2$-torsion
phenomena and therefore all vanish in the setup of this text. However, the
Thomason descent spectral sequences at the infinite places contribute ordinary
(as in: not Tate-modified) Galois cohomology groups $H^{0}$ at these
places.\ They can be non-zero even for complex places, and contribute
$p$-primary torsion in the $H^{0}$-groups. It is \emph{these} summands which
we additionally see next to the $P^{0}$-group in Lemma
\ref{lemma_specseq_for_restricted_product_inside}. In a nutshell: The descent
spectral sequence does not know that civilized people have chosen to use Tate
cohomology at the infinite places and therefore it outputs, what it always
outputs:\ Plain Galois cohomology. Since both $\mathbf{R}$ and $\mathbf{C}$
have $p$-cohomological dimension zero (as $p$ is odd), this whole discussion
only affects the $H^{0}$-groups. For more on this, see Pitfall
\ref{pitfall_l1}.
\end{remark}

\begin{proof}
There are only finitely many infinite places, so $L_{K(1)}$ commutes with
their product. Hence, for the rest of the proof, we focus on the finite places
and just drag the infinite ones along tacitly. We compute the homotopy groups
of each $L_{K(1)}(K/p^{k})(\mathcal{-})$ for $\mathcal{O}_{v}$ resp. $F_{v}$,
using the Thomason descent spectral sequence. For brevity, we only spell this
out for $\mathcal{O}_{v}$. We obtain%
\[
H^{2}(\mathcal{O}_{v},\mathbf{Z}/p^{k}\mathbf{Z}(j+1))\hookrightarrow\pi
_{2j}L_{K(1)}(K/p^{k})(\mathcal{O}_{v})\twoheadrightarrow H^{0}(\mathcal{O}%
_{v},\mathbf{Z}/p^{k}\mathbf{Z}(j))
\]
and%
\[
\pi_{2j-1}L_{K(1)}(K/p^{k})(\mathcal{O}_{v})\cong H^{1}(\mathcal{O}%
_{v},\mathbf{Z}/p^{k}\mathbf{Z}(j))
\]
for arbitrary $j\in\mathbf{Z}$ and $k\geq1$. Thus, we get for the even-indexed
groups%
\begin{align}
&  \prod_{v\in S\setminus S^{\prime}}H^{2}(\mathcal{O}_{v},\mathbf{Z}%
/p^{k}\mathbf{Z}(j+1))\label{lhupp0}\\
&  \qquad\qquad\hookrightarrow\pi_{2j}\left(  \prod_{v\in S\setminus
S^{\prime}}L_{K(1)}(K/p^{k})(\mathcal{O}_{v})\right)  \twoheadrightarrow
\prod_{v\in S\setminus S^{\prime}}H^{0}(\mathcal{O}_{v},\mathbf{Z}%
/p^{k}\mathbf{Z}(j))\text{,}%
\end{align}
Similarly, for the odd-indexed groups,%
\begin{equation}
\pi_{2j-1}L_{K(1)}\prod_{v\in S\setminus S^{\prime}}(K/p^{k})(\mathcal{O}%
_{v})\cong\prod_{v\in S\setminus S^{\prime}}H^{1}(\mathcal{O}_{v}%
,\mathbf{Z}/p^{k}\mathbf{Z}(j))\text{.} \label{lhupp1}%
\end{equation}
For $L_{K(1)}(K/p^{k})(F_{v})$ the computation is analogous. Define%
\[
\mathfrak{X}_{S^{\prime}}:=\prod_{v\in S\setminus S^{\prime}}L_{K(1)}%
(K/p^{k})(\mathcal{O}_{v})\times\prod_{v\in S^{\prime}}L_{K(1)}(K/p^{k}%
)(F_{v})\text{,}\qquad\text{and}\qquad\mathfrak{X}_{\infty}:=\left.
\underset{S^{\prime}\subseteq S\text{, }\#S^{\prime}<\infty}%
{\operatorname*{colim}\nolimits^{\mathsf{K}}}\right.  \mathfrak{X}_{S^{\prime
}}\text{,}%
\]
i.e.,
\begin{equation}
\mathfrak{X}_{\infty}=L_{K(1)}\left(  \left.  \underset{S^{\prime}\subseteq
S\text{, }\#S^{\prime}<\infty}{\operatorname*{colim}\nolimits^{\mathsf{Sp}}%
}\right.  \mathfrak{X}_{S^{\prime}}\right)  \text{,} \label{l_lim}%
\end{equation}
By Hovey's computation \cite[Cor. 3.5]{MR2342007}, there is a natural exact
sequence%
\begin{equation}
L_{0}(\left.  \operatorname*{colim}\nolimits_{S^{\prime}}^{\mathsf{Ab}%
}\right.  \pi_{2j-1}\mathfrak{X}_{S^{\prime}})\hookrightarrow\pi
_{2j-1}\mathfrak{X}_{\infty}\twoheadrightarrow L_{1}(\left.
\operatorname*{colim}\nolimits_{S^{\prime}}^{\mathsf{Ab}}\right.  \pi
_{2j-2}\mathfrak{X}_{S^{\prime}}) \label{lhupp3}%
\end{equation}
in $\mathsf{Ab}$ (or equivalently, in $\widehat{\mathsf{Ab}}_{p}$). As we have
stressed in the notation, the inner colimits are taken in the category of
abelian groups\footnote{We stress this because the colimits in $\widehat
{\mathsf{Ab}}_{p}$ are different. In fact, they amount to computing $L_{0}$.}.
By Eq. \ref{lhupp1} we have%
\begin{align*}
\left.  \operatorname*{colim}\nolimits_{S^{\prime}}^{\mathsf{Ab}}\right.
\pi_{2j-1}(\mathfrak{X}_{S^{\prime}})  &  =\operatorname*{colim}_{S^{\prime}%
}\left(  \prod_{v\in S\setminus S^{\prime}}H^{1}(\mathcal{O}_{v}%
,\mathbf{Z}/p^{k}\mathbf{Z}(j))\times\prod_{v\in S^{\prime}}H^{1}%
(F_{v},\mathbf{Z}/p^{k}\mathbf{Z}(j))\right) \\
&  =P_{S}^{1}(F,\mathbf{Z}/p^{k}\mathbf{Z}(j))\text{.}%
\end{align*}
From Eq. \ref{lhupp0} we get by an analogous computation an exact sequence of
abelian groups%
\begin{equation}
P_{S}^{2}(F,\mathbf{Z}/p^{k}\mathbf{Z}(j))\hookrightarrow\left.
\operatorname*{colim}\nolimits_{S^{\prime}}^{\mathsf{Ab}}\right.  \pi
_{2j-2}\mathfrak{X}_{S^{\prime}}\twoheadrightarrow P_{S}^{0}(F,\mathbf{Z}%
/p^{k}\mathbf{Z}(j-1))\oplus\bigoplus_{v\in S_{\infty}}H^{0}(F_{v}%
,\mathbf{Z}/p^{k}\mathbf{Z}(j-1))\text{.} \label{lw1x}%
\end{equation}
Now we are ready to compute the terms in Eq. \ref{lhupp3}. Taking Eq.
\ref{lw1x} as the input for the long exact sequence of derived completion,
i.e., Eq. \ref{l_ExactSequenceOfLFunctors}, we may compute its terms by Lemma
\ref{lemma_DerivedCompletionOfPGroups} and obtain (by the vanishing of both
$L_{1}$-terms on the $P_{S}^{\ast}$-groups), $L_{1}\left(  \left.
\operatorname*{colim}\nolimits_{S^{\prime}}^{\mathsf{Ab}}\right.  \pi
_{2j-2}\mathfrak{X}_{S^{\prime}}\right)  =0$ and we deduce that Eq.
\ref{lhupp3} simplifies to%
\[
P_{S}^{2}(F,\mathbf{Z}/p^{k}\mathbf{Z}(j))\hookrightarrow\pi_{2j-1}%
\mathfrak{X}_{\infty}\twoheadrightarrow P_{S}^{0}(F,\mathbf{Z}/p^{k}%
\mathbf{Z}(j-1))\oplus\bigoplus_{v\in S_{\infty}}H^{0}(F_{v},\mathbf{Z}%
/p^{k}\mathbf{Z}(j-1))\text{,}%
\]
and under this identification the two arrows in this sequence are just those
in Eq. \ref{lw1x}. For the even-indexed homotopy groups, the computation is
analogous, just with some roles of terms exchanged. This proves (1) and (2).
Specialized to Eq. \ref{l_lim}, the natural morphism $(-)\rightarrow
L_{K(1)}(-)$ becomes a morphism%
\begin{equation}
\phi\colon\left.  \underset{S^{\prime}\subseteq S\text{, }\#S^{\prime}<\infty
}{\operatorname*{colim}\nolimits^{\mathsf{Sp}}}\right.  \mathfrak{X}%
_{S^{\prime}}\longrightarrow\mathfrak{X}_{\infty}=L_{K(1)}\left(  \left.
\underset{S^{\prime}\subseteq S\text{, }\#S^{\prime}<\infty}%
{\operatorname*{colim}\nolimits^{\mathsf{Sp}}}\right.  \mathfrak{X}%
_{S^{\prime}}\right)  \text{.} \label{lw8a}%
\end{equation}
Now we may run the same argument as above for the left side. All of the above
goes through, but is simpler. We obtain a descent spectral sequence of exactly
the same shape. As the map $\phi$ in Eq. \ref{lw8a} induces an isomorphism
between these spectral sequences, we deduce that $\phi$ induces an equivalence
of spectra. This proves (3).
\end{proof}

The proof also shows that the $E_{2}$-page of the spectral sequence reads%
\begin{equation}%
\begin{tabular}
[c]{rccc}%
$\vdots\,\,$ & \multicolumn{1}{|c}{$\vdots$} & $\vdots$ & $\vdots$\\
$\mathsf{2}$ & \multicolumn{1}{|c}{$P_{S}^{0}(F,\mathbf{Z}/p^{k}%
(-1))\oplus\bigoplus_{v\in S_{\infty}}H^{0}(F_{v},\mathbf{Z}/p^{k}%
\mathbf{Z}(-1))$} & $P_{S}^{1}(F,\mathbf{Z}/p^{k}(-1))$ & $P_{S}%
^{2}(F,\mathbf{Z}/p^{k}(-1))$\\
$\mathsf{1}$ & \multicolumn{1}{|c}{$0$} & $0$ & $0$\\
$\mathsf{0}$ & \multicolumn{1}{|c}{$P_{S}^{0}(F,\mathbf{Z}/p^{k}%
(0))\oplus\bigoplus_{v\in S_{\infty}}H^{0}(F_{v},\mathbf{Z}/p^{k}%
\mathbf{Z}(0))%
{\cellcolor{lg}}%
$} & $P_{S}^{1}(F,\mathbf{Z}/p^{k}(0))$ & $P_{S}^{2}(F,\mathbf{Z}/p^{k}(0))$\\
$\mathsf{-1}$ & \multicolumn{1}{|c}{$0$} & $0%
{\cellcolor{lg}}%
$ & $0$\\
$\mathsf{-2}$ & \multicolumn{1}{|c}{$P_{S}^{0}(F,\mathbf{Z}/p^{k}%
(1))\oplus\bigoplus_{v\in S_{\infty}}H^{0}(F_{v},\mathbf{Z}/p^{k}%
\mathbf{Z}(1))$} & $P_{S}^{1}(F,\mathbf{Z}/p^{k}(1))$ & $P_{S}^{2}%
(F,\mathbf{Z}/p^{k}(1))%
{\cellcolor{lg}}%
$\\
$\vdots\,\,$ & \multicolumn{1}{|c}{$\vdots$} & $\vdots$ & $\vdots
$\\\cline{2-4}
& $\mathsf{0}$ & $\mathsf{1}$ & $\mathsf{2}$%
\end{tabular}
\ \ \label{lwui2}%
\end{equation}
with only three non-zero columns, which is certainly what one would expect
from a restricted product version of the descent spectral sequences. Again,
the shaded regions indicate the location of the $\pi_{0}$-diagonal.

\subsection{$K(1)$ on the outside}

Next, we compute the left side of Eq. \ref{l_m4}. This is a computation mostly
in the ordinary stable homotopy category, with a detour to the motivic one.

\subsubsection{The $K(1)$-local integral ad\`{e}les}

\begin{lemma}
\label{lemma_crit}Suppose $k\geq1$. Let $S^{\prime}$ be a finite set of finite
places such that $S^{\prime}\subseteq S$. Assume that $S^{\prime}$ contains
all the places of $S$ which lie over the prime $p$. Then the natural morphism%
\begin{equation}
L_{K(1)}\left(
{\textstyle\prod\nolimits_{v\in S\setminus S^{\prime}}}
(K/p^{k})(\mathcal{O}_{v})\right)  \longrightarrow%
{\textstyle\prod\nolimits_{v\in S\setminus S^{\prime}}}
L_{K(1)}(K/p^{k})(\mathcal{O}_{v}) \label{lwui3}%
\end{equation}
is an equivalence. As the intrinsic product in the $K(1)$-local homotopy
category can be computed in $\mathsf{Sp}$, the meaning of the product on the
right side is unambiguous.
\end{lemma}

For this, we need to dig a little into the details of computing $L_{K(1)}$ in
terms of inverting Bott elements. The proof is a bit lengthy and deferred to
Appendix \S \ref{appendix_ProofOfLemma_Crit}.

\subsubsection{Conclusion of the proof}

Now we have assembled all the ingredients for our proof.

\begin{proof}
[Proof of Thm. \ref{thm_compar1} for general $S$]Pick $k\geq1$. Following
Prop. \ref{prop_IdentifyAdelicBlocks} and unravelling the definition of a
restricted product (see \S \ref{subsect_RestrictedProductAsExactCategory}), we
find%
\begin{align*}
&  L_{K(1)}(K/p^{k})(\mathsf{LCA}_{\mathcal{O}_{S,ad}})\cong L_{K(1)}\left(
K(\mathsf{LCA}_{\mathcal{O}_{S,ad}})\otimes_{\mathsf{Sp}}\mathbb{S}%
/p^{k}\right) \\
&  \qquad\cong L_{K(1)}\left(  K\left(  \underset{S^{\prime}\subseteq S\text{,
}\#S^{\prime}<\infty}{\operatorname*{colim}}%
{\textstyle\prod\nolimits_{v\in S\setminus S^{\prime}}}
\mathsf{Proj}_{\mathcal{O}_{v},fg}\times%
{\textstyle\prod\nolimits_{v\in S^{\prime}}}
\mathsf{Mod}_{F_{v},fg}\right)  \otimes_{\mathsf{Sp}}\mathbb{S}/p^{k}\right)
\\
&  \qquad\cong L_{K(1)}\left(  \left.  \underset{S^{\prime}\subseteq S\text{,
}\#S^{\prime}<\infty}{\operatorname*{colim}\nolimits^{\mathsf{Sp}}}\right.
(K/p^{k})\left(
{\textstyle\prod\nolimits_{v\in S\setminus S^{\prime}}}
\mathsf{Proj}_{\mathcal{O}_{v},fg}\times%
{\textstyle\prod\nolimits_{v\in S^{\prime}}}
\mathsf{Mod}_{F_{v},fg}\right)  \right)
\end{align*}
since $K$-theory and smashing commutes with filtering colimits. The colimit is
taken in spectra. In general, the functor $L_{K(1)}$ does \textit{not} commute
with filtering colimits. It is built from the $\operatorname*{KU}%
$-localization (which \textit{is} smashing), followed by $p$-completion
\cite[Prop. 3.2, $L_{1}$ loc. cit. agrees with $L_{\operatorname*{KU}}$%
]{MR1069739}. The $p$-completion is the sole reason for the failure to be
smashing. However, as $K/p^{k}$ is a $p$-torsion spectrum, the $p$-completion
is eventually constant. Thus, restricted to the subcategory of $p$-torsion
spectra, we actually can move $L_{K(1)}$ past the colimit. We obtain%
\[
\cong\left.  \underset{S^{\prime}\subseteq S\text{, }\#S^{\prime}<\infty
}{\operatorname*{colim}\nolimits^{\mathsf{Sp}}}\right.  \left(  L_{K(1)}%
\left(
{\textstyle\prod\nolimits_{v\in S\setminus S^{\prime}}}
(K/p^{k})(\mathcal{O}_{v})\right)  \times%
{\textstyle\prod\nolimits_{v\in S^{\prime}}}
L_{K(1)}(K/p^{k})(F_{v})\right)
\]
as $K$-theory commutes with arbitrary products of exact categories (also
infinite ones) and $L_{K(1)}$ commutes with finite products.\ Combining this
with Lemma \ref{lemma_crit} (\footnote{The cited lemma demands that
$S^{\prime}$ contains all primes which lie over the prime $p$. But those
$S^{\prime}$ are cofinal by $\frac{1}{p}\in\mathcal{O}_{S}$.}), we get%
\begin{equation}
\cong\left.  \underset{S^{\prime}\subseteq S\text{, }\#S^{\prime}<\infty
}{\operatorname*{colim}\nolimits^{\mathsf{Sp}}}\right.  \left(
{\textstyle\prod\nolimits_{v\in S\setminus S^{\prime}}}
L_{K(1)}(K/p^{k})(\mathcal{O}_{v})\times%
{\textstyle\prod\nolimits_{v\in S^{\prime}}}
L_{K(1)}(K/p^{k})(F_{v})\right)  \text{.} \label{l_wcio1}%
\end{equation}
Finally, use the equivalence of Lemma
\ref{lemma_specseq_for_restricted_product_inside} (3) to see that we may
replace the colimit on the left in Eq. \ref{l_wcio1} by $\left.
\underset{S^{\prime}\subseteq S\text{, }\#S^{\prime}<\infty}%
{\operatorname*{colim}\nolimits^{\mathsf{K}}}\right.  $. Summarizing the above
discussion, we have provided a levelwise equivalence for all $k\geq1$. It is
easy to see that this equivalence is compatible under the transition maps in
the respective $k\mapsto(-)/p^{k}$ inverse systems. As both sides in our claim
are $p$-complete spectra, and thus can be written as the homotopy inverse
limits of their $(-)/p^{k}$-terms, the levelwise equivalence implies that we
have an equivalence on the nose.
\end{proof}

\begin{corollary}
The descent spectral sequence of Lemma
\ref{lemma_specseq_for_restricted_product_inside} is valid verbatim for the
limit term replaced by%
\[
\mathfrak{X}_{\infty}:=L_{K(1)}(K/p^{k})(\mathsf{LCA}_{\mathcal{O}_{S}%
,ad})\text{.}%
\]

\end{corollary}

\begin{proof}
By Thm. \ref{thm_compar1} these spectra are equivalent.
\end{proof}

\section{Non-archimedean locally compact modules\label{sect_NALCAModules}}

In this section, we discuss a slight relaxation of our assumptions. Let $F$ be
a number field and we denote by $S_{\infty}$ the set of its infinite places.
Let $S$ be \textit{any} set of places of $F$ (possibly infinite, possibly
empty, and $S_{\infty}$ need not be contained in $S$).

Define%
\begin{equation}
F_{\left\langle S\right\rangle }:=\bigoplus_{v\in S_{\infty}\setminus(S\cap
S_{\infty})}F_{v}\text{,} \label{lmips9}%
\end{equation}
i.e., $F_{\left\langle S\right\rangle }$ is the ideal in $F_{\infty}%
:=F\otimes_{\mathbf{Q}}\mathbf{R}$ corresponding to the infinite places of $F$
which are missing in $S$. Regard $F_{\left\langle S\right\rangle }$ as a
ring\footnote{It is usually not a subring of $F_{\infty}$ since in the running
conventions of this text all rings are unital and ring homomorphisms need to
preserve the unit. Proper factors in product rings therefore fail to be
subrings.}.

\begin{example}
\label{example_AllInfinitePlaces}If $S$ is a set of places containing all the
infinite places, then $F_{\left\langle S\right\rangle }=0$ is the zero ring.
\end{example}

\begin{example}
If $S$ only contains finite places, $F_{\left\langle S\right\rangle
}=F_{\infty}$.
\end{example}

Consider the functor%
\[
\rho\colon\mathsf{Proj}_{F_{\left\langle S\right\rangle },fg}\longrightarrow
\mathsf{LCA}_{\mathcal{O}_{S}}%
\]
which sends a finitely generated projective $F_{\left\langle S\right\rangle }%
$-module to itself, equipped with the real vector space topology.

\begin{lemma}
The functor $\rho$ is exact and fully faithful, and%
\[
\mathsf{D}^{b}\left(  F_{\left\langle S\right\rangle }\right)  \rightarrow
\mathsf{D}^{b}\left(  \mathsf{LCA}_{\mathcal{O}_{S}}\right)
\]
is fully faithful. It realizes $\mathsf{D}^{b}\left(  F_{\left\langle
S\right\rangle }\right)  $ as a thick triangulated full subcategory. If $S$
contains all infinite places, this is just the embedding of the zero
triangulated category.
\end{lemma}

\begin{proof}
The exactness is clear. Suppose $M,N\in\mathsf{Proj}_{F_{\left\langle
S\right\rangle },fg}$. We compute%
\begin{equation}
\rho(M,N)\colon\operatorname*{RHom}\nolimits_{F_{\left\langle S\right\rangle
}}(M,N)\longrightarrow\operatorname*{RHom}\nolimits_{\mathsf{LCA}%
_{\mathcal{O}_{S}}}(\rho M,\rho N)\text{.} \label{lmips3}%
\end{equation}
For the $\operatorname*{Ext}^{0}$-group we note that all $F_{\left\langle
S\right\rangle }$-module homomorphisms are $\mathbf{R}$-linear and that
\textit{all} $\mathbf{R}$-linear endomorphisms of a finite-dimensional real
vector space are continuous. Conversely, any continuous $\mathcal{O}_{F}%
$-module morphism between finite-dimensional real vector spaces is
automatically $\mathbf{R}$-linear (from $\mathcal{O}_{F}$-linearity, deduce
$F$-linearity and then argue by density). Thus, $\rho(M,N)$ is an isomorphism
on $\operatorname*{Hom}$-groups. Since all modules on the left side of Eq.
\ref{lmips3} are projective as algebraic $F_{\left\langle S\right\rangle }%
$-modules, all $\operatorname*{Ext}_{F_{\left\langle S\right\rangle }}^{r}%
$-groups for $r\geq1$ vanish. However, by Thm. \ref{thm_MainDecomp} (5) all
vector $\mathcal{O}_{S}$-modules are projective in $\mathsf{LCA}%
_{\mathcal{O}_{S}}$, so for $r\geq1$ all $\operatorname*{Ext}^{r}$-groups also
vanish in $\mathsf{D}^{b}\left(  \mathsf{LCA}_{\mathcal{O}_{S}}\right)  $. The
resulting embedding as a triangulated subcategory is thick. The last part of
the claim reduces to Example \ref{example_AllInfinitePlaces}.
\end{proof}

\begin{definition}
\label{def_NonArchLCAOS}Let $F$ be a number field and $S$ any set of places
(possibly infinite, possibly empty, and not necessarily containing the
infinite places). We call%
\[
\mathsf{D\,}\left.  \mathsf{LC}\mathcal{O}_{S}\right.  :=\mathsf{D}_{\infty
}^{b}\left(  \mathsf{LCA}_{\mathcal{O}_{S}}\right)  /\mathsf{D}_{\infty}%
^{b}\left(  F_{\left\langle S\right\rangle }\right)
\]
the $\infty$-category of \emph{locally compact }$\mathcal{O}_{S}%
$\emph{-modules}. Note that the definition of $\mathsf{D\,}\left.
\mathsf{LC}\mathcal{O}_{S}\right.  $ really depends on the set $S$ and not, as
one could mistakenly infer from the notation, just on the ring $\mathcal{O}%
_{S}$. We hope that this will not lead to any confusion.
\end{definition}

If $S$ contains all infinite places, we just recover $\mathsf{D}_{\infty}%
^{b}\left(  \mathsf{LCA}_{\mathcal{O}_{S}}\right)  $ as in Def. \ref{def_1}.
In most situations of interest, the quotient in Def. \ref{def_NonArchLCAOS}
can already be formed on the level of exact categories. We discuss this now.

\begin{lemma}
\label{lemma_SHasAtLeastOneFinitePlace}Suppose $S$ is any set of places of
$F$, but containing at least one finite place. Then $\mathsf{Proj}%
_{F_{\left\langle S\right\rangle },fg}$ is a Serre subcategory of
$\mathsf{LCA}_{\mathcal{O}_{S}}$, the quotient $\mathsf{LCA}_{\mathcal{O}_{S}%
}/\mathsf{Proj}_{F_{\left\langle S\right\rangle },fg}$ exists in the
$2$-category of exact categories, and%
\[
\mathsf{D\,}\left.  \mathsf{LC}\mathcal{O}_{S}\right.  \cong\mathsf{D}%
_{\infty}^{b}\left(  \mathsf{LCA}_{\mathcal{O}_{S}}/\mathsf{Proj}%
_{F_{\left\langle S\right\rangle },fg}\right)  \text{.}%
\]
Analogously, the quotient $\mathsf{LCA}_{\mathcal{O}_{S},ad}/\mathsf{Proj}%
_{F_{\left\langle S\right\rangle },fg}$ exists in the $2$-category of exact
categories and agrees with the essential image of $\mathsf{LCA}_{\mathcal{O}%
_{S},ad}$ in $\mathsf{LCA}_{\mathcal{O}_{S}}/\mathsf{Proj}_{F_{\left\langle
S\right\rangle },fg}$.
\end{lemma}

\begin{definition}
\label{def_2}Suppose $S$ is any set of places of $F$, but containing at least
one finite place. Define%
\begin{equation}
\left.  \mathsf{LC}\mathcal{O}_{S}\right.  :=\mathsf{LCA}_{\mathcal{O}_{S}%
}/\mathsf{Proj}_{F_{\left\langle S\right\rangle },fg}\text{.} \label{lmips1}%
\end{equation}
Whenever the previous definition was available, i.e., if $S$ contains all
infinite places, this definition is compatible with the former: $\left.
\mathsf{LC}\mathcal{O}_{S}\right.  \cong\mathsf{LCA}_{\mathcal{O}_{S}}$. As
with Def. \ref{def_NonArchLCAOS}, this new category $\left.  \mathsf{LC}%
\mathcal{O}_{S}\right.  $ depends on the set $S$ and not just on the ring
$\mathcal{O}_{S}$ itself. We also define $\left.  \mathsf{LC}\mathcal{O}%
_{S,ad}\right.  $ as the essential image of $\mathsf{LCA}_{\mathcal{O}_{S}%
,ad}$ in the quotient in Eq. \ref{lmips1}.
\end{definition}

\begin{proof}
[Proof of Lemma \ref{lemma_SHasAtLeastOneFinitePlace}]Since vector
$\mathcal{O}_{S}$-modules are injective\footnote{and as far as this argument
is concerned, using projectivity would work just as well} by Thm.
\ref{thm_MainDecomp} (5), all extensions of vector modules are vector modules.
From this one easily deduces that $\mathsf{Proj}_{F_{\left\langle
S\right\rangle },fg}$ is always extension-closed in $\mathsf{LCA}%
_{\mathcal{O}_{S}}$. Thus, we only need to show that subobjects and quotients
of vector $F_{\left\langle S\right\rangle }$-modules in $\mathsf{LCA}%
_{\mathcal{O}_{S}}$ are again vector $F_{\left\langle S\right\rangle }%
$-modules. To this end, let $V$ be any vector $\mathcal{O}_{S}$-module. Since
$S$ contains a finite place, we may pick some $f\in(\mathcal{O}_{S}^{\times
}\cap\mathcal{O}_{F})\setminus\mathcal{O}_{F}^{\times}$. If $U\hookrightarrow
V$ is any admissible subobject, then $U$ is a closed subgroup in a real vector
space. According to \cite[Thm. 6]{MR0442141}, this implies that $U\simeq
\mathbf{Z}^{m}\oplus\mathbf{R}^{n}$ for suitable $0\leq m,n<\infty$ in the
category of LCA groups. Define $U_{\mathbf{R}}$ as the further subgroup
corresponding to the $\mathbf{R}^{n}$-summand. For any $\alpha\in
\mathcal{O}_{S}$, $\alpha U_{\mathbf{R}}$ must be a connected subgroup of the
LCA group $U$, and therefore lie inside $U_{\mathbf{R}}$. We conclude that
$U_{\mathbf{R}}$ is an algebraic $\mathcal{O}_{S}$-submodule of $U$. It
follows that $U/U_{\mathbf{R}}$, whose underlying subgroup is $\mathbf{Z}^{m}%
$, is an algebraic $\mathcal{O}_{S}$-module. Assume $U/U_{\mathbf{R}}\neq0$.
Since $f^{\pm1}(U/U_{\mathbf{R}})=U/U_{\mathbf{R}}$, $f$ must be integral over
$\mathbf{Z}$, and therefore $f^{\pm1}\in\mathcal{O}_{F}$, i.e., $f\in
\mathcal{O}_{F}^{\times}$, which is a contradiction. Hence, $U/U_{\mathbf{R}%
}=0$ and it follows that $U$ is a vector module. Analogously, all quotients of
$V$ must be vector modules. Now, if $V$ is not just any vector $\mathcal{O}%
_{S}$-module, but specifically an $F_{v}$-vector space, it is easy to see that
also all its vector submodules and quotients will be $F_{v}$-vector spaces.
This shows that all subobjects and quotients of $F_{\left\langle
S\right\rangle }$-vector modules again are such, proving that $\mathsf{Proj}%
_{F_{\left\langle S\right\rangle },fg}$ is a Serre subcategory. Next, by
\cite[Prop. 1.4]{hr} (which is applicable since all subobjects in
$\mathsf{Proj}_{F_{\left\langle S\right\rangle }}$ are admissible subobjects)
$\mathsf{Proj}_{F_{\left\langle S\right\rangle },fg}$ is deflation-percolating

\begin{itemize}
\item in the quasi-abelian category $\mathsf{LCA}_{\mathcal{O}_{S}}$, as well as

\item in the quasi-abelian category $\mathsf{LCA}_{\mathcal{O}_{S},ad}$ (Cor.
\ref{cor_LCAAdQuasiAbelian}).
\end{itemize}

Since the essential image of $\mathsf{Proj}_{F_{\left\langle S\right\rangle
},fg}$ consists only of injective objects in $\mathsf{LCA}_{\mathcal{O}_{S}}$,
we may next invoke \cite[Thm. 1.5]{hr2} to see that the $2$-categorical
quotient%
\begin{equation}
\mathsf{LCA}_{\mathcal{O}_{S}}/\mathsf{Proj}_{F_{\left\langle S\right\rangle
},fg} \label{lmips4}%
\end{equation}
as an exact category exists, and analogously for $\mathsf{LCA}_{\mathcal{O}%
_{S},ad}/\mathsf{Proj}_{F_{\left\langle S\right\rangle },fg}$. From the
localization theorems loc. cit., we get a Verdier fiber sequence%
\[
\mathsf{D}_{(F_{\left\langle S\right\rangle },fg),\infty}^{b}(\mathsf{LCA}%
_{\mathcal{O}_{S}})\longrightarrow\mathsf{D}_{\infty}^{b}(\mathsf{LCA}%
_{\mathcal{O}_{S}})\longrightarrow\mathsf{D}_{\infty}^{b}(\mathsf{LCA}%
_{\mathcal{O}_{S}}/\mathsf{Proj}_{F_{\left\langle S\right\rangle }%
,fg})\text{.}%
\]
Here $\mathsf{D}_{(F_{\left\langle S\right\rangle },fg),\infty}^{b}%
(\mathsf{LCA}_{\mathcal{O}_{S}})$ can be defined as coming from the dg
category of complexes with homology in $\mathsf{Proj}_{F_{\left\langle
S\right\rangle },fg}$. By projectivity, the inclusion $\mathsf{D}_{\infty}%
^{b}(\mathsf{Proj}_{F_{\left\langle S\right\rangle },fg})\overset{\sim
}{\rightarrow}\mathsf{D}_{(F_{\left\langle S\right\rangle },fg),\infty}%
^{b}(\mathsf{LCA}_{\mathcal{O}_{S}})$ is an equivalence. Thus, the derived
$\infty$-category of the exact category in Eq. \ref{lmips4} agrees with the
quotient taken in Def. \ref{def_NonArchLCAOS}.
\end{proof}

\begin{example}
If $S$ is empty, $\mathsf{Proj}_{F_{\left\langle S\right\rangle },fg}$ is not
a Serre subcategory of $\mathsf{LCA}_{\mathcal{O}_{F}}$. For any finitely
generated flat $\mathcal{O}_{F}$-module $M$, equip $M_{\mathbf{R}}%
:=M\otimes_{\mathbf{Z}}\mathbf{R}$ with the real vector space topology. Then
the sequence%
\begin{equation}
M\hookrightarrow M_{\mathbf{R}}\twoheadrightarrow M_{\mathbf{R}}/M
\label{lwui7a}%
\end{equation}
is exact in $\mathsf{LCA}_{\mathcal{O}_{F}}$. It shows that vector modules can
have subobjects and quotients which fail to be vector modules. Thus,
$\mathsf{Proj}_{F_{\left\langle S\right\rangle },fg}$ is not a Serre
subcategory. Moreover, assume a quotient $Q:=\mathsf{LCA}_{\mathcal{O}_{F}%
}/\mathsf{Proj}_{F_{\left\langle S\right\rangle },fg}$ exists in the
$2$-category of exact categories. Then the quotient functor $q\colon
\mathsf{LCA}_{\mathcal{O}_{F}}\rightarrow Q$ is an exact functor, and
therefore along with $M_{\mathbf{R}}$, also $M$ and $M_{\mathbf{R}}/M$ would
have to be mapped to zero objects. From this it follows that any vector
module, any finitely generated $\mathcal{O}_{F}$-module (also the non-flat
ones, as follows from resolving them by free ones), and also all modules whose
underlying topological space is a finite-dimensional torus, must be zero in
$Q$. We see that $\ker(q)$ is a lot larger than $\mathsf{Proj}%
_{F_{\left\langle S\right\rangle },fg}$.
\end{example}

We define $\mathbb{A}_{S}:=\left.  \underset{v\in S\;}{\prod\nolimits^{\prime
}}\right.  (F_{v}:\mathcal{O}_{v})$ as in Eq. \ref{lcvn1}, but this time we do
not longer require that $S$ contains all infinite places. Then $\mathbb{A}%
_{S}$ (just as well as $\mathbb{A}_{S\cup S_{\infty}}$) is a locally compact
topological $\mathcal{O}_{S}$-module in $\mathsf{LCA}_{\mathcal{O}_{S}}$, and
pins down an object in the quotient category $\left.  \mathsf{LC}%
\mathcal{O}_{S}\right.  $.

\begin{proposition}
\label{prop_IdentifyAdelicBlocks_2}Suppose $\mathsf{A}$ is a stable
presentable $\infty$-category. Suppose $K\colon\operatorname*{Cat}_{\infty
}^{\operatorname*{ex}}\rightarrow\mathsf{A}$ is a localizing invariant which
need not commute with filtering colimits. Let $S$ be any set of places of $F$ containing

\begin{itemize}
\item at least one finite place or

\item at least all infinite places.
\end{itemize}

Then

\begin{enumerate}
\item there is an exact and duality-preserving equivalence of exact categories
with duality%
\[
\Psi\colon\left.  \underset{v\in S\;}{\prod\nolimits^{\prime}}\right.
(\mathsf{Proj}_{F_{v},fg}:\mathsf{Proj}_{\mathcal{O}_{v},fg})\overset{\sim
}{\rightarrow}\left.  \mathsf{LC}\mathcal{O}_{S,ad}\right.  \text{;}%
\]

\item the sequence%
\begin{equation}
K(\mathcal{O}_{S})\longrightarrow K\left(  \left.  \mathsf{LC}\mathcal{O}%
_{S,ad}\right.  \right)  \longrightarrow K(\left.  \mathsf{LC}\mathcal{O}%
_{S}\right.  ) \label{lmips2}%
\end{equation}
induced by the exact functor $(-)\mapsto(-)\otimes_{\mathcal{O}_{S}}%
\mathbb{A}_{S}$ as the first arrow, is a fiber sequence in $\mathsf{A}$;

\item and if we assume that $K$ is a localizing invariant commuting with
filtering colimits and (only if $\#S=\infty$) that $K$ commutes with infinite
products, then this fiber sequence can be rewritten as%
\[
K(\mathcal{O}_{S})\longrightarrow\left.  \underset{v\in S\;}{\prod
\nolimits^{\prime}}\right.  \left(  K(F_{v}):K(\mathcal{O}_{v})\right)
\longrightarrow K(\left.  \mathsf{LC}\mathcal{O}_{S}\right.  )\text{;}%
\]

\item and if $K$ refers to non-connective $K$-theory,%
\[
L_{K(1)}K(\left.  \mathsf{LC}\mathcal{O}_{S,ad}\right.  )\cong\left.
\underset{S^{\prime}\subseteq S\text{, }\#S^{\prime}<\infty}%
{\operatorname*{colim}\nolimits^{\mathsf{K}}}\right.
{\textstyle\prod\nolimits_{v\in S\setminus S^{\prime}}^{\mathsf{K}}}
L_{K(1)}K(\mathcal{O}_{v})\times%
{\textstyle\prod\nolimits_{v\in S^{\prime}}^{\mathsf{K}}}
L_{K(1)}K\left(  F_{v}\right)  \text{.}%
\]

\end{enumerate}
\end{proposition}

This reduces to Prop. \ref{prop_IdentifyAdelicBlocks}, Thm.
\ref{thm_main_weak}, Thm. \ref{thm_main_strong} if $S$ contains all infinite
places, and otherwise is more general.

\begin{proof}
(1) Exclusively relying on Prop. \ref{prop_IdentifyAdelicBlocks}, we get a
duality-preserving exact equivalence
\[
\Psi\colon\left.  \underset{v\in S\cup S_{\infty}\;}{\prod\nolimits^{\prime}%
}\right.  (\mathsf{Proj}_{F_{v},fg}:\mathsf{Proj}_{\mathcal{O}_{v}%
,fg})\overset{\sim}{\longrightarrow}\mathsf{LCA}_{\mathcal{O}_{S},ad}\text{.}%
\]
Now factor out $\mathsf{Proj}_{F_{\left\langle S\right\rangle },fg}\cong%
\prod_{S_{\infty}\setminus(S\cap S_{\infty})}\mathsf{Proj}_{F_{v},fg}$
compatibly from both exact categories. (2) Analogously, consider%
\begin{equation}%
\xymatrix{
& K(F_{\left\langle S\right\rangle}) \ar[d] \ar[r]^{\simeq} &  K(F_{\left
\langle S\right\rangle}) \ar[d] \\
K(\mathcal{O}_{S}) \ar[r] \ar@{=}[d] & K\left(  \mathsf{LCA}_{\mathcal{O}%
_{S},ad}\right) \ar[r] \ar[d] \ar@{}[dr]|{\Diamond}
& K(\mathsf{LCA}_{\mathcal{O}_{S}}) \ar[d] \\
\operatorname{fib}(\alpha) \ar[r] & K\left(  \left.\mathsf{LC}\mathcal
{O}_{S,ad}\right.\right) \ar[r]_{\alpha} & K(\left.\mathsf{LC}\mathcal{O}%
_{S}\right.),
}
\label{l_Diag_RemoveArchimedeanPlacesCategorically}%
\end{equation}
where the middle row is the fiber sequence of Thm. \ref{thm_main_weak}. The
middle and right downward sequence are fiber sequences, and the commutativity
of the square ($\Diamond$), come from Lemma
\ref{lemma_SHasAtLeastOneFinitePlace} and the corresponding Verdier
localization sequences for the attached derived $\infty$-categories. In
particular, $\alpha$ denotes the induced map on the cofibers. The equivalence
$K(F_{\left\langle S\right\rangle })\overset{\sim}{\rightarrow}%
K(F_{\left\langle S\right\rangle })$ settles that ($\Diamond$) is
bi-Cartesian. This in turn implies that the natural map $K(\mathcal{O}%
_{S})\rightarrow\operatorname*{fib}(\alpha)$ is an equivalence. This
establishes Eq. \ref{lmips2}. (3) Proceed as in the proof of Thm.
\ref{thm_main_strong}. (4) Follows directly from Thm. \ref{thm_compar1} and
(1) as we only quotient out the contribution of at most finitely many factors.
\end{proof}

\begin{example}
Let $S\ $be some set of places which does not contain all infinite places.
Then $S_{\infty}\setminus(S\cap S_{\infty})\neq\varnothing$. One might wonder
whether $\left.  \mathsf{LC}\mathcal{O}_{S}\right.  $ can also be defined as
the full subcategory $\mathcal{U}\subset\mathsf{LCA}_{\mathcal{O}_{S}}$ such
that adelic module summands isomorphic to $F_{v}$ with $v\in S_{\infty
}\setminus(S\cap S_{\infty})$ are not permitted. This seems unlikely. First,
note that%
\[
\mathbb{A}_{(S\cup S_{\infty})}/\mathcal{O}_{S}%
\]
is a compact connected $\mathcal{O}_{S}$-module with no vector module
summands. Therefore, the exact ad\`{e}le sequence%
\begin{equation}
\mathcal{O}_{S}\hookrightarrow\mathbb{A}_{(S\cup S_{\infty})}%
\twoheadrightarrow\mathbb{A}_{(S\cup S_{\infty})}/\mathcal{O}_{S}
\label{lmips5}%
\end{equation}
of Example \ref{example_AdeleSequence} shows that $\mathcal{U}$ is not
extension-closed in $\mathsf{LCA}_{\mathcal{O}_{S}}$ and thus fails to be a
fully exact subcategory. Hence, it is not clear how to choose an exact
structure on $\mathcal{U}$. Second, the image of Eq. \ref{lmips5} in
$\mathsf{LCA}_{\mathcal{O}_{S}}/\mathsf{Proj}_{F_{\left\langle S\right\rangle
},fg}$, i.e.,
\[
\mathcal{O}_{S}\hookrightarrow\mathbb{A}_{S}\twoheadrightarrow\mathbb{A}%
_{(S\cup S_{\infty})}/\mathcal{O}_{S}%
\]
is exact in $\left.  \mathsf{LC}\mathcal{O}_{S}\right.  $. Nonetheless, while
all objects in this sequence indeed lie in $\mathcal{U}$, this sequence is not
exact on the underlying abelian groups. This means that, even if $\mathcal{U}$
can be equipped with an exact structure making it equivalent to $\left.
\mathsf{LC}\mathcal{O}_{S}\right.  $, the exact structure will have to be
rather exotic.
\end{example}

\section{\label{sect_TraceMap}The trace map}

In this section $K$ stands for non-connective $K$-theory, and we focus on the
category $\left.  \mathsf{LC}\mathcal{O}_{S}\right.  $ in the situation where
$S$ contains \textit{no infinite places}. We fix an odd prime number $p$. We
also make some further choices which we will allow to vary later: Namely, $F$
will denote a number field, $S\supseteq S_{\infty}$ is a set of finite places
of $F$ such that $\frac{1}{p}\in\mathcal{O}_{S}$ and $k\geq1$ is any natural
number. Such choices being made, Prop. \ref{prop_IdentifyAdelicBlocks_2} gives
us a fiber sequence of $K(\mathcal{O}_{S})$-modules. We tensor it with
$\mathbb{S}/p^{k}$, so that%
\[
(K/p^{k})(\mathcal{O}_{S})\longrightarrow(K/p^{k})(\left.  \mathsf{LC}%
\mathcal{O}_{S,ad}\right.  )\longrightarrow(K/p^{k})(\left.  \mathsf{LC}%
\mathcal{O}_{S}\right.  )
\]
is a fiber sequence of $(K/p^{k})(\mathcal{O}_{S})$-modules. We apply the
$K(1)$-localization functor for the prime $p$. We get a fiber sequence of
$L_{K(1)}(K/p^{k})(\mathcal{O}_{S})$-modules%
\begin{equation}
L_{K(1)}(K/p^{k})(\mathcal{O}_{S})\overset{(\star)}{\rightarrow}%
L_{K(1)}(K/p^{k})(\left.  \mathsf{LC}\mathcal{O}_{S,ad}\right.  )\rightarrow
L_{K(1)}(K/p^{k})(\left.  \mathsf{LC}\mathcal{O}_{S}\right.  )\text{.}
\label{lmz1}%
\end{equation}
We added the symbol $(\star)$ to facilitate referring to this map below. The
fiber sequence induces a long exact sequence of homotopy groups. We will only
study it around $\pi_{0}$ in this section. Note that the analogue of the
restricted product descent spectral sequence of Lemma
\ref{lemma_specseq_for_restricted_product_inside} for $L_{K(1)}(K/p^{k}%
)(\left.  \mathsf{LC}\mathcal{O}_{S,ad}\right.  )$ instead of $\mathsf{LCA}%
_{\mathcal{O}_{S}}$ now has the $E_{2}$-page%
\begin{equation}%
\begin{tabular}
[c]{rccccc}%
$\vdots\,\,$ & \multicolumn{1}{|c}{} & $\vdots$ & $\vdots$ & $\vdots$ & \\
$\mathsf{2}$ & \multicolumn{1}{|c}{$\quad0\quad$} & $P_{S}^{0}(F,\mathbf{Z}%
/p^{k}(-1))$ & $P_{S}^{1}(F,\mathbf{Z}/p^{k}(-1))$ & $P_{S}^{2}(F,\mathbf{Z}%
/p^{k}(-1))$ & $\quad0\quad$\\
$\mathsf{1}$ & \multicolumn{1}{|c}{$\quad0\quad$} & $0$ & $0$ & $0$ &
$\quad0\quad$\\
$\mathsf{0}$ & \multicolumn{1}{|c}{$\quad0\quad$} & $P_{S}^{0}(F,\mathbf{Z}%
/p^{k}(0))%
{\cellcolor{lg}}%
$ & $P_{S}^{1}(F,\mathbf{Z}/p^{k}(0))$ & $P_{S}^{2}(F,\mathbf{Z}/p^{k}(0))$ &
$\quad0\quad$\\
$\mathsf{-1}$ & \multicolumn{1}{|c}{$\quad0\quad$} & $0$ & $0%
{\cellcolor{lg}}%
$ & $0$ & $\quad0\quad$\\
$\mathsf{-2}$ & \multicolumn{1}{|c}{$\quad0\quad$} & $P_{S}^{0}(F,\mathbf{Z}%
/p^{k}(1))$ & $P_{S}^{1}(F,\mathbf{Z}/p^{k}(1))$ & $P_{S}^{2}(F,\mathbf{Z}%
/p^{k}(1))%
{\cellcolor{lg}}%
$ & $\quad0\quad$\\
$\vdots\,\,$ & \multicolumn{1}{|c}{} & $\vdots$ & $\vdots$ & $\vdots$ &
\\\cline{2-6}
& $\mathsf{-1}$ & $\mathsf{0}$ & $\mathsf{1}$ & $\mathsf{2}$ & $\mathsf{\cdots
}$%
\end{tabular}
\ \ \ \ \label{lwui2_For_LCO}%
\end{equation}
which differs from Diag. \ref{lwui2} by having lost the contribution some
$H^{0}$-groups attached to the infinite places in the $0$-th column.

We obtain the two middle rows of the following diagram (we shall explain the
origin of the terms below)%
\begin{equation}%
\Scale[0.92]{\xymatrix{
& & 0 \ar[d] \\
H^{2}(\mathcal{O}_{S},\tfrac{1}{p^{k}}\mathbf{Z/Z}(1)) \ar[d]_{\beta^{2} }
\ar@{^{(}->}[r] &
\pi_{0}L_{K(1)}(K/p^{k})(\mathcal{O}_{S}) \ar[d]^{g_S} \ar@{->>}[r] &
H^{0}(\mathcal{O}_{S},\tfrac{1}{p^{k}}\mathbf{Z}/\mathbf{Z}(0))
\ar@{^{(}->}[d]^{\beta^{0} } \\
P_{S}^{2}(F,\tfrac{1}{p^{k}}\mathbf{Z}/\mathbf{Z}(1)) \ar@{->>}[d] \ar@
{^{(}->}[r] &
\pi_{0}L_{K(1)}(K/p^{k})({{\mathsf{LC}\mathcal{O}_{S,ad}}}) \ar@{->>}%
[d] \ar@{->>}[r] & P_{S}^{0}(F,\tfrac{1}{p^{k}}\mathbf{Z}/\mathbf{Z}(0)) \\
\tfrac{1}{p^{k}}\mathbf{Z}/\mathbf{Z} \ar[r] & \operatorname*{coker}%
\left(  g_{S}\right)  \text{,}
}}
\label{lBigDiagram}%
\end{equation}
where the middle downward column is induced from $\pi_{0}$ applied to
$(\star)$ in Eq. \ref{lmz1}, i.e., $g_{S}:=\pi_{0}(\star)$. The top row stems
from the Thomason descent spectral sequence applied to $\operatorname*{Spec}%
\mathcal{O}_{S}$, and bottom row comes from our restricted product version of
the descent spectral sequence, Diag. \ref{lwui2_For_LCO}. The downward arrows,
aside from $(\star)$ itself, come from the compatibility of the descent
spectral sequences along the ring homomorphisms $\mathcal{O}_{S}\rightarrow
F_{v}$ for the completions.

\begin{remark}
To elaborate a little: the arrow $(\star)$ is induced from $\mathsf{Proj}%
_{\mathcal{O}_{S},fg}\rightarrow\left.  \mathsf{LC}\mathcal{O}_{S,ad}\right.
$ by Thm. \ref{thm_main_strong}. Then Eq. \ref{lct4} relates it to the ring
map tensoring with\footnote{Unravelling the definition of $\left.
\mathsf{LC}\mathcal{O}_{S}\right.  $ in terms of $\mathsf{LCA}_{\mathcal{O}%
_{S}}$, we have to add the infinite places $S_{\infty}$ to this adelic product
as long as working in $\mathsf{LCA}_{\mathcal{O}_{S}}$. However, once we then
go to the quotient category $\left.  \mathsf{LC}\mathcal{O}_{S}\right.  $,
these summands become zero.} $\mathbb{A}_{S}=\left.  \prod\nolimits_{v\in
S\cup S_{\infty}}^{\prime}\right.  (F_{v}:\mathcal{O}_{v})$, giving
compatibility of \textit{all} the descent spectral sequences which appear in
the restricted product. For each factor of the restricted product the map is
induced from the exact functors%
\[
M\mapsto M\otimes_{\mathcal{O}_{S}}\mathcal{O}_{v}\qquad\text{(resp. }M\mapsto
M\otimes_{\mathcal{O}_{S}}F_{v}\text{)}%
\]
and the descent spectral sequences are compatible under ring maps.
\end{remark}

Write $G_{S}:=\operatorname*{Gal}(F_{S}/F)$, where $F_{S}$ is the maximal
Galois extension of $F$ which is unramified outside $S$. By the Poitou--Tate
exact sequence (\cite[Thm. 4.10]{MR2261462}) the map $\beta^{0}$ is injective
and $\beta^{2}$ has cokernel%
\begin{align}
H^{0}\left(  \mathcal{O}_{S},\left(  \tfrac{1}{p^{k}}\mathbf{Z}/\mathbf{Z}%
(1)\right)  ^{D}\right)  ^{\ast}  &  =H^{0}\left(  \mathcal{O}_{S}%
,\operatorname*{Hom}\left(  \tfrac{1}{p^{k}}\mathbf{Z}/\mathbf{Z}%
(1),F^{\operatorname*{sep}\times}\right)  \right)  ^{\ast}\label{lwxa1}\\
&  =H^{0}(\mathcal{O}_{S},\mathbf{Z}/p^{k}\mathbf{Z})^{\ast}=\left.  \tfrac
{1}{p^{k}}\mathbf{Z}\right/  \mathbf{Z}\subseteq\mathbf{Q}_{p}/\mathbf{Z}%
_{p}\text{.}\nonumber
\end{align}
We freely have used the notation loc. cit. Here $(-)^{D}$ denotes the dual (it
would correspond to $\mathcal{H}om(-,\mathbb{G}_{m})$ when written as
\'{e}tale sheaves), and $(-)^{\ast}:=\operatorname*{Hom}(-,\mathbf{Q}%
_{p}/\mathbf{Z}_{p})$ is the discrete variant of the $p$-typical Pontryagin
dual. This explains the two further terms spelled out in Diagram
\ref{lBigDiagram}.

Apply the snake lemma to Diagram \ref{lBigDiagram}. This yields the first of
the injections%
\begin{equation}
E_{F,S,k}\colon\tfrac{1}{p^{k}}\mathbf{Z}/\mathbf{Z}\hookrightarrow
\operatorname*{coker}\left(  g_{S}\right)  \hookrightarrow\pi_{0}%
L_{K(1)}(K/p^{k})(\left.  \mathsf{LC}\mathcal{O}_{S}\right.  ) \label{l_EMap}%
\end{equation}
and the second comes from the exactness of the sequence of homotopy groups
induced from Eq. \ref{lmz1}. Note that if $S\subseteq S^{\prime}$ are two
choices of places, there is an exact functor%
\begin{equation}
\eta_{S^{\prime}/S}\colon\left.  \mathsf{LC}\mathcal{O}_{S^{\prime}}\right.
\longrightarrow\left.  \mathsf{LC}\mathcal{O}_{S}\right.  \text{,}
\label{lwb2}%
\end{equation}
which simply forgets the $\mathcal{O}_{S^{\prime}}$-module structure in favour
of the $\mathcal{O}_{S}$-module structure. The topology remains unchanged.
This is obviously an exact functor.

\begin{lemma}
\label{lemma_KeyCompatibilityWithS}Suppose $S^{\prime}$ is a further set of
finite places of $F$ such that $S\subseteq S^{\prime}$. Then the diagram%
\begin{equation}%
\xymatrix{
\tfrac{1}{p^{k}}\mathbf{Z}/\mathbf{Z} \ar@{^{(}->}[r]^-{E_{F,S^{\prime},k}}
\ar@{=}[d] & \pi_{0}L_{K(1)}(K/p^k)({\mathsf{LC}\mathcal{O}_{S^{\prime}}}%
) \ar[d]^{\eta_{S^{\prime}/S}} \\
\tfrac{1}{p^{k}}\mathbf{Z}/\mathbf{Z} \ar@{^{(}->}[r]_-{E_{F,S,k}} &
\pi_{0}L_{K(1)}(K/p^k)({\mathsf{LC}\mathcal{O}_{S}})
}
\label{lmx5}%
\end{equation}
is commutative.
\end{lemma}

\begin{proof}
(Step 1) If we consider the categories of adelic modules as restricted
products of exact categories as in Eq. \ref{lmx1}, there is an inclusion
functor $i$, sending a smaller restricted product into a bigger and placing
zero objects in the novel factors, and a projection functor $p$, forgetting
factors.
\begin{equation}
\left.  \underset{v\in S^{\prime}\;}{\prod\nolimits^{\prime}}\right.
(\mathsf{Proj}_{F_{v},fg}:\mathsf{Proj}_{\mathcal{O}_{v},fg})\underset
{i}{\overset{p}{\rightleftarrows}}\left.  \underset{v\in S\;}{\prod
\nolimits^{\prime}}\right.  (\mathsf{Proj}_{F_{v},fg}:\mathsf{Proj}%
_{\mathcal{O}_{v},fg}) \label{lmx6}%
\end{equation}
The functors $p$ and $i$ are clearly $\mathcal{O}_{S}$-linear and exact.
Moreover, we have%
\begin{equation}
p\circ i=\operatorname*{id}\text{.} \label{l_PI_Equation}%
\end{equation}
Under the exact equivalence of Prop. \ref{prop_IdentifyAdelicBlocks}, we may
also regard them as exact functors%
\[
\left.  \mathsf{LC}\mathcal{O}_{S^{\prime},ad}\right.  \underset{i}%
{\overset{p}{\rightleftarrows}}\left.  \mathsf{LC}\mathcal{O}_{S,ad}\right.
\text{.}%
\]
The subdiagram of%
\[%
\xymatrix{
{\mathsf{LC}\mathcal{O}_{S^{\prime},ad}} \ar@{..>}@<2ex>[d]^{p} \ar@{^{(}%
->}[r] &
{\mathsf{LC}\mathcal{O}_{S^{\prime}}} \ar[d]^{\eta_{S^{\prime}/S}}  \\
{\mathsf{LC}\mathcal{O}_{S,ad}} \ar@<2ex>[u]^{i} \ar@{^{(}->}[r] &
{\mathsf{LC}\mathcal{O}_{S}}}%
\]
consisting solely of the \textit{solid} arrows, commutes. The horizontal
arrows stem from the inclusion of the fully exact subcategories of adelic
modules. The dotted arrow $p$ would \textit{not} commute with the other
arrows. As all functors in this diagram, even $p$, are exact, we get an
induced diagram of $K$-theory spectra, we may $K(1)$-localize the latter and
take $\pi_{0}$. We arrive at the diagram%
\begin{equation}%
\xymatrix{
\pi_{0}L_{K(1)}(K/p^k)({\mathsf{LC}\mathcal{O}_{S^{\prime},ad}}) \ar@
{..>}@<2ex>[d]^{p_{\ast}} \ar[r]^{\lambda} &
\pi_{0}L_{K(1)}(K/p^k)({\mathsf{LC}\mathcal{O}_{S^{\prime}}}) \ar
[d]^{\eta_{S^{\prime}/S,\ast}}  \\
\pi_{0}L_{K(1)}(K/p^k)({\mathsf{LC}\mathcal{O}_{S,ad}}) \ar@<2ex>[u]^{i_{\ast
}} \ar[r]_{\gamma} &
\pi_{0}L_{K(1)}(K/p^k)({\mathsf{LC}\mathcal{O}_{S}})}.
\label{lmf1}%
\end{equation}
Again, still only the solid arrows necessarily commute, i.e., $\eta
_{S^{\prime}/S,\ast}\circ\lambda\circ i_{\ast}=\gamma$. However, we can
conclude that the diagram also commutes with respect to the dotted map
$p_{\ast}$ if we restrict to the subgroup $\operatorname*{im}(i_{\ast
})\subseteq\pi_{0}L_{K(1)}(K/p^{k})(\left.  \mathsf{LC}\mathcal{O}_{S^{\prime
},ad}\right.  )$ because for $x:=i_{\ast}y$ we find%
\[
(\eta_{S^{\prime}/S,\ast}\circ\lambda)(i_{\ast}y)=\gamma y=(\gamma\circ
p_{\ast})(i_{\ast}y)
\]
thanks to Eq. \ref{l_PI_Equation}. This is the stategy of the proof:
\textit{At least restricted to }$\operatorname*{im}(i_{\ast})$\textit{,
Diagram \ref{lmf1} commutes (also) with respect to} $p_{\ast}$.\newline(Step
2)\ Now we check the commutativity of Diagram \ref{lmx5}. Suppose $x$ is given
in the upper left corner of said diagram. It gets mapped under the top
horizontal arrow%
\begin{equation}
x\in\left.  \tfrac{1}{p^{k}}\mathbf{Z}\right/  \mathbf{Z\hookrightarrow
}\operatorname*{coker}\left(  g_{S^{\prime}}\right)  \hookrightarrow\pi
_{0}L_{K(1)}(K/p^{k})(\left.  \mathsf{LC}\mathcal{O}_{S^{\prime}}\right.  )
\label{lnd1}%
\end{equation}
to the upper right corner of the diagram. As%
\[
\operatorname*{coker}(g_{S^{\prime}})=\frac{\pi_{0}L_{K(1)}(K/p^{k})(\left.
\mathsf{LC}\mathcal{O}_{S^{\prime},ad}\right.  )}{\operatorname*{im}%
(g_{S^{\prime}})}%
\]
(see Diagram \ref{lBigDiagram}), we may pick some representative $\tilde{x}%
\in\pi_{0}L_{K(1)}(K/p^{k})(\left.  \mathsf{LC}\mathcal{O}_{S^{\prime}%
,ad}\right.  )$. Now consider the diagram%
\begin{equation}%
\xymatrix{
P_{S^{\prime}}^{2}(F,\tfrac{1}{p^{k}}\mathbf{Z}/\mathbf{Z}(1))
\ar@{^{(}->}[r] &
\overset{\tilde{x}\in}{\pi_{0}L_{K(1)}(K/p^{k})({\mathsf{LC}\mathcal
{O}_{S^{\prime},ad}})} \\
P_{S}^{2}(F,\tfrac{1}{p^{k}}\mathbf{Z}/\mathbf{Z}(1)) \ar[u]^{i_{\ast}}
\ar@{^{(}->}[r] &
\pi_{0}L_{K(1)}(K/p^{k})({\mathsf{LC}\mathcal{O}_{S,ad}}).
\ar[u]_{i_{\ast}}
}
\label{lmx7}%
\end{equation}
Note that when we unravel the definition of $P_{(-)}^{2}$, the map $i_{\ast}$
on the left is again the one induced from the inclusion of the additional
factors corresponding to the places $S^{\prime}\setminus S$ as in Eq.
\ref{lmx6}. We observe that we may assume without loss of generality that the
representative $\tilde{x}$ was chosen as an element of the subgroup
$P_{S^{\prime}}^{2}(F,\tfrac{1}{p^{k}}\mathbf{Z}/\mathbf{Z}(1))$ on the upper
left in Diagram \ref{lmx7}. This follows from the construction of the map
$E_{F,S^{\prime},k}$, and notably from Diagram \ref{lBigDiagram} (all elements
in the subgroup $\tfrac{1}{p^{k}}\mathbf{Z}/\mathbf{Z}$ inside
$\operatorname*{coker}\left(  g_{S^{\prime}}\right)  $ admit representatives
coming from the $P_{S^{\prime}}^{2}$-group). Finally, we inspect what
$i_{\ast}$ does on the left in Diagram \ref{lmx7}: Unravelling the leftmost
column in Diagram \ref{lBigDiagram}, we get%
\begin{equation}%
\xymatrix{
H^{2}(\mathcal{O}_{S^{\prime}},\tfrac{1}{p^{k}}\mathbf{Z}/\mathbf{Z}%
(1)) \ar[r]^{\beta^{2} } &
\overset{\tilde{x}\in}{P_{S^{\prime}}^{2}(F,\tfrac{1}{p^{k}}\mathbf{Z}%
/\mathbf{Z}(1))}
\ar@{->>}[r]^-{\operatorname{inv}} &
\tfrac{1}{p^{k}}\mathbf{Z}/\mathbf{Z} \ar@{=}[d] \\
H^{2}(\mathcal{O}_{S},\tfrac{1}{p^{k}}\mathbf{Z}/\mathbf{Z}(1))
\ar[u] \ar[r]_{\beta^{2} } &
\underset{\tilde{x}_{\operatorname*{new}}\in}{P_{S}^{2}(F,\tfrac{1}{p^{k}%
}\mathbf{Z}/\mathbf{Z}(1))} \ar[u]^{i_{\ast}}
\ar@{->>}[r]_-{\operatorname{inv}} &
\tfrac{1}{p^{k}}\mathbf{Z}/\mathbf{Z}.
}
\label{lnd2}%
\end{equation}
Here the left upward arrow is induced from the ring inclusion $\mathcal{O}%
_{S}\subseteq\mathcal{O}_{S^{\prime}}$. It is easy to see that this indeed
commutes with the arrow $i_{\ast}$ in the middle column\footnote{The middle
upward arrow in Diagram \ref{lnd2} agrees with the leftmost upward arrow in
Diagram \ref{lmx7}.} (but note that this does \textit{not} follow by any sort
of functoriality in $i_{\ast}$. Instead, one rather starts with the ring map
$\mathcal{O}_{S}\subseteq\mathcal{O}_{S^{\prime}}$ on \'{e}tale cohomology and
needs to check that this indeed \textit{also} induces $i_{\ast}$ on the middle
terms $P^{2}$). The key point is that under the map `$\operatorname*{inv}$'
the representative $\tilde{x}$ is sent to the $x$ we had started with in Eq.
\ref{lnd1}. As the induced upward map on the quotients in Diagram \ref{lnd2},
i.e., the upward arrow all on the right, is the identity, it follows that we
can also find a class $\tilde{x}_{\operatorname*{new}}\in P_{S}^{2}%
(F,\tfrac{1}{p^{k}}\mathbf{Z}/\mathbf{Z}(1))$ in the lower row of Diagram
\ref{lnd2}, mapping to the same $x$ in $\tfrac{1}{p^{k}}\mathbf{Z}/\mathbf{Z}%
$. Under the lower horizontal arrow in Diagram \ref{lmx7}, this novel choice
of representative maps to a class in $\pi_{0}L_{K(1)}(K/p^{k})(\left.
\mathsf{LC}\mathcal{O}_{S,ad}\right.  )$ which, under $i_{\ast}$, maps to
another possible choice of representative of our element $x$. This shows that
in\ Diagram \ref{lmf1} of Step 1 we can find a representative living in the
subgroup $\operatorname*{im}(i_{\ast})$. As we had explained, restricted to
this subgroup, Diagram \ref{lmf1} also commutes with respect to the map
$p_{\ast}$. Hence, applying $p_{\ast}$ to the representative $i_{\ast}%
\tilde{x}_{\operatorname*{new}}$ on the right in%
\[
x\in\left.  \tfrac{1}{p^{k}}\mathbf{Z}\right/  \mathbf{Z\hookrightarrow
}\operatorname*{coker}\left(  g_{S^{\prime}}\right)  =\frac{\pi_{0}%
L_{K(1)}(K/p^{k})(\left.  \mathsf{LC}\mathcal{O}_{S^{\prime},ad}\right.
)}{\operatorname*{im}(g_{S^{\prime}})}%
\]
corresponds (by Step 1) to the map induced from $\eta_{S^{\prime}/S}$, and is
the identity on $\left.  \tfrac{1}{p^{k}}\mathbf{Z}\right/  \mathbf{Z}$. This
proves our claim.
\end{proof}

Whenever we change $k$, say so that $k^{\prime}\leq k$, multiplication by
$p^{k-k^{\prime}}$ induces a natural map%
\begin{equation}
K/p^{k^{\prime}}\longrightarrow K/p^{k}\text{.} \label{lwb1}%
\end{equation}
For an extension $F^{\prime}/F$ of number fields, write $S(F^{\prime})$ for
the set of places of $F^{\prime}$ which lie over the places $S$ of $F$. Then
the extension induces a finite morphism of schemes
\[
\pi\colon\operatorname*{Spec}\mathcal{O}_{F^{\prime},S(F^{\prime}%
)}\longrightarrow\operatorname*{Spec}\mathcal{O}_{F,S}\text{.}%
\]
We can construct a corresponding pushforward. On the level of exact
categories, we define%
\[
\pi_{\ast}\colon\left.  \mathsf{LC}\mathcal{O}_{F^{\prime},S(F^{\prime}%
)}\right.  \longrightarrow\left.  \mathsf{LC}\mathcal{O}_{F,S}\right.
\]
by simply forgetting the $\mathcal{O}_{F^{\prime},S(F^{\prime})}$-module
structure in favour of the $\mathcal{O}_{F,S}$-module structure. This is
clearly an exact functor.

\begin{proposition}
\label{prop_TraceBigCompat}Fix an odd prime $p$. Among all triples $(F,S,k)$
of number fields $F$, finite places $S$ such that $\frac{1}{p}\in
\mathcal{O}_{S}$, and $k\geq1$, the diagram%
\begin{equation}%
\xymatrix{
\tfrac{1}{p^{k^{\prime}}}\mathbf{Z}/\mathbf{Z} \ar@{^{(}->}[d]
\ar@{^{(}->}[r]^-{E_{F^{\prime},S^{\prime},k^{\prime}}} &
\pi_{0}L_{K(1)}(K/p^{k^{\prime}})({\mathsf{LC}\mathcal{O}_{F^{\prime
},S^{\prime}}}) \ar[d] \\
\tfrac{1}{p^{k}}\mathbf{Z}/\mathbf{Z}
\ar@{^{(}->}[r]_-{E_{F,S,k}} &
\pi_{0}L_{K(1)}(K/p^{k})({\mathsf{LC}\mathcal{O}_{F,S}})
}
\label{lwxa2}%
\end{equation}
commutes whenever $F^{\prime}/F$ is an extension (\textquotedblleft whenever
going to a smaller field\textquotedblright), $S^{\prime}\supseteq S(F^{\prime
})$, and $k^{\prime}\leq k$ (\textquotedblleft whenever increasing
$k$\textquotedblright). The right downward arrow is the induced map on
$K$-theory induced from the exact functor%
\begin{equation}
\pi_{\ast}\circ\eta_{S^{\prime}/S(F^{\prime})}\text{.} \label{lwxa3}%
\end{equation}
This functor is literally the same as the one forgetting the $\mathcal{O}%
_{F^{\prime},S^{\prime}}$-module structure in favour of the $\mathcal{O}%
_{F,S}$-module structure, keeping the same locally compact topology.
\end{proposition}

\begin{proof}
We subdivide this into various steps. (1) First, assume $F^{\prime}=F$ and
$S^{\prime}=S$, so only $k$ changes. Then we only need to follow the map in
Eq. \ref{lwb1}. On $K$-theory, this stems from multiplication by a power of
$p$, and therefore, passing through the descent spectral sequence, corresponds
to multiplication by $p$ also on the \'{e}tale sheaves on the $E_{2}$-page.
Under Eq. \ref{lwxa1} both Cartier duality $(-)^{D}$ and then $(-)^{\ast}$
reverse and reverse back this map, finally leading to the injection of the
left in Diagram \ref{lwxa2}. (2)\ Hence, for the rest of the proof we may
assume that $k$ is fixed. Next, assume $F^{\prime}=F$ and $k^{\prime}=k$, so
only $S$ changes. Then our claim reduces to Lemma
\ref{lemma_KeyCompatibilityWithS}. (3)\ Thus, we are left with the situation
that $F^{\prime}/F$ is some finite extension, $k^{\prime}=k$ and $S^{\prime
}\supseteq S(F^{\prime})$. Then by Step 2 it is sufficient to prove the
compatibility for $S^{\prime}=S(F^{\prime})$ itself.\ First, we note that
under the inclusion of the fully exact subcategory of adelic modules, the
diagram%
\[%
\xymatrix{
{\mathsf{LC}\mathcal{O}_{F^{\prime},S(F^{\prime}),ad}}
\ar@{^{(}->}[r] \ar[d]_{\pi_{\ast}} &
{\mathsf{LC}\mathcal{O}_{F^{\prime},S(F^{\prime})}}
\ar[d]^{\pi_{\ast}} \\
{\mathsf{LC}\mathcal{O}_{F,S,ad}} \ar@{^{(}->}[r] &
{\mathsf{LC}\mathcal{O}_{F,S}}\text{,}
}%
\]
commutes, i.e., if we forget the $\mathcal{O}_{F^{\prime},S(F^{\prime})}%
$-module structure in favour of the $\mathcal{O}_{F,S}$-module structure, an
adelic module in the former sense indeed will be an adelic module in the
latter sense. Then under Prop. \ref{prop_IdentifyAdelicBlocks}, i.e.,%
\[
\Psi\colon\underset{v^{\prime}\in S^{\prime}\;}{\prod\nolimits^{\prime}%
}(\mathsf{Proj}_{F_{v^{\prime}}^{\prime},fg}:\mathsf{Proj}_{\mathcal{O}%
_{F^{\prime},v^{\prime}},fg})\overset{\sim}{\rightarrow}\left.  \mathsf{LC}%
\mathcal{O}_{F^{\prime},S(F^{\prime}),ad}\right.
\]
the functor $\pi_{\ast}$ gets identified with the exact functors on
$\mathsf{Proj}_{F_{v^{\prime}}^{\prime},fg}$ (resp. $\mathsf{Proj}%
_{\mathcal{O}_{F^{\prime},v^{\prime}},fg}$) which also forget the
$F_{v^{\prime}}^{\prime}$-vector space structure in favour of the $F_{v}%
$-vector space structure (for $v^{\prime}$ running through the finitely many
places over a place $v$ of $F$), and analogously for $\mathcal{O}_{F^{\prime
},v^{\prime}}$, which are finite $\mathcal{O}_{F,v}$-algebras. This is the
same exact functor underlying finite pushforwards on
ordinary\footnote{ordinary, as opposed to locally compact $K$-theory}
$K$-theory. The situation now is essentially reduced to the one in
Blumberg--Mandell \cite[Thm. 1.5]{MR4121155}. By the functoriality of Brauer
groups under finite field extensions, the corestriction on Galois cohomology
induces the identity on the Hasse invariants. This settles the remaining
compatibility and concludes the proof.
\end{proof}

By our assumption that $\frac{1}{p}\in\mathcal{O}_{S}$, there is an exact
functor%
\begin{equation}
\mathsf{LCA}_{\mathcal{O}_{S}}\longrightarrow\mathsf{LCA}_{\mathbf{Z}\left[
\frac{1}{p}\right]  } \label{lgg1}%
\end{equation}
which forgets the $\mathcal{O}_{S}$-module structure in favour of the
$\mathbf{Z[}\frac{1}{p}]$-module structure coming from $\mathbf{Z[}\frac{1}%
{p}]\subseteq\mathcal{O}_{S}$. It induces an exact functor%
\[
\left.  \mathsf{LC}\mathcal{O}_{S}\right.  \longrightarrow\left.
\mathsf{LC}\mathcal{O}_{\mathbf{Q},\{p\}}\right.  \text{.}%
\]
Following the idea of Blumberg--Mandell's \cite[Thm. 1.5]{MR4121155}, the
category $\left.  \mathsf{LC}\mathcal{O}_{\mathbf{Q},\{p\}}\right.  $ comes
with a very natural construction of a trace map.

\begin{theorem}
[{Blumberg--Mandell, \cite[Thm. 1.5]{MR4121155}}]\label{thm_BMTraceForQ}Let
$K$ be non-connective $K$-theory. There is a canonical isomorphism%
\begin{equation}
\pi_{0}L_{K(1)}(K/p^{k})(\left.  \mathsf{LC}\mathcal{O}_{\mathbf{Q}%
,\{p\}}\right.  )\cong\tfrac{1}{p^{k}}\mathbf{Z}/\mathbf{Z}\subseteq
\mathbf{Q}_{p}/\mathbf{Z}_{p}\text{.} \label{lgg2}%
\end{equation}

\end{theorem}

This now can be linked up to the Brown--Comenetz dual. We recall this type of
duality in \S \ref{sect_Appendix_DualityInKOneLocalHptyCat}. By Eq.
\ref{lqwxt3} there is a canonical isomorphism%
\begin{equation}
\operatorname*{Hom}\nolimits_{Ho(\mathsf{Sp})}(Z,I_{\mathbf{Q}_{p}%
/\mathbf{Z}_{p}}\mathbb{S}_{\widehat{p}})\cong\operatorname*{Hom}%
\nolimits_{\mathsf{Ab}}\left(  \pi_{0}(Z\otimes_{\mathsf{Sp}}\mathcal{N}%
),\mathbf{Q}_{p}/\mathbf{Z}_{p}\right)  \label{lqtxa1}%
\end{equation}
for all $K(1)$-local spectra $Z$, where $\mathcal{N}:=\Sigma^{-1}%
\mathbb{S}\mathbf{Q}_{p}/\mathbf{Z}_{p}$ and $\mathbb{S}_{\widehat{p}}$
denotes the $K(1)$-local sphere\footnote{this is just the tensor unit for the
intrinsic symmetric monoidal structure of the $K(1)$-local homotopy category}.
By the compatibility under increasing $k$, the maps%
\begin{equation}
\pi_{0}L_{K(1)}(K/p^{k})(\left.  \mathsf{LC}\mathcal{O}_{\mathbf{Q}%
,\{p\}}\right.  )\longrightarrow\tfrac{1}{p^{k}}\mathbf{Z}/\mathbf{Z}
\label{lqtxb1}%
\end{equation}
induce a canonical map on the colimit, which by the following little
computation which is valid for any spectrum $X$,%
\begin{align}
\operatorname*{colim}X/p^{k}  &  \cong\operatorname*{colim}\left(
X\otimes_{\mathsf{Sp}}\mathbb{S}\mathbf{Z}/p^{k}\mathbf{Z}\right)
\label{lwcv4a}\\
&  \cong X\otimes_{\mathsf{Sp}}\mathbb{S}\mathbf{Q}_{p}/\mathbf{Z}_{p}\cong
X\otimes_{\mathsf{Sp}}\Sigma\mathcal{N}\cong\Sigma X\otimes_{\mathsf{Sp}%
}\mathcal{N}\text{,}\nonumber
\end{align}
defines a map%
\begin{equation}
\pi_{0}(\Sigma L_{K(1)}K(\left.  \mathsf{LC}\mathcal{O}_{\mathbf{Q}%
,\{p\}}\right.  )\otimes_{\mathsf{Sp}}\mathcal{N})\longrightarrow
\mathbf{Q}_{p}/\mathbf{Z}_{p}\text{.} \label{lwcv4}%
\end{equation}
By Eq. \ref{lqtxa1} this map uniquely determines a map
\[
\Sigma L_{K(1)}K(\left.  \mathsf{LC}\mathcal{O}_{\mathbf{Q},\{p\}}\right.
)\longrightarrow I_{\mathbf{Q}_{p}/\mathbf{Z}_{p}}\mathbb{S}_{\widehat{p}}%
\]
and therefore a map $L_{K(1)}K(\left.  \mathsf{LC}\mathcal{O}_{\mathbf{Q}%
,\{p\}}\right.  )\rightarrow\Sigma^{-1}I_{\mathbf{Q}_{p}/\mathbf{Z}_{p}%
}\mathbb{S}_{\widehat{p}}$, which by $p$-completeness (and more concretely by
Eq. \ref{lrrb1}) is equivalent to specifying a map to the Anderson dual of the
sphere%
\begin{equation}
L_{K(1)}K(\left.  \mathsf{LC}\mathcal{O}_{\mathbf{Q},\{p\}}\right.
)\longrightarrow I_{\mathbf{Z}_{p}}\mathbb{S}_{\widehat{p}}\text{.}
\label{lgg4}%
\end{equation}
This is the key input for the trace map.

\begin{definition}
\label{def_BigTrace}Let $K$ be non-connective $K$-theory. We define a trace
map%
\[
\operatorname*{Tr}\nolimits_{F,S}\colon L_{K(1)}K(\left.  \mathsf{LC}%
\mathcal{O}_{F,S}\right.  )\longrightarrow I_{\mathbf{Z}_{p}}\mathbb{S}%
_{\widehat{p}}%
\]
by composing the map induced from the exact functor in Eq. \ref{lgg1} and the
map in Eq. \ref{lgg4}.
\end{definition}

Usually, we work with its finite level version. From $X\overset{\cdot
p}{\longrightarrow}X\longrightarrow X/p$, where $X$ is arbitrary, dualization
yields $I_{\mathbf{Z}_{p}}X\overset{\cdot p}{\longleftarrow}I_{\mathbf{Z}_{p}%
}X\longleftarrow I_{\mathbf{Z}_{p}}(X/p)$ and therefore%
\begin{equation}
(I_{\mathbf{Z}_{p}}X)/p\cong\Sigma I_{\mathbf{Z}_{p}}(X/p)\text{.}
\label{lmips27}%
\end{equation}
Reducing the trace map of Def. \ref{def_BigTrace} modulo $p^{k}$, we find%
\begin{equation}
L_{K(1)}(K/p^{k})(\left.  \mathsf{LC}\mathcal{O}_{S}\right.  )\overset
{\operatorname*{Tr}\nolimits_{F,S}/p^{k}}{\longrightarrow}(I_{\mathbf{Z}_{p}%
}\mathbb{S}_{\widehat{p}})/p^{k}\underset{\text{Eq. \ref{lmips27}}}%
{\overset{\sim}{\longrightarrow}}\Sigma I_{\mathbf{Z}_{p}}(\mathbb{S}%
/p^{k})\underset{\text{Eq. \ref{lrrb1}}}{\overset{\sim}{\longrightarrow}%
}I_{\mathbf{Q}_{p}/\mathbf{Z}_{p}}\mathbb{S}/p^{k}\text{.}
\label{l_FiniteLevelVersionOfTraceMap}%
\end{equation}
We call this the \emph{finite level version} of the trace map. For $\left.
\mathsf{LC}\mathcal{O}_{\mathbf{Q},\{p\}}\right.  $ this recovers Eq.
\ref{lqtxb1}.

The construction of Def. \ref{def_BigTrace} generalizes \cite[Thm.
1.5]{MR4121155} to more general and possibly infinite choices of places $S$.
Unlike loc. cit. we reduce to the case of $\left.  \mathsf{LC}\mathcal{O}%
_{\mathbf{Q},\{p\}}\right.  $ by the forgetful functor $\left.  \mathsf{LC}%
\mathcal{O}_{F,S}\right.  \rightarrow\left.  \mathsf{LC}\mathcal{O}%
_{\mathbf{Q},\{p\}}\right.  $ instead of proving the existence of a unique
system of sections. However, a formulation in terms of such a unique system of
sections is also possible:

\begin{proposition}
[Compatibility with method of \cite{MR4121155}]\label{prop_VIsASectionOfE}Fix
an odd prime $p$. Let $K$ be non-connective $K$-theory.

\begin{enumerate}
\item Among all triples $(F,S,k)$ of number fields $F$, sets of finite places
$S$ such that $\frac{1}{p}\in\mathcal{O}_{S}$, and $k\geq1$, there is a unique
system of abelian group homomorphisms%
\[
V_{F,S,k}\colon\pi_{0}L_{K(1)}(K/p^{k})(\left.  \mathsf{LC}\mathcal{O}%
_{F,S}\right.  )\longrightarrow\tfrac{1}{p^{k}}\mathbf{Z}/\mathbf{Z}%
\]
with $V_{F,S,k}\circ E_{F,S,k}=\operatorname*{id}_{\tfrac{1}{p^{k}}%
\mathbf{Z}_{p}/\mathbf{Z}_{p}}$, such that the following holds: Suppose
$(F,S,k)$ and $(F^{\prime},S^{\prime},k^{\prime})$ are such that $F^{\prime
}/F$ is an extension, $S^{\prime}$ contains $S(F^{\prime})$, and $k^{\prime
}\leq k$. Then%
\begin{equation}%
\xymatrix{
\pi_{0}L_{K(1)}(K/p^{k^{\prime}})({\mathsf{LC}\mathcal{O}_{F^{\prime
},S^{\prime}}}) \ar[r] \ar[d] &
\tfrac{1}{p^{k^{\prime}}}\mathbf{Z}/\mathbf{Z} \ar@{^{(}->}[d] \\
\pi_{0}L_{K(1)}(K/p^{k})({\mathsf{LC}\mathcal{O}_{F,S}}) \ar[r] &
\tfrac{1}{p^{k}}\mathbf{Z}/\mathbf{Z}
}
\label{lcz3}%
\end{equation}
commutes, where the left downward arrow is induced from the same exact functor
$\left.  \mathsf{LC}\mathcal{O}_{F^{\prime},S^{\prime}}\right.  \rightarrow
\left.  \mathsf{LC}\mathcal{O}_{F,S}\right.  $ as in Eq. \ref{lwxa3}, i.e.,
forgetting the $\mathcal{O}_{F^{\prime},S^{\prime}}$-module structure in
favour of the $\mathcal{O}_{F,S}$-module structure.

\item Assembling the maps in (1) as a colimit and using Eq. \ref{lwcv4a}, the
maps%
\[
V_{F,S}\colon\pi_{0}(\Sigma L_{K(1)}K(\left.  \mathsf{LC}\mathcal{O}%
_{F,S}\right.  )\otimes_{\mathsf{Sp}}\mathcal{N})\longrightarrow\mathbf{Q}%
_{p}/\mathbf{Z}_{p}%
\]
define maps%
\[
\Sigma L_{K(1)}K(\left.  \mathsf{LC}\mathcal{O}_{F,S}\right.  )\longrightarrow
I_{\mathbf{Q}_{p}/\mathbf{Z}_{p}}\mathbb{S}_{\widehat{p}}%
\]
which after applying $\Sigma^{-1}$ and identifying the shifted Brown--Comenetz
dual as an Anderson dual, agree with $\operatorname*{Tr}\nolimits_{F,S}$.
\end{enumerate}
\end{proposition}

\begin{proof}
(1) First, we restrict our attention to triples $(F,S,k)$ of the following
shape: $F$ is an arbitrary number field, $k\geq1$ is arbitrary, but%
\begin{equation}
S(F):=\{\text{places above }p\}\cup S_{\infty}\text{.} \label{lcz4}%
\end{equation}
Restricted to these triples, which can be indexed merely by pairs $(F,k)$, a
unique system of $V_{F,S,k}$ exists by the construction of Blumberg--Mandell
\cite[Thm. 1.5]{MR4121155}. For the sake of completeness, we recall this: For
$(\mathbf{Q},k)$ Diagram \ref{lBigDiagram} specializes to%
\[%
\Scale[0.92]{\xymatrix{
& & 0 \ar[d] \\
H^{2}({\mathbf{Z}\left[  \frac{1}{p}\right]},\tfrac{1}{p^{k}}\mathbf
{Z/Z}(1)) \ar[d]_{\beta^{2} }
\ar@{^{(}->}[r] &
\pi_{0}L_{K(1)}(K/p^{k})({\mathbf{Z}\left[  \frac{1}{p}\right]}) \ar
[d]^{g_{\{p,\infty\}}} \ar@{->>}[r] &
H^{0}({\mathbf{Z}\left[  \frac{1}{p}\right]},\tfrac{1}{p^{k}}\mathbf
{Z}/\mathbf{Z}(0))
\ar@{=}[d]^{\beta^{0} } \\
P_{\{p,\infty\}}^{2}({\mathbf{Q}},\tfrac{1}{p^{k}}\mathbf{Z}/\mathbf
{Z}(1)) \ar@{->>}[d] \ar@{^{(}->}[r] &
\pi_{0}L_{K(1)}(K/p^{k})({\mathsf{LC}\mathcal{O}_{\mathbf{Q},\{p\}}}%
) \ar@{->>}[d] \ar@{->>}[r] & H^{0}({\mathbf{Q}_{p}},\tfrac{1}{p^{k}}%
\mathbf{Z}/\mathbf{Z}(0)) \ar[d] \\
\tfrac{1}{p^{k}}\mathbf{Z}/\mathbf{Z} \ar[r]_{\cong} & \operatorname
*{coker}\left(  g_{\{p,\infty\}}\right) \ar[r] & 0 \text{,}
}}%
\]
where $\beta^{0}$ is an isomorphism. This means that $V_{\mathbf{Q}%
,\{p\},k}:=E_{\mathbf{Q},\{p\},k}^{-1}$ is the unique possible choice (this
provides the proof of Theorem \ref{thm_BMTraceForQ} as well). Now by Prop.
\ref{prop_TraceBigCompat} all triples $(F,S,k)$ and $(F^{\prime},S^{\prime
},k^{\prime})$ for which we have to check the commutativity of Diagram
\ref{lcz3}, the maps $E_{F,S,k}$ form a compatible system of injections. This
means that we (leaving $k$ fixed), we can always move downward to the triple
$(\mathbf{Q},\{p\},k)$, which is the (unique) final object in the poset of all
triples with fixed $k$, and use the uniquely determined section $V_{\mathbf{Q}%
,\{p\},k}$ there. (2) By Prop. \ref{prop_TraceBigCompat} this system is
exactly the one used in Def. \ref{def_BigTrace} to represent the maps
$\operatorname*{Tr}\nolimits_{F,S}$. This settles the claim.
\end{proof}

\section{\label{sect_Duality}Duality}

In this section $K$ denotes non-connective $K$-theory. Let $F$ be a number
field and let $S$ be a (possibly infinite) set of places of $F$ such that
$\frac{1}{p}\in\mathcal{O}_{S}$.$\ $In this section, we shall prove two global
duality results:

\begin{itemize}
\item ($S$ contains all infinite places) The first duality concerns
$\mathsf{LCA}_{\mathcal{O}_{S}}$, i.e., the category of all locally compact
$\mathcal{O}_{S}$-modules.

\item ($S$ arbitrary) The second duality concerns $\left.  \mathsf{LC}%
\mathcal{O}_{S}\right.  $, i.e., the modified category of
\S \ref{sect_NALCAModules}, where we can remove unwanted infinite places from
$S$.
\end{itemize}

We feel that the category $\mathsf{LCA}_{\mathcal{O}_{S}}$ is a bit more
interesting due to the simplicity of its definition. The category $\left.
\mathsf{LC}\mathcal{O}_{S}\right.  $ is a little less natural, yet it produces
the cleaner duality statement.

\subsection{$S$ contains all infinite places\label{subsect_SInfPlaces}}

In this subsection we assume that $S$ contains all infinite places and that
$\frac{1}{p}\in\mathcal{O}_{S}$. We shall set up a descent spectral sequence
for $L_{K(1)}K(\mathsf{LCA}_{\mathcal{O}_{S}})$. We shall rely on the
presentation of Milne \cite[Ch. 1, \S 4]{MR2261462}. We abbreviate by
$G_{S}:=\operatorname*{Gal}(F_{S}/F)$ the Galois group of the maximal Galois
extension $F_{S}/F$ which is unramified outside $S$. For $L/F$ any finite
extension, we write

$S^{\prime}$ for the places of $L$ which lie above places of $S$,

$J_{L,S}:=\left.  \underset{v\in S^{\prime}\;}{\prod\nolimits^{\prime}%
}\right.  \left(  L_{v}^{\times}:\mathcal{O}_{L,v}^{\times}\right)  $ for the
$S^{\prime}$-id\`{e}les of $L$,

$E_{L,S}:=\mathcal{O}_{L,S^{\prime}}^{\times}$ for the $S^{\prime}$-unit group
of $L$,

$C_{L,S}:=J_{L,S}/E_{L,S}$ for the $S^{\prime}$-id\`{e}le class group,

Moreover, $E_{S}:=\operatorname*{colim}_{L}E_{L,S}$, $J_{S}%
:=\operatorname*{colim}_{L}J_{L,S}$, and $C_{S}:=\operatorname*{colim}%
_{L}C_{L,S}$ denote the associated Galois modules over $G_{S}$.

\begin{proposition}
[Undualized spectral sequence]\label{prop_duality1}Let $K$ denote
non-connective $K$-theory. Let $p$ be an odd prime, $F$ a number field, $S$ a
(possibly infinite) set of places of $F$ containing the infinite places and
such that $\frac{1}{p}\in\mathcal{O}_{S}$. Suppose $k\geq1$. Then there is a
convergent spectral sequence%
\begin{align*}
&  E_{2}^{i,j}:=\operatorname*{Ext}\nolimits_{G_{S}}^{i}\left(  \mathbf{Z}%
/p^{k}(1+\tfrac{j}{2}),C_{S}\right) \\
&  \qquad\qquad\Rightarrow\pi_{-i-j}L_{K(1)}(K/p^{k})K(\mathsf{LCA}%
_{\mathcal{O}_{S}})\text{,}%
\end{align*}
where $C_{S}$ denotes the $S$-id\`{e}le class group of the number field.
\end{proposition}

The $E_{2}$-page has the shape%

\[%
\begin{tabular}
[c]{rccccc}%
$\vdots\,\,$ & \multicolumn{1}{|c}{} & $\vdots$ & $\vdots$ & $\vdots$ & \\
$\mathsf{2}$ & \multicolumn{1}{|c}{$\quad0\quad$} & $\operatorname*{Ext}%
\nolimits_{G_{S}}^{0}\left(  \mathbf{Z}/p^{k}(2),C_{S}\right)  $ &
$\operatorname*{Ext}\nolimits_{G_{S}}^{1}\left(  \mathbf{Z}/p^{k}%
(2),C_{S}\right)  $ & $\operatorname*{Ext}\nolimits_{G_{S}}^{2}\left(
\mathbf{Z}/p^{k}(2),C_{S}\right)  $ & $\quad0\quad$\\
$\mathsf{1}$ & \multicolumn{1}{|c}{$\quad0\quad$} & $0$ & $0$ & $0$ &
$\quad0\quad$\\
$\mathsf{0}$ & \multicolumn{1}{|c}{$\quad0\quad$} & $\operatorname*{Ext}%
\nolimits_{G_{S}}^{0}\left(  \mathbf{Z}/p^{k}(1),C_{S}\right)
{\cellcolor{lg}}%
$ & $\operatorname*{Ext}\nolimits_{G_{S}}^{1}\left(  \mathbf{Z}/p^{k}%
(1),C_{S}\right)  $ & $\operatorname*{Ext}\nolimits_{G_{S}}^{2}\left(
\mathbf{Z}/p^{k}(1),C_{S}\right)  $ & $\quad0\quad$\\
$\mathsf{-1}$ & \multicolumn{1}{|c}{$\quad0\quad$} & $0$ & $0%
{\cellcolor{lg}}%
$ & $0$ & $\quad0\quad$\\
$\mathsf{-2}$ & \multicolumn{1}{|c}{$\quad0\quad$} & $\operatorname*{Ext}%
\nolimits_{G_{S}}^{0}\left(  \mathbf{Z}/p^{k}(0),C_{S}\right)  $ &
$\operatorname*{Ext}\nolimits_{G_{S}}^{1}\left(  \mathbf{Z}/p^{k}%
(0),C_{S}\right)  $ & $\operatorname*{Ext}\nolimits_{G_{S}}^{2}\left(
\mathbf{Z}/p^{k}(0),C_{S}\right)
{\cellcolor{lg}}%
$ & $\quad0\quad$\\
$\vdots\,\,$ & \multicolumn{1}{|c}{} & $\vdots$ & $\vdots$ & $\vdots$ &
\\\cline{2-6}
& $\mathsf{-1}$ & $\mathsf{0}$ & $\mathsf{1}$ & $\mathsf{2}$ & $\mathsf{\cdots
}$%
\end{tabular}
\ \ \
\]
with differential $d_{2}\colon E_{2}^{i,j}\rightarrow E_{2}^{i+2,j-1}$. The
shaded regions mark the $\pi_{0}$-diagonal.

\begin{proof}
(Step 1) Consider the exact functor%
\[
\mathsf{Proj}_{\mathcal{O}_{S},fg}\longrightarrow\mathsf{LCA}_{\mathcal{O}%
_{S}}\text{,}\qquad M\longmapsto M\otimes_{\mathcal{O}_{S}}\mathbb{A}%
_{S}\text{.}%
\]
Thanks to Prop. \ref{prop_IdentifyAdelicBlocks}, we may identify the induced
map on $K$-theory, which we shall call $\delta$ in this proof, with%
\begin{equation}
\delta\colon K(\mathcal{O}_{S})\longrightarrow\left.  \underset{v\in
S\;}{\prod\nolimits^{\prime}}\right.  \left(  K(F_{v}):K(\mathcal{O}%
_{v})\right)  \text{,} \label{lwui4}%
\end{equation}
where the induced maps on the factors of the restricted products are $M\mapsto
M\otimes_{\mathcal{O}_{S}}F_{v}$ (resp. $M\mapsto M\otimes_{\mathcal{O}_{S}%
}\mathcal{O}_{v}$) and therefore correspond to scheme morphisms. Now apply
$L_{K(1)}$ to Eq. \ref{lwui4}. Thanks to Theorem \ref{thm_compar1} and Lemma
\ref{lemma_specseq_for_restricted_product_inside} (3), $L_{K(1)}(\delta)$
unravels to%
\begin{equation}
L_{K(1)}K(\mathcal{O}_{S})\longrightarrow\left.  \underset{S^{\prime}\subseteq
S\text{, }\#S^{\prime}<\infty}{\operatorname*{colim}\nolimits^{\mathsf{Sp}}%
}\right.  \left(
{\textstyle\prod\nolimits_{v\in S\setminus S^{\prime}}}
L_{K(1)}K(\mathcal{O}_{v})\times%
{\textstyle\prod\nolimits_{v\in S^{\prime}}}
L_{K(1)}K\left(  F_{v}\right)  \right)  \text{.} \label{lwui5}%
\end{equation}
Thanks to the levelwise compatibility of all transition maps in the colimit
with the respective Thomason descent spectral sequences, we get a descent
spectral sequence for the homotopy cofiber of the map in Eq. \ref{lwui5}. From
now on, we switch to working mod $p^{k}$. We will spell out the spectral
sequences in terms of Galois cohomology. As in Def. \ref{def_PGroups}, write
$G_{S}:=\operatorname*{Gal}(F_{S}/F)$, where $F_{S}$ is the maximal Galois
extension of $F$ which is unramified outside $S$. Then \'{e}tale sheaves on
$\operatorname*{Spec}\mathcal{O}_{S}$ or $\operatorname*{Spec}F_{v}$ can both
be regarded as $G_{S}$-modules, using that the local Galois group
$G_{v}:=\operatorname*{Gal}(F_{v}^{\operatorname*{sep}}/F_{v})$ identifies
with some decomposition subgroup in $G_{S}$. We write $\operatorname*{I}%
\nolimits_{G_{v}}^{G_{S}}$ for coinduction of Galois modules. The spectral
sequence of the homotopy cofiber takes the shape
\begin{equation}
E_{2}^{i,j}:=H^{i}\left(  G_{S},\operatorname*{cofib}\left(  \mathbf{Z}%
/p^{k}(-\tfrac{j}{2})\overset{\delta}{\rightarrow}\left.  \underset{v\in S\;}{%
{\textstyle\prod\nolimits^{\prime}}
}\right.  \left(  \operatorname*{I}\nolimits_{G_{v}}^{G_{S}}\mathbf{Z}%
/p^{k}(-\tfrac{j}{2}):\operatorname*{I}\nolimits_{G_{v},\operatorname*{ur}%
}^{G_{S}}\mathbf{Z}/p^{k}(-\tfrac{j}{2})\right)  \right)  \right)
\label{ljbj1}%
\end{equation}
with limit term%
\[
\Rightarrow\pi_{-i-j}\left(  \operatorname*{cofib}\left(  L_{K(1)}%
(K/p^{k})(\mathcal{O}_{S})\overset{\delta}{\rightarrow}\left.  \underset{v\in
S\;}{%
{\textstyle\prod\nolimits^{\prime}}
}\right.  \left(  L_{K(1)}(K/p^{k})(F_{v}):L_{K(1)}(K/p^{k})(\mathcal{O}%
_{v})\right)  \right)  \right)  \text{.}%
\]
Thanks to our fiber sequence from \S \ref{sect_TheMainFiberSequence}, the
limit term simplifies as follows.%
\begin{align}
&  \cdots\Rightarrow\pi_{-i-j}\operatorname*{cofib}\left(  L_{K(1)}%
(K/p^{k})(\mathcal{O}_{S})\overset{\delta}{\longrightarrow}L_{K(1)}%
(K/p^{k})(\mathsf{LCA}_{\mathcal{O}_{S},ad})\right) \nonumber\\
&  \qquad\cong\pi_{-i-j}L_{K(1)}\operatorname*{cofib}\left(  (K/p^{k}%
)(\mathcal{O}_{S})\overset{\delta}{\longrightarrow}(K/p^{k})(\mathsf{LCA}%
_{\mathcal{O}_{S},ad})\right) \label{ljk1}\\
&  \qquad\cong\pi_{-i-j}L_{K(1)}(K/p^{k})K(\mathsf{LCA}_{\mathcal{O}_{S}%
})\qquad\text{(by Thm. \ref{thm_main_strong}).}\nonumber
\end{align}
(Step 2) We proceed by unravelling the $E_{2}$-page of the spectral sequence
in Eq. \ref{ljbj1}. We rewrite the Galois cohomology groups in terms of
$H^{i}(G_{S},-)\cong\operatorname*{Ext}\nolimits_{G_{S}}^{i}(\mathbf{Z},-)$.
For brevity, we drop spelling out the second slot in the restricted products
$\left.  \underset{v\in S\;}{%
{\textstyle\prod\nolimits^{\prime}}
}\right.  (-:-)$. Since $\operatorname*{RHom}$ commutes with cones,
\begin{align}
&  \operatorname*{RHom}\nolimits_{G_{S}}\left(  \operatorname*{cofib}\left(
\mathbf{Z}/p^{k}(-\tfrac{j}{2})\rightarrow\left.  \underset{v\in S\;}{%
{\textstyle\prod\nolimits^{\prime}}
}\right.  \operatorname*{I}\nolimits_{G_{v}}^{G_{S}}\mathbf{Z}/p^{k}%
(-\tfrac{j}{2})\right)  \right) \nonumber\\
&  \cong\operatorname*{cofib}\left(  \operatorname*{RHom}\nolimits_{G_{S}%
}\left(  \mathbf{Z},\mathbf{Z}/p^{k}(-\tfrac{j}{2})\right)  \rightarrow
\operatorname*{RHom}\nolimits_{G_{S}}\left(  \mathbf{Z},\left.  \underset{v\in
S\;}{%
{\textstyle\prod\nolimits^{\prime}}
}\right.  \operatorname*{I}\nolimits_{G_{v}}^{G_{S}}\mathbf{Z}/p^{k}%
(-\tfrac{j}{2})\right)  \right) \nonumber\\
&  \cong\operatorname*{cofib}\left(  \operatorname*{RHom}\nolimits_{G_{S}%
}\left(  \mathbf{Z},\mathbf{Z}/p^{k}(-\tfrac{j}{2})\right)  \rightarrow\left.
\underset{v\in S\;}{%
{\textstyle\prod\nolimits^{\prime}}
}\right.  \operatorname*{RHom}\nolimits_{G_{S}}\left(  \mathbf{Z}%
,\operatorname*{I}\nolimits_{G_{v}}^{G_{S}}\mathbf{Z}/p^{k}(-\tfrac{j}%
{2})\right)  \right)  \label{lwui6}%
\end{align}
since $\operatorname*{Hom}$-groups commute with arbitrary limits in their
second argument and moreover the present $\operatorname*{Hom}$-group commutes
also with filtering colimits in the second argument since $\mathbf{Z}$ is a
compact object in $G_{S}$-modules. Now we use input from Class Field
Theory:\ We shall employ the exact sequence of $G_{S}$-modules%
\begin{equation}
E_{S}\hookrightarrow J_{S}\twoheadrightarrow C_{S} \label{ljbj4}%
\end{equation}
of \cite[Ch. 1, \S 4, Proof of the main theorem]{MR2261462}. The definitions
can be found loc. cit. For later use, we also record the following
isomorphisms: In each later use $M$ will be some Tate twist $\mathbf{Z}%
/p^{k}(j)$ and then $M^{d}=\mathbf{Z}/p^{k}(1-j)$, and $(-)^{\ast
}=\operatorname*{Hom}_{\mathsf{Ab}}(-,\mathbf{Q}/\mathbf{Z})$ is the
non-topological version of the Pontryagin dual (we review its relation to the
genuine Pontryagin dual in Example \ref{ex_PontryaginWithQZCoeffs}).%
\begin{equation}%
\begin{tabular}
[c]{rclll}%
$\operatorname*{Ext}\nolimits_{G_{S}}^{r}\left(  M,C_{S}\right)  $ & $\cong$ &
$H^{2-r}(\mathcal{O}_{S},M)^{\ast}$ & $\qquad$ & for $r\geq1$\\
$\operatorname*{Ext}\nolimits_{G_{S}}^{r}\left(  M,E_{S}\right)  $ & $\cong$ &
$H^{r}(\mathcal{O}_{S},M^{d})$ & $\qquad$ & for $r\geq0$\\
$\operatorname*{Ext}\nolimits_{G_{S}}^{r}\left(  M,J_{S}\right)  $ & $\cong$ &
$P_{S}^{r}(F,M^{d})$ & $\qquad$ & for $r\geq1$%
\end{tabular}
\ \ \ \ \ \ \ \ \ \ \label{ljbj2}%
\end{equation}
The first isomorphism is \cite[Thm. 4.6 (a)]{MR2261462}. The conditions are
met since $M$ is finite, and the condition $p\in P$ (in the notation loc.
cit.) is satisfied since $\frac{1}{p}\in\mathcal{O}_{S}$ (by the
$p$-cyclotomic tower, see \cite[Ch. 1, beginning of \S 4]{MR2261462} for an
explanation). As $M$ is annihilated by $p$, so are both sides of the
isomorphism, so we need not restrict to the $p$-torsion part as in loc. cit.
The second isomorphism is \cite[Lemma 4.12]{MR2261462}. The third isomorphism
is \cite[Lemma 4.13]{MR2261462}, again using that we shall only invoke this
for finite $M$. The restriction to $r\geq1$ in two of these isomorphisms
causes some subtleties. Let us discuss this in detail. For $r=0$ the correct
replacement of the last isomorphism in Eq. \ref{ljbj2} is (also by \cite[Lemma
4.13]{MR2261462})%
\begin{align}
\operatorname*{Ext}\nolimits_{G_{S}}^{0}\left(  M,J_{S}\right)   &  \cong%
\prod_{v\in S}H^{0}(G_{v},M^{d})\label{ljbj2_c}\\
&  \cong P_{S}^{0}(F,M^{d})\oplus\bigoplus_{v\in S_{\infty}}H^{0}(F_{v}%
,M^{d})\text{.}\nonumber
\end{align}
Note that this is precisely the \textquotedblleft anomaly\textquotedblright%
\ which we have already encountered in Lemma
\ref{lemma_specseq_for_restricted_product_inside}, Rmk.
\ref{rmk_TheStoryAboutTateCohomology2} and Diagram \ref{lwui2}. As such, this
apparent \textquotedblleft anomaly\textquotedblright\ actually fits perfectly
into our framework.\newline(Step 3) Using these facts (notably the last two
equations from Eqs. \ref{ljbj2} and the $r=0$ correction in Eq. \ref{ljbj2_c}%
), we may identify Eq. \ref{lwui6} with%
\begin{align}
&  \cong\operatorname*{cofib}\left(  \operatorname*{RHom}\nolimits_{G_{S}%
}\left(  \mathbf{Z}/p(1+\tfrac{j}{2}),E_{S}\right)  \rightarrow
\operatorname*{RHom}\nolimits_{G_{S}}\left(  \mathbf{Z}/p(1+\tfrac{j}%
{2}),J_{S}\right)  \right) \nonumber\\
&  \cong\operatorname*{RHom}\nolimits_{G_{S}}\left(  \mathbf{Z}/p(1+\tfrac
{j}{2}),\operatorname*{cofib}\left(  E_{S}\longrightarrow J_{S}\right)
\right) \label{ljk2}\\
&  \cong\operatorname*{RHom}\nolimits_{G_{S}}\left(  \mathbf{Z}/p(1+\tfrac
{j}{2}),C_{S}\right)  \text{,}\nonumber
\end{align}
where the last isomorphism comes from Eq. \ref{ljbj4}. Summarizing this
computation, we have obtained a convergent spectral sequence%
\begin{align}
&  E_{2}^{i,j}:=\operatorname*{Ext}\nolimits_{G_{S}}^{i}\left(  \mathbf{Z}%
/p(1+\tfrac{j}{2}),C_{S}\right) \label{ljk3}\\
&  \qquad\qquad\Rightarrow\pi_{-i-j}L_{K(1)}(K/p)K(\mathsf{LCA}_{\mathcal{O}%
_{S}})\text{.}\nonumber
\end{align}
This proves Prop. \ref{prop_duality1}.
\end{proof}

\subsection{Arbitrary $S$\label{subsect_ArbitraryS}}

In this subsection $S$ can be any set of places of $F$ as long as $\frac{1}%
{p}\in\mathcal{O}_{S}$. The latter implies that $S$ contains at least one
finite place and thus Def. \ref{def_2} is available. Along with the Galois
modules defined in \S \ref{subsect_SInfPlaces}, we need the following: For
$L/F$ a fixed finite extension, write

$S_{\infty}^{\prime}$ for the infinite places of $L$,

$V_{L,S}:=\left.  \underset{v\in S_{\infty}^{\prime}\setminus(S^{\prime}\cap
S_{\infty}^{\prime})}{\prod}\right.  L_{v}^{\times}$ for the
unwanted\footnote{in the sense that we want to remove them a posteriori, just
as \textsc{LCA}$_{\mathcal{O}_{S}}$ was defined by quotienting the
corresponding archimedean factors away.} infinite place factors of the
$S^{\prime}$-ideles,

and $V_{S}:=\operatorname*{colim}_{L}V_{L,S}$. The choice of places
$S_{\infty}^{\prime}\setminus(S^{\prime}\cap S_{\infty}^{\prime})$ corresponds
to the one in Eq. \ref{lmips9}.

\begin{lemma}
\label{lemma_JSRemovePlaces}We have a commutative diagram%
\begin{equation}%
\xymatrix{
& V_{S} \ar[r]^{\simeq} \ar@{^{(}->}[d] & V_{S} \ar@{^{(}->}[d] \\
E_{S} \ar@{^{(}->}[r] \ar@{=}[d] & J_{S} \ar@{->>}[r] \ar@{->>}[d] & C_{S}
\ar@{->>}[d] \\
E_{S} \ar@{^{(}->}[r] & J_{S}^{\setminus\infty} \ar[r] & C_{S}^{\setminus
\infty}\text{,}
}
\label{l_Diag_RemoveArchimedeanPlacesForGaloisModules}%
\end{equation}
where $J_{S}^{\setminus\infty}$, $C_{S}^{\setminus\infty}$ are
\textit{defined} to be the respective cokernels.
\end{lemma}

\begin{proof}
Straightforward.
\end{proof}

Diag. \ref{l_Diag_RemoveArchimedeanPlacesForGaloisModules} is the Galois
module counterpart of Diagram
\ref{l_Diag_RemoveArchimedeanPlacesCategorically}.

\begin{proposition}
[Dualized spectral sequence]\label{prop_duality2}Let $K$ denote non-connective
$K$-theory. Let $p$ be an odd prime, $F$ a number field, $S$ any (possibly
infinite) set of places such that $\frac{1}{p}\in\mathcal{O}_{S}$. Suppose
$k\geq1$. Then there is a convergent spectral sequence%
\begin{align*}
&  E_{2}^{i,j}:=\operatorname*{Ext}\nolimits_{G_{S}}^{i}\left(  \mathbf{Z}%
/p^{k}(1+\tfrac{j}{2}),C_{S}^{\setminus\infty}\right) \\
&  \qquad\qquad\Rightarrow\pi_{-i-j}L_{K(1)}(K/p^{k})(\left.  \mathsf{LC}%
\mathcal{O}_{S}\right.  )\text{.}%
\end{align*}
with differential $d_{r}^{i,j}\colon E_{r}^{i,j}\rightarrow E_{r}^{i+r,i-r+1}$
of cohomological type. Alternatively, one may phrase its $E_{2}$-page terms
as
\begin{equation}
E_{2}^{i,j}\cong H^{2-i}(\mathcal{O}_{S},\mathbf{Z}/p^{k}(1+\tfrac{j}%
{2}))^{\ast} \label{lmips14}%
\end{equation}
for all $i,j\in\mathbf{Z}$. We shall call this the\emph{ dualized form}.
\end{proposition}

We may spell out the entries of the $E_{2}$-page:%
\begin{equation}%
\begin{tabular}
[c]{rccccc}%
$\vdots\,\,$ & \multicolumn{1}{|c}{} & $\vdots$ & $\vdots$ & $\vdots$ & \\
$\mathsf{2}$ & \multicolumn{1}{|c}{$\quad0\quad$} & $H^{2}(\mathcal{O}%
_{S},\mathbf{Z}/p^{k}(2))^{\ast}$ & $H^{1}(\mathcal{O}_{S},\mathbf{Z}%
/p^{k}(2))^{\ast}$ & $H^{0}(\mathcal{O}_{S},\mathbf{Z}/p^{k}(2))^{\ast}$ &
$\quad0\quad$\\
$\mathsf{1}$ & \multicolumn{1}{|c}{$\quad0\quad$} & $0$ & $0$ & $0$ &
$\quad0\quad$\\
$\mathsf{0}$ & \multicolumn{1}{|c}{$\quad0\quad$} & $H^{2}(\mathcal{O}%
_{S},\mathbf{Z}/p^{k}(1))^{\ast}%
{\cellcolor{lg}}%
$ & $H^{1}(\mathcal{O}_{S},\mathbf{Z}/p^{k}(1))^{\ast}$ & $H^{0}%
(\mathcal{O}_{S},\mathbf{Z}/p^{k}(1))^{\ast}$ & $\quad0\quad$\\
$\mathsf{-1}$ & \multicolumn{1}{|c}{$\quad0\quad$} & $0$ & $0%
{\cellcolor{lg}}%
$ & $0$ & $\quad0\quad$\\
$\mathsf{-2}$ & \multicolumn{1}{|c}{$\quad0\quad$} & $H^{2}(\mathcal{O}%
_{S},\mathbf{Z}/p^{k}(0))^{\ast}$ & $H^{1}(\mathcal{O}_{S},\mathbf{Z}%
/p^{k}(0))^{\ast}$ & $H^{0}(\mathcal{O}_{S},\mathbf{Z}/p^{k}(0))^{\ast}%
{\cellcolor{lg}}%
$ & $\quad0\quad$\\
$\vdots\,\,$ & \multicolumn{1}{|c}{} & $\vdots$ & $\vdots$ & $\vdots$ &
\\\cline{2-6}
& $\mathsf{-1}$ & $\mathsf{0}$ & $\mathsf{1}$ & $\mathsf{2}$ & $\mathsf{\cdots
}$%
\end{tabular}
\ \label{lmips14_b}%
\end{equation}

\begin{proof}
This is entirely analogous to the proof of Prop. \ref{prop_duality1}, since
$\left.  \mathsf{LC}\mathcal{O}_{S}\right.  $ differs from $\mathsf{LCA}%
_{\mathcal{O}_{S}}$ just by quotienting out some factors corresponding to the
infinite places, and the same happens with $C_{S}^{\setminus\infty}$. Hence,
we will only indicate how to modify the proof of Prop. \ref{prop_duality1},
and our changes only really concern Step 2 of that proof.\newline(Step 1) We
proceed as in Step 1 of the previous proof, i.e., we also consider%
\[
\delta\colon K(\mathcal{O}_{S})\longrightarrow\left.  \underset{v\in
S\;}{\prod\nolimits^{\prime}}\right.  \left(  K(F_{v}):K(\mathcal{O}%
_{v})\right)  \text{.}%
\]
Eq. \ref{lwui5} is still true, now using Prop.
\ref{prop_IdentifyAdelicBlocks_2} (1), (4) instead. The limit term of the
resulting spectral sequence can be computed as in Eq. \ref{ljk1}, now relying
on the modified fiber sequence of Prop. \ref{prop_IdentifyAdelicBlocks_2} (3)
instead. Eq. \ref{ljk1}, describing the limit term, becomes%
\begin{align*}
E_{2}^{i,j}  &  :=H^{i}\left(  G_{S},\ldots\right) \\
&  \qquad\Rightarrow\pi_{-i-j}L_{K(1)}(K/p^{k})K(\left.  \mathsf{LC}%
\mathcal{O}_{S}\right.  )
\end{align*}
for the modified category $\left.  \mathsf{LC}\mathcal{O}_{S}\right.  $ of
Def. \ref{def_2}. Everything then goes through verbatim until, inclusively,
Eq. \ref{lwui6}. This isomorphism is valid as stated, however the
identification with $E_{S}\hookrightarrow J_{S}\twoheadrightarrow C_{S}$ of
Eq. \ref{ljbj4} does not quite work because some infinite places might be
missing which \textit{are} present in $J_{S}$ and $C_{S}$.\newline(Step 2,
where things differ a bit) We amend the mismatch at the archimedean places by
Lemma \ref{lemma_JSRemovePlaces}. From the columns of Diag.
\ref{l_Diag_RemoveArchimedeanPlacesForGaloisModules} we get long exact
sequences%
\[
0\rightarrow\operatorname*{Ext}\nolimits_{G_{S}}^{0}(M,V_{S})\rightarrow
\operatorname*{Ext}\nolimits_{G_{S}}^{0}(M,J_{S})\rightarrow
\operatorname*{Ext}\nolimits_{G_{S}}^{0}(M,J_{S}^{\setminus\infty}%
)\rightarrow\operatorname*{Ext}\nolimits_{G_{S}}^{1}(M,V_{S})\rightarrow
\cdots\text{.}%
\]
Since both $\mathbf{C}^{\times}$ and $\mathbf{R}^{\times}$ are $p$-divisible
(since $p$ is odd) and our $M$ is always a finite $p$-primary Galois module,
\cite[Example 0.8]{MR2261462} yields%
\begin{align*}
\operatorname*{Ext}\nolimits_{G_{S}}^{r}(M,V_{S})  &  \cong H^{r}%
(G_{S},\operatorname*{Hom}(M,V_{S}))\qquad\\
&  \cong\bigoplus_{v\in S_{\infty}\setminus(S\cap S_{\infty})}H^{r}%
(G_{S},\operatorname*{I}\nolimits_{G_{v}}^{G_{S}}M^{d})\qquad\text{(for all
}r\geq0\text{).}%
\end{align*}
Since $G_{v}$ for $v\in S_{\infty}$ has \'{e}tale $p$-cohomological dimension
zero, this is zero for $r\geq1$. An analogous computation applies to the
$C_{S}$-counterpart. Thus,%
\[
0\rightarrow\bigoplus_{v\in S_{\infty}\setminus(S\cap S_{\infty})}H^{0}%
(G_{S},\operatorname*{I}\nolimits_{G_{v}}^{G_{S}}M^{d})\rightarrow
\operatorname*{Ext}\nolimits_{G_{S}}^{0}\left(  M,J_{S}\right)  \rightarrow
\operatorname*{Ext}\nolimits_{G_{S}}^{0}(M,J_{S}^{\setminus\infty}%
)\rightarrow0
\]
is exact. Thus, Eq. \ref{ljbj2_c} implies%
\begin{equation}
\operatorname*{Ext}\nolimits_{G_{S}}^{0}(M,J_{S}^{\setminus\infty})\cong
P_{S}^{0}(F,M^{d})\text{.}\nonumber
\end{equation}
For $C_{S}^{\setminus\infty}$ the story is analogous, but a little more
involved: In the computation in \cite[Ch. 1, \S 4, p. 59]{MR2261462} (or
spelled out in a lot more detail in \cite[Lemma 8.6.12]{MR2392026}) the kernel
of the left downward arrow loc. cit. is $\bigoplus_{v\in S_{\infty}%
\setminus(S\cap S_{\infty})}H^{0}(G_{S},\operatorname*{I}\nolimits_{G_{v}%
}^{G_{S}}M^{d})$, so that we get%
\[
0\rightarrow\bigoplus_{v\in S_{\infty}\setminus(S\cap S_{\infty})}H^{0}%
(G_{S},\operatorname*{I}\nolimits_{G_{v}}^{G_{S}}M^{d})\rightarrow
\operatorname*{Ext}\nolimits_{G_{S}}^{0}\left(  M,C_{S}\right)  \rightarrow
\operatorname*{Ext}\nolimits_{G_{S}}^{0}\left(  M,C_{S}^{\setminus\infty
}\right)  \rightarrow0\text{,}%
\]
but by \cite[Thm. 4.6 (b)]{MR2261462}%
\[
\operatorname*{coker}\left(  \operatorname*{Ext}\nolimits_{G_{S}}^{0}\left(
M,D_{S}(L)\right)  \longrightarrow\operatorname*{Ext}\nolimits_{G_{S}}%
^{0}\left(  M,C_{S}\right)  \text{ }\right)  \text{ }\cong H^{2}%
(\mathcal{O}_{S},M)^{\ast}%
\]
(for a suitable extension $L/F$), one can connect both computations to obtain
$\operatorname*{Ext}\nolimits_{G_{S}}^{0}(M,C_{S}^{\setminus\infty})\cong
H^{2}(\mathcal{O}_{S},M)^{\ast}$. Combining this with the vanishing of the
higher $\operatorname*{Ext}$-groups for $V_{S}$ and Eq. \ref{ljbj2}, we get%
\begin{equation}%
\begin{tabular}
[c]{rclll}%
$\operatorname*{Ext}\nolimits_{G_{S}}^{r}(M,C_{S}^{\setminus\infty})$ &
$\cong$ & $H^{2-r}(\mathcal{O}_{S},M)^{\ast}$ & $\qquad$ & for $r\in
\mathbf{Z}$\\
$\operatorname*{Ext}\nolimits_{G_{S}}^{r}(M,E_{S})$ & $\cong$ & $H^{r}%
(\mathcal{O}_{S},M^{d})$ & $\qquad$ & for $r\in\mathbf{Z}$\\
$\operatorname*{Ext}\nolimits_{G_{S}}^{r}(M,J_{S}^{\setminus\infty})$ &
$\cong$ & $P_{S}^{r}(F,M^{d})$ & $\qquad$ & for $r\in\mathbf{Z}$%
\end{tabular}
\ \ \ \ \ \ \ \label{lmips10}%
\end{equation}
along with an exact sequence of $G_{S}$-modules%
\[
E_{S}\hookrightarrow J_{S}^{\setminus\infty}\twoheadrightarrow C_{S}%
^{\setminus\infty}\text{.}%
\]
(Step 3) Now we can proceed as in Step 3 of the proof of Prop.
\ref{prop_duality1}. Eq. \ref{lwui6} unravels as
\begin{align*}
&  \cong\operatorname*{cofib}\left(  \operatorname*{RHom}\nolimits_{G_{S}%
}\left(  \mathbf{Z}/p(1+\tfrac{j}{2}),E_{S}\right)  \rightarrow
\operatorname*{RHom}\nolimits_{G_{S}}\left(  \mathbf{Z}/p(1+\tfrac{j}%
{2}),J_{S}^{\setminus\infty}\right)  \right) \\
&  \cong\operatorname*{RHom}\nolimits_{G_{S}}\left(  \mathbf{Z}/p(1+\tfrac
{j}{2}),\operatorname*{cofib}\left(  E_{S}\longrightarrow J_{S}^{\setminus
\infty}\right)  \right) \\
&  \cong\operatorname*{RHom}\nolimits_{G_{S}}\left(  \mathbf{Z}/p(1+\tfrac
{j}{2}),C_{S}^{\setminus\infty}\right)  \text{,}%
\end{align*}
which is the counterpart of Eq. \ref{ljk2}. As in Eq. \ref{ljk3}, we arrive at
a spectral sequence whose $E_{2}$- and limit page are%
\begin{align*}
&  E_{2}^{i,j}:=\operatorname*{Ext}\nolimits_{G_{S}}^{i}\left(  \mathbf{Z}%
/p^{k}(1+\tfrac{j}{2}),C_{S}^{\setminus\infty}\right) \\
&  \qquad\qquad\Rightarrow\pi_{-i-j}L_{K(1)}(K/p^{k})K(\left.  \mathsf{LC}%
\mathcal{O}_{S}\right.  )\text{.}%
\end{align*}
Our claim follows from using the first isomorphism of Eq. \ref{lmips10} to
rewrite the $E_{2}$-page terms under the duality pairing of the underlying
class formation.
\end{proof}

\subsection{Proof of the Duality Theorem}

As in \S \ref{sect_Appendix_DualityInKOneLocalHptyCat} we write $\mathsf{K}$
for the $K(1)$-local homotopy category at the prime $p$. Its intrinsic tensor
structure arises from $K(1)$-localizing the usual smash product.

\begin{theorem}
[Duality, Main Formulation]\label{thm_DualityMain}Let $K$ denote
non-connective $K$-theory. Let $p$ be an odd prime, $F$ a number field, $S$ a
(possibly infinite) set of finite places of $F$ such that $\frac{1}{p}%
\in\mathcal{O}_{S}$. Then the pairing%
\begin{equation}
L_{K(1)}K(\mathsf{LCA}_{\mathcal{O}_{S}})\otimes_{\mathsf{K}}L_{K(1)}%
K(\mathcal{O}_{S})\longrightarrow L_{K(1)}K(\mathsf{LCA}_{\mathcal{O}_{S}%
})\overset{\operatorname*{Tr}_{F,S}}{\longrightarrow}I_{\mathbf{Z}_{p}%
}\mathbb{S}_{\widehat{p}} \label{lwo1}%
\end{equation}
is left perfect\footnote{We recall our usage of the terms \emph{left/right
perfect} in \S \ref{subsect_PerfectPairings}.}. Its adjoint%
\begin{equation}
L_{K(1)}K(\left.  \mathsf{LC}\mathcal{O}_{S}\right.  )\longrightarrow
I_{\mathbf{Z}_{p}}L_{K(1)}K(\mathcal{O}_{S}) \label{lwo2}%
\end{equation}
is an equivalence. If $S$ contains only finitely many places, the pairing is
left \emph{and} right perfect, and in particular the opposite adjoint%
\begin{equation}
L_{K(1)}K(\mathcal{O}_{S})\longrightarrow I_{\mathbf{Z}_{p}}L_{K(1)}%
K(\mathsf{LCA}_{\mathcal{O}_{S}}) \label{lwo3}%
\end{equation}
is also an equivalence.
\end{theorem}

At present, chromatic homotopy theory for condensed spectra is not available
in the literature. However, once this is established, the above result will
give a left and right perfect pairing of condensed spectra even if $S$
infinite. The failure of reflexivity is a problem of the category
$\mathsf{Ab}$, imported to $\mathsf{Sp}$, and not really an artifact of arithmetic.

\begin{proof}
(Step 1) The bi-exact pairing%
\begin{equation}
\mathsf{LCA}_{\mathcal{O}_{S}}\times\mathsf{Proj}_{\mathcal{O}_{S}%
,fg}\longrightarrow\mathsf{LCA}_{\mathcal{O}_{S}} \label{lmx4rep}%
\end{equation}
of Eq. \ref{lmx4} (\S \ref{subsect_RingStructures}) induces a pairing on
$K$-theory, and using multiplication on the Moore spectrum, also on
$K/p^{k}\cong K\otimes_{\mathsf{Sp}}\mathbb{S}/p^{k}$. The resulting pairing%
\[
(K/p^{k})(\mathsf{LCA}_{\mathcal{O}_{S}})\otimes_{\mathsf{Sp}}(K/p^{k}%
)(\mathcal{O}_{S})\longrightarrow(K/p^{k})(\mathsf{LCA}_{\mathcal{O}_{S}})
\]
then gives rise to a $K(1)$-local pairing%
\[
L_{K(1)}(K/p^{k})(\mathsf{LCA}_{\mathcal{O}_{S}})\otimes_{\mathsf{K}}%
L_{K(1)}(K/p^{k})(\mathcal{O}_{S})\longrightarrow L_{K(1)}(K/p^{k}%
)(\mathsf{LCA}_{\mathcal{O}_{S}})
\]
and

\begin{itemize}
\item $E_{\bullet}^{\bullet,\bullet}(\mathsf{LCA}_{\mathcal{O}_{S}})$, the
descent spectral sequences for $\mathsf{LCA}_{\mathcal{O}_{S}}$ supplied by
Prop. \ref{prop_duality1}, as well as

\item $E_{\bullet}^{\bullet,\bullet}(\mathcal{O}_{S})$, the ordinary descent
spectral sequence as in
\S \ref{sect_IndividualThomasonDescentSpectralSequences},
\end{itemize}

acquire an induced multiplicative structure. We want to unravel this pairing
of spectral sequences%
\begin{equation}
E_{r}^{i,j}(\mathsf{LCA}_{\mathcal{O}_{S}})\otimes_{\mathbf{Z}}E_{r}%
^{i^{\prime},j^{\prime}}(\mathcal{O}_{S})\longrightarrow E_{r}^{i+i^{\prime
},j+j^{\prime}}(\mathsf{LCA}_{\mathcal{O}_{S}})\text{.} \label{lmips22}%
\end{equation}
In order to compute what the multiplication does on $K$-theory, note that our
pairing is compatible with the square%
\begin{equation}%
\xymatrix{
\mathsf{Proj}_{\mathcal{O}_{S},fg} \times\mathsf{Proj}_{\mathcal{O}_{S},fg}
\ar[d] \ar[r]^-{\otimes} & \mathsf{Proj}_{\mathcal{O}_{S},fg} \ar[d] \\
\mathsf{LCA}_{\mathcal{O}_{S},ad} \times\mathsf{Proj}_{\mathcal{O}_{S},fg}
\ar[r]_-{\otimes} & \mathsf{LCA}_{\mathcal{O}_{S},ad}.
}
\label{lmips23}%
\end{equation}
This renders $K(\mathsf{LCA}_{\mathcal{O}_{S},ad})$ a $K(\mathcal{O}_{S}%
)$-algebra and by Thm. \ref{thm_main_strong} we have%
\[
K(\mathsf{LCA}_{\mathcal{O}_{S}})\cong\operatorname*{cofib}\left(
K(\mathcal{O}_{S})\rightarrow K(\mathsf{LCA}_{\mathcal{O}_{S},ad})\right)
\text{.}%
\]
Thus, we may compute the pairing on the level of%
\[
K(\mathsf{LCA}_{\mathcal{O}_{S},ad})\otimes K(\mathcal{O}_{S})\longrightarrow
K(\mathsf{LCA}_{\mathcal{O}_{S},ad})\text{.}%
\]
But $\mathsf{LCA}_{\mathcal{O}_{S},ad}$ can, under Prop.
\ref{prop_IdentifyAdelicBlocks}, be identified with a restricted product
category. Using this identification, our pairing reduces to the usual tensor
product%
\[
\mathsf{Proj}_{F_{v,}fg}\otimes\mathsf{Proj}_{\mathcal{O}_{S},fg}%
\longrightarrow\mathsf{Proj}_{F_{v,}fg}\qquad\text{resp.}\qquad\mathsf{Proj}%
_{\mathcal{O}_{v,}fg}\otimes\mathsf{Proj}_{\mathcal{O}_{S},fg}\longrightarrow
\mathsf{Proj}_{\mathcal{O}_{v,}fg}\text{.}%
\]
The descent spectral sequences $E_{\bullet}^{\bullet,\bullet}(\mathsf{LCA}%
_{\mathcal{O}_{S}})$ and $E_{\bullet}^{\bullet,\bullet}(\mathcal{O}_{S})$ both
come from filtering the respective $K$-theory spectra. Since for the usual
tensor product it is known that the pairing on the spectral sequence is
compatible with the cup product, we conclude that the same is true in our
setting. We shall denote the cup product by \textquotedblleft$\smile
$\textquotedblright\ below. It will be useful to rewrite these cup products in
terms of $\operatorname*{Ext}$-groups:

Let $a,b,u,v\in\mathbf{Z}$ be arbitrary. Recall that for each pairing%
\[
\mathbf{Z}/p^{k}\left(  u\right)  \otimes\mathbf{Z}/p^{k}\left(  v\right)
\longrightarrow\mathbf{Z}/p^{k}\left(  u+v\right)
\]
the adjoint%
\begin{equation}
\mathbf{Z}/p^{k}\left(  u\right)  \longrightarrow\operatorname*{Hom}\left(
\mathbf{Z}/p^{k}\left(  v\right)  ,\mathbf{Z}/p^{k}\left(  u+v\right)
\right)  \tag{$\ast$}%
\end{equation}
induces a commutative square\footnote{Recall that $\operatorname*{Hom}$-groups
compose in the order $\operatorname*{Hom}(B,C)\times\operatorname*{Hom}%
(A,B)\rightarrow\operatorname*{Hom}(A,C)$.}%
\[%
\xymatrix{
H^{a}\left(  -,\mathbf{Z}/p^{k}\left(  u\right)  \right) \otimes H^{b}%
\left(  -,\mathbf{Z}/p^{k}\left(  v\right)  \right)
\ar[r]^-{\smile} \ar[d]^{\ast\otimes1} &
H^{a+b}\left(  -,\mathbf{Z}/p^{k}\left(  u+v\right)  \right) \ar@{=}[dd] \\
\operatorname*{Ext}^{a}(\mathbf{Z},\operatorname*{Hom}\left(  \mathbf{Z}%
/p^{k}\left(  v\right)  ,\mathbf{Z}/p^{k}\left(  u+v\right)  \right))
\otimes H^{b}\left(  -,\mathbf{Z}/p^{k}\left(  v\right)  \right) \ar[d] \\
\operatorname*{Ext}^{a}(\mathbf{Z}/p^{k}\left(  v\right)  ,\mathbf{Z}%
/p^{k}\left(  u+v\right)  ) \otimes\operatorname*{Ext}^{b}\left(
\mathbf{Z},\mathbf{Z}/p^{k}\left(  v\right)  \right) \ar[r] &
\operatorname*{Ext}^{a+b}\left(  \mathbf{Z},\mathbf{Z}/p^{k}\left(  u+v\right)
\right)
}%
\]
expressing the cup product on Galois cohomology through the Yoneda product on
$\operatorname*{Ext}$-groups (see \cite[Prop. 16.17]{MR4174395} or \cite[Ch.
I, Prop. 0.14]{MR2261462}).

Applying this to the pairing of Eq. \ref{lmips22}, we obtain%
\begin{align}
&  \operatorname*{Ext}\nolimits_{G_{S}}^{i}\left(  \mathbf{Z}/p^{k}%
(1+\tfrac{j}{2}),C_{S}\right)  \otimes\operatorname*{Ext}\nolimits_{G_{S}%
}^{i^{\prime}}\left(  \mathbf{Z},\mathbf{Z}/p^{k}(-\tfrac{j^{\prime}}%
{2})\right) \label{lmips18}\\
&  \qquad\qquad\qquad\qquad\longrightarrow\left\{
\begin{array}
[c]{ll}%
\operatorname*{Ext}\nolimits_{G_{S}}^{i+i^{\prime}}\left(  \mathbf{Z}%
,C_{S}\right)  & \text{for }j+j^{\prime}=-2\\
0 & \text{otherwise.}%
\end{array}
\right. \nonumber
\end{align}
(Step 2) The above story descends to the quotient category $\left.
\mathsf{LC}\mathcal{O}_{S}\right.  $: The bi-exact pairing of Eq.
\ref{lmx4rep} induces a bi-exact pairing%
\begin{equation}
\left.  \mathsf{LC}\mathcal{O}_{S}\right.  \times\mathsf{Proj}_{\mathcal{O}%
_{S},fg}\longrightarrow\left.  \mathsf{LC}\mathcal{O}_{S}\right.
\label{lmx4repB}%
\end{equation}
since a real vector space, tensored with a finite rank free $\mathcal{O}_{S}%
$-module, again carries a real vector space topology. More specifically, it
follows that $F_{\left\langle S\right\rangle }$-modules (equipped with the
real topology) are sent to $F_{\left\langle S\right\rangle }$-modules. The
pairing now has the format%
\begin{equation}
K(\left.  \mathsf{LC}\mathcal{O}_{S}\right.  )\otimes_{\mathsf{Sp}%
}K(\mathcal{O}_{S})\longrightarrow K(\left.  \mathsf{LC}\mathcal{O}%
_{S}\right.  )\text{.} \label{lmips20}%
\end{equation}
One gets an analogous spectral sequence pairing, with $C_{S}^{\setminus\infty
}$ instead of $C_{S}$. All one has to do is checking that all arguments in
Step 1 are compatible with quotienting out $\mathsf{Proj}_{F_{\left\langle
S\right\rangle },fg}$, and replace the spectral sequence $E_{\bullet}%
^{\bullet,\bullet}(\mathsf{LCA}_{\mathcal{O}_{S}})$ of Prop.
\ref{prop_duality1} by $E_{\bullet}^{\bullet,\bullet}(\left.  \mathsf{LC}%
\mathcal{O}_{S}\right.  )$ of Prop. \ref{prop_duality2}. We will omit the
details, as the proof of Prop. \ref{prop_duality2} already took the time to
explain how to implement this switch. Instead of Eq. \ref{lmips18}, we now get%
\begin{align}
&  \underline{\operatorname*{Ext}\nolimits_{G_{S}}^{i}\left(  \mathbf{Z}%
/p^{k}(1+\tfrac{j}{2}),C_{S}^{\setminus\infty}\right)  }\otimes
\operatorname*{Ext}\nolimits_{G_{S}}^{i^{\prime}}\left(  \mathbf{Z}%
,\mathbf{Z}/p^{k}(-\tfrac{j^{\prime}}{2})\right) \nonumber\\
&  \qquad\qquad\qquad\qquad\longrightarrow\left\{
\begin{array}
[c]{ll}%
\underline{\operatorname*{Ext}\nolimits_{G_{S}}^{i+i^{\prime}}\left(
\mathbf{Z},C_{S}^{\setminus\infty}\right)  } & \text{for }j+j^{\prime}=-2\\
0 & \text{otherwise.}%
\end{array}
\right.  \label{lmips19}%
\end{align}
Ignore the underlines for the moment, we only added them to refer to these
terms later. The situation $i+i^{\prime}=2$ is special: then the pairing of
Eq. \ref{lmips18} takes values in $\operatorname*{Ext}\nolimits_{G_{S}}%
^{2}(\mathbf{Z},C_{S})$, which under the invariant map of the underlying class
formation%
\begin{equation}
\operatorname*{inv}\colon\operatorname*{Ext}\nolimits_{G_{S}}^{2}\left(
\mathbf{Z},C_{S}\right)  \cong H^{2}(\mathcal{O}_{S},C_{S})\cong%
\mathbf{Q}/\mathbf{Z}\text{,} \label{lwcz1}%
\end{equation}
is the target of the Poitou--Tate dualizing pairing. This carries over to
$C_{S}^{\setminus\infty}$ without change: We have $\operatorname*{Ext}%
\nolimits_{G_{S}}^{2}(M,C_{S}^{\setminus\infty})\cong\operatorname*{Ext}%
\nolimits_{G_{S}}^{2}\left(  M,C_{S}\right)  $. This comes from the vanishing
of the higher Galois cohomology of $V_{S}$, as was discussed in the proof of
Prop. \ref{prop_duality2}. Restricting to $i+i^{\prime}=2$, we therefore get a
pairing%
\begin{align}
&  \underline{\operatorname*{Ext}\nolimits_{G_{S}}^{i}\left(  \mathbf{Z}%
/p^{k}(1+\tfrac{j}{2}),C_{S}^{\setminus\infty}\right)  }\otimes
\operatorname*{Ext}\nolimits_{G_{S}}^{i^{\prime}}\left(  \mathbf{Z}%
,\mathbf{Z}/p^{k}(-\tfrac{j^{\prime}}{2})\right) \nonumber\\
&  \qquad\qquad\qquad\qquad\longrightarrow\left\{
\begin{array}
[c]{ll}%
\underline{\operatorname*{Ext}\nolimits_{G_{S}}^{2}\left(  \mathbf{Z}%
,C_{S}^{\setminus\infty}\right)  }\cong\mathbf{Q}/\mathbf{Z} & \text{for
}i+i^{\prime}=2\text{ and }j+j^{\prime}=-2\\
0 & \text{otherwise.}%
\end{array}
\right.  \label{lmips19_pairing_form}%
\end{align}
(Step 3) The next step is best explained in a toy model: Suppose $p\colon
A\times B\longrightarrow\mathbf{Q}/\mathbf{Z}$ is a pairing of abelian groups.
If it is left perfect, i.e., its adjoint $A\rightarrow B^{\ast}%
:=\operatorname*{Hom}(B,\mathbf{Q}/\mathbf{Z})$ is an equivalence, we may
isomorphically replace $A$ by its dual. This transforms the pairing $p$ into
the evaluation pairing.%
\begin{equation}%
\xymatrix{
A \times B \ar[r]^-{p} \ar[d] & \mathbf{Q}/\mathbf{Z} \ar@{=}[d] \\
B^{\ast} \times B \ar[r]_-{\operatorname{ev}} & \mathbf{Q}/\mathbf{Z}
}
\label{lmips24}%
\end{equation}
For the present proof, we will apply this mechanism to the duality pairing of
the underlying class formation of global class field theory. The pairing is%
\begin{equation}
\operatorname*{Ext}\nolimits_{G_{S}}^{i}(M,C_{S}^{\setminus\infty})\otimes
H^{2-i}(\mathcal{O}_{S},M)\longrightarrow\operatorname*{Ext}\nolimits_{G_{S}%
}^{2}(M,C_{S}^{\setminus\infty})\overset{\sim}{\longleftarrow}%
\operatorname*{Ext}\nolimits_{G_{S}}^{2}(M,C_{S})\overset{\sim}{\underset
{\operatorname*{inv}}{\longrightarrow}}\mathbf{Q}/\mathbf{Z}\text{,}
\label{lmips25}%
\end{equation}
which is left perfect and induces the isomorphism%
\[
\Upsilon\colon\operatorname*{Ext}\nolimits_{G_{S}}^{i}(M,C_{S}^{\setminus
\infty})\overset{\sim}{\longrightarrow}H^{2-i}(\mathcal{O}_{S},M)^{\ast}%
\qquad\quad\text{for all }i\in\mathbf{Z}\text{.}%
\]
See Milne \cite[Ch. I, \S 1, The main theorem, around Thm. 1.8]{MR2261462} or
Harari \cite[\S 16.3, beginning]{MR4174395}. We have already used $\Upsilon$
before: It is the isomorphism providing the \textit{dualized form} of Eq.
\ref{lmips14} for our spectral sequence $E_{\bullet}^{\bullet,\bullet}(\left.
\mathsf{LC}\mathcal{O}_{S}\right.  )$, and is the same isomorphism used in Eq.
\ref{lmips10} (and even before in Eq. \ref{ljbj2}). Now feed the pairing of
Eq. \ref{lmips25} into the mechanism of Eq. \ref{lmips24}. Note that thanks to
restricting to $i+i^{\prime}=2$ and $j+j^{\prime}=-2$, this is just an
instance of Eq. \ref{lmips25}, applied to the special case $M:=\mathbf{Z}%
/p^{k}(1+\tfrac{j}{2})$. Thus, Eq. \ref{lmips24} becomes%
\[%
\xymatrix{
{\operatorname*{Ext}\nolimits_{G_{S}}^{i}\left(  \mathbf{Z}/p^{k}(1+\tfrac
{j}{2}),C_{S}^{\setminus\infty}\right)} \otimes{\operatorname*{Ext}%
\nolimits_{G_{S}}^{i^{\prime}}\left(  \mathbf{Z},\mathbf{Z}/p^{k}%
(-\tfrac{j^{\prime}}{2})\right)} \ar[r] \ar[d]^{\Upsilon\otimes1} &
{\operatorname*{Ext}\nolimits_{G_{S}}^{2}\left(  \mathbf{Z},C_{S}%
^{\setminus\infty}\right)} \ar[r]^-{\sim} \ar[d]^*[@]{\sim} &
\mathbf{Q}/\mathbf{Z} \ar@{=}[d] \\
{H^{2-i}(\mathcal{O}_{S},\mathbf{Z}/p^{k}(1+\tfrac{j}{2}))^{\ast}}
\otimes{\operatorname*{Ext}\nolimits_{G_{S}}^{i^{\prime}}\left(  \mathbf
{Z},\mathbf{Z}/p^{k}(-\tfrac{j^{\prime}}{2})\right)} \ar[r]_-{\operatorname
{ev}} &  {H^{2-(i+i^{\prime})}(\mathcal{O}_{S},\mathbf{Z}/p^{k})^{\ast}}
\ar[r]_-{\sim} & \mathbf{Q}/\mathbf{Z}.
}%
\]
We may therefore continue with the pairing in the lower row. It is zero unless
$i+i^{\prime}=2$ and $j+j^{\prime}=-2$, and in this case becomes the
evaluation pairing%
\begin{align}
&  H^{2-i}(\mathcal{O}_{S},\mathbf{Z}/p^{k}(1+\tfrac{j}{2}))^{\ast}%
\otimes\operatorname*{Ext}\nolimits_{G_{S}}^{2-i}\left(  \mathbf{Z}%
,\mathbf{Z}/p^{k}(1+\tfrac{j}{2})\right) \nonumber\\
&  \qquad\qquad\qquad\qquad\underset{\operatorname*{ev}}{\longrightarrow}%
H^{0}(\mathcal{O}_{S},\mathbf{Z}/p^{k})^{\ast}\cong\mathbf{Z}/p^{k}\text{.}
\label{limps26b}%
\end{align}
By construction, this map is a section to the canonical inclusion of
$H^{0}(\mathcal{O}_{S},\mathbf{Z}/p^{k})^{\ast}\cong\mathbf{Z}/p^{k}$ into
$\pi_{0}(K/p^{k})(\left.  \mathsf{LC}\mathcal{O}_{S}\right.  )$. So far, the
entire discussion has concerned the pairing of Eq. \ref{lmips20}. Next, we
postcompose the $K(1)$-localized form of this pairing with the trace map of
\S \ref{sect_TraceMap}.%
\[
L_{K(1)}K(\left.  \mathsf{LC}\mathcal{O}_{S}\right.  )\otimes_{\mathsf{K}%
}L_{K(1)}K(\mathcal{O}_{S})\longrightarrow L_{K(1)}K(\left.  \mathsf{LC}%
\mathcal{O}_{S}\right.  )\overset{\operatorname*{Tr}\nolimits_{F,S}%
}{\longrightarrow}I_{\mathbf{Z}_{p}}\mathbb{S}_{\widehat{p}}\text{.}%
\]
Modulo $p^{k}$ we get the corresponding pairing with the finite level version
of the trace (Eq. \ref{l_FiniteLevelVersionOfTraceMap}), i.e.,%
\begin{equation}
L_{K(1)}(K/p^{k})(\left.  \mathsf{LC}\mathcal{O}_{S}\right.  )\otimes
_{\mathsf{Sp}}L_{K(1)}(K/p^{k})(\mathcal{O}_{S})\rightarrow L_{K(1)}%
(K/p^{k})(\left.  \mathsf{LC}\mathcal{O}_{S}\right.  )\overset
{\operatorname*{Tr}\nolimits_{F,S}/p^{k}}{\longrightarrow}I_{\mathbf{Q}%
_{p}/\mathbf{Z}_{p}}\mathbb{S}/p^{k}\text{.} \label{lmips26}%
\end{equation}
By construction, on $\pi_{0}$ the rightmost map in Eq. \ref{lmips26} amounts,
by construction of the trace map, to the map $V_{F,S,k}$ in%
\begin{equation}%
\xymatrix{
H^{0}(\mathcal{O}_{S},\mathbf{Z}/p^{k})^{\ast} \ar@{^{(}->}[r]^-{E_{F,S,k}}
\ar[dr]_*[@]{\sim} & \pi_{0}L_{K(1)}(K/p^{k})(\left.  \mathsf{LC}\mathcal
{O}_{S}\right.  ) \ar[d]^{V_{F,S,k}=\pi_{0}({\operatorname*{Tr}\nolimits
_{F,S}/p^{k}})} \\
& {\frac{1}{p^{k}}\mathbf{Z}/\mathbf{Z}},
}
\label{lmips28}%
\end{equation}
where%
\[
E_{F,S,k}\colon H^{0}(\mathcal{O}_{S},\mathbf{Z}/p^{k})^{\ast}\hookrightarrow
\pi_{0}L_{K(1)}(K/p^{k})(\left.  \mathsf{LC}\mathcal{O}_{S}\right.  )
\]
is the natural inclusion coming from the spectral sequence of $E_{\bullet
}^{\bullet,\bullet}(\left.  \mathsf{LC}\mathcal{O}_{S}\right.  )$. It agrees,
as the name suggests, which the explicit construction of Eq. \ref{l_EMap}. The
commutativity of Diag. \ref{lmips28} was proven in Prop.
\ref{prop_VIsASectionOfE} (1), and part (2) of the same result shows that
$\pi_{0}(\operatorname*{Tr}\nolimits_{F,S}/p^{k})=V_{F,S,k}$.\ We conclude
that the pairing in Eq. \ref{lmips26} induces on the homotopy groups
$\pi_{\ast}$ a graded pairing
\[
\pi_{r}\otimes\pi_{s}\longrightarrow\pi_{r+s}\text{.}%
\]
For $r+s\neq0$ this pairing is zero and for $r+s=0$ it is a left perfect
pairing, or, in other words, the adjoint%
\begin{equation}
L_{K(1)}(K/p^{k})(\left.  \mathsf{LC}\mathcal{O}_{S}\right.  )\longrightarrow
F_{\mathsf{Sp}}\left(  L_{K(1)}(K/p^{k})(\mathcal{O}_{S}),I_{\mathbf{Q}%
_{p}/\mathbf{Z}_{p}}\mathbb{S}/p^{k}\right)  \cong I_{\mathbf{Q}%
_{p}/\mathbf{Z}_{p}}L_{K(1)}(K/p^{k})(\mathcal{O}_{S})
\label{lmips_Step3_Conclusion}%
\end{equation}
is an equivalence.\newline(Step 4) In Step 3 we have proven that Eq.
\ref{lmips26} is a left perfect pairing for all $k\geq1$. It remains to show
that our claim is true without going modulo $p^{k}$. This is harmless as our
spectra are $p$-complete (see Rmk. \ref{rmk_CanTestPerfectPairingModP}). In
detail: The adjoint of the pairing in Eq. \ref{lmips20} is
\begin{align*}
L_{K(1)}K(\left.  \mathsf{LC}\mathcal{O}_{S}\right.  )  &  \longrightarrow
F_{\mathsf{Sp}}\left(  L_{K(1)}K(\mathcal{O}_{S}),L_{K(1)}K(\left.
\mathsf{LC}\mathcal{O}_{S}\right.  )\right) \\
&  \longrightarrow F_{\mathsf{Sp}}\left(  L_{K(1)}K(\mathcal{O}_{S}%
),I_{\mathbf{Z}_{p}}\mathbb{S}_{\widehat{p}}\right)  \cong I_{\mathbf{Z}_{p}%
}L_{K(1)}K(\mathcal{O}_{S})\text{.}%
\end{align*}
We need to check that this map is an equivalence. We find that%
\[
L_{K(1)}K(\left.  \mathsf{LC}\mathcal{O}_{S}\right.  )\longrightarrow
I_{\mathbf{Z}_{p}}L_{K(1)}K(\mathcal{O}_{S})
\]
is by Eq. \ref{lrrb1} equivalent to%
\[
\Sigma L_{K(1)}K(\left.  \mathsf{LC}\mathcal{O}_{S}\right.  )\longrightarrow
I_{\mathbf{Q}_{p}/\mathbf{Z}_{p}}L_{K(1)}K(\mathcal{O}_{S})
\]
being an equivalence. Since both sides are $p$-complete, it suffices to check
that this map is an equivalence $\operatorname{mod}p$, i.e., it suffices to
show that%
\[
\left(  \Sigma L_{K(1)}K(\left.  \mathsf{LC}\mathcal{O}_{S}\right.  )\right)
/p\longrightarrow\left(  I_{\mathbf{Q}_{p}/\mathbf{Z}_{p}}L_{K(1)}%
K(\mathcal{O}_{S})\right)  /p
\]
is an equivalence. From $X\overset{\cdot p}{\rightarrow}X\rightarrow X/p$ we
get $I_{\mathbf{Q}_{p}/\mathbf{Z}_{p}}X\overset{\cdot p}{\leftarrow
}I_{\mathbf{Q}_{p}/\mathbf{Z}_{p}}X\leftarrow I_{\mathbf{Q}_{p}/\mathbf{Z}%
_{p}}(X/p)$ and therefore $(I_{\mathbf{Q}_{p}/\mathbf{Z}_{p}}X)/p\cong%
\Sigma(I_{\mathbf{Q}_{p}/\mathbf{Z}_{p}}(X/p))$. Hence,%
\[
\Sigma L_{K(1)}(K/p)(\left.  \mathsf{LC}\mathcal{O}_{S}\right.
)\longrightarrow\Sigma\left(  I_{\mathbf{Q}_{p}/\mathbf{Z}_{p}}L_{K(1)}%
(K/p)(\mathcal{O}_{S})\right)  \text{.}%
\]
Thus, it remains to check that%
\[
L_{K(1)}(K/p)(\left.  \mathsf{LC}\mathcal{O}_{S}\right.  )\longrightarrow
I_{\mathbf{Q}_{p}/\mathbf{Z}_{p}}L_{K(1)}(K/p)(\mathcal{O}_{S})
\]
is an equivalence. However, this was already settled in Step 3, notably Eq.
\ref{lmips_Step3_Conclusion}. This proves that the pairing in Eq. \ref{lwo1}
is left perfect, and equivalently that Eq. \ref{lwo2} is an equivalence. If
$S$ is finite, all Galois cohomology groups on the $E_{2}$-page of the
spectral sequence of Prop. \ref{prop_duality2} are finite. Thus, so are the
homotopy groups of $L_{K(1)}(K/p^{k})(\left.  \mathsf{LC}\mathcal{O}%
_{S}\right.  )$. Just by finiteness, these spectra are reflexive for
Brown--Comenetz duality. Hence, the pairing in Eq. \ref{lmips26} will
additionally be right perfect and an analogue of Eq.
\ref{lmips_Step3_Conclusion} holds for the opposite adjoint. The same argument
as in Step 4 then settles the opposite duality, i.e., Eq. \ref{lwo3}. This
finishes the proof.
\end{proof}

\subsection{Consequences \&\ Complements}

We may immediately deduce the simplified formulation from the introduction.
Let $p$ be an odd prime, $F$ a number field, $S$ a (possibly infinite) set of
finite places of $F$ such that $\frac{1}{p}\in\mathcal{O}_{S}$.

\begin{proof}
[Proof of Thm. \ref{thmann_ThmA}]In order to avoid having to set up $\left.
\mathsf{LC}\mathcal{O}_{S}\right.  $ in full generality already in the
introduction, we briefly defined $\mathsf{LCA}_{\mathcal{O}_{S}}^{\circ}$ as
the quotient of $\mathsf{LCA}_{\mathcal{O}_{S}}$ modulo those $\mathcal{O}%
_{S}$-modules whose underlying LCA\ group is a real vector space.\ The latter
is the same as the category of all vector $\mathcal{O}_{S}$-modules and thus
$\mathsf{LCA}_{\mathcal{O}_{S}}^{\circ}=_{def}\left.  \mathsf{LC}%
\mathcal{O}_{S}\right.  $ since by assumption $S$ contains only finite places.
\end{proof}

\subsubsection{Transfers\label{subsect_Transfers}}

Now let $S$ be any (possibly infinite) set of places of $F$, but such that
$\frac{1}{p}\in\mathcal{O}_{F,S}$. If $F^{\prime}/F$ is a finite extension,
let $S(F^{\prime})$ denote the places of $F^{\prime}$ lying above $S$. Let
$S^{\prime}$ be any set of places of $F^{\prime}$ such that $S^{\prime}$
contains $S(F^{\prime})$.

\begin{definition}
\label{def_Transfers}We denote the extension by $f\colon\operatorname*{Spec}%
F^{\prime}\rightarrow\operatorname*{Spec}F$.

\begin{enumerate}
\item Define the \emph{pushforward}%
\[
f_{\ast}\colon\left.  \mathsf{LC}\mathcal{O}_{F^{\prime},S^{\prime}}\right.
\longrightarrow\left.  \mathsf{LC}\mathcal{O}_{F,S}\right.
\]
on the level of $\mathsf{LCA}_{\mathcal{O}_{F^{\prime},S^{\prime}}}$ as
forgetting the locally compact $\mathcal{O}_{F^{\prime},S^{\prime}}$-module
structure in favour of the $\mathcal{O}_{F^{\prime},S^{\prime}}$-module structure.

\item Suppose $S^{\prime}=S(F^{\prime})$. Define the \emph{pullback}%
\begin{align*}
f^{\ast}\colon\left.  \mathsf{LC}\mathcal{O}_{F,S}\right.   &  \longrightarrow
\left.  \mathsf{LC}\mathcal{O}_{F^{\prime},S^{\prime}}\right. \\
M  &  \longmapsto M\otimes_{\mathcal{O}_{F,S}}\mathcal{O}_{F^{\prime
},S^{\prime}}%
\end{align*}
on the level of $\mathsf{LCA}_{\mathcal{O}_{F,S}}$ by the topological tensor
product of \cite{MR0215016}. Without relying on this tensor product, it can
equivalently be defined as%
\begin{equation}
f^{\ast}M:=\operatorname*{Hom}\nolimits_{\mathsf{LCA}_{\mathcal{O}_{F^{\prime
},S^{\prime}}}}(\mathcal{O}_{F^{\prime}},M^{\vee})^{\vee}\text{,}
\label{lmips32}%
\end{equation}
where the $\operatorname*{Hom}$-group is equipped with the compact-open topology.
\end{enumerate}

Both functors are exact and induce maps on $K$-theory, which we will also call
the pushforward and pullback.
\end{definition}

The pushforward sends $F_{\left\langle S^{\prime}\right\rangle }^{\prime}%
$-modules to $F_{\left\langle S\right\rangle }$-modules, so even though we
have only really given the definition for $\mathsf{LCA}_{\mathcal{O}%
_{F^{\prime},S^{\prime}}}$, it follows that it remains well defined on the
quotient $\left.  \mathsf{LC}\mathcal{O}_{F^{\prime},S^{\prime}}\right.  $.
The same applies to the pullback.

The definition of the pushforward also makes sense for infinite extensions
$F^{\prime}/F$. As this leaves the scope of this text, we will not discuss
this further.

\begin{proof}
Only the pullback requires a discussion. The Hom-tensor adjunction yields%
\begin{align*}
\mathcal{O}_{F^{\prime},S^{\prime}}\otimes_{\mathcal{O}_{F,S}}M  &
\cong\mathcal{O}_{F^{\prime}}\otimes_{\mathcal{O}_{F}}M\\
&  \cong\operatorname*{Hom}\nolimits_{\mathsf{LCA}_{\mathcal{O}_{F^{\prime}}}%
}(\mathcal{O}_{F^{\prime}}\otimes_{\mathcal{O}_{F}}M,\mathbb{T})^{\vee}%
\cong\operatorname*{Hom}\nolimits_{\mathsf{LCA}_{\mathcal{O}_{F}}}%
(\mathcal{O}_{F^{\prime}},M^{\vee})^{\vee}\text{.}%
\end{align*}
Since $\mathcal{O}_{F^{\prime}}$ is a finitely generated discrete group, the
group $\operatorname*{Hom}\nolimits_{\mathsf{LCA}_{\mathcal{O}_{F}}%
}(\mathcal{O}_{F^{\prime}},M^{\vee})$, equipped with the compact-open
topology, is a locally compact $\mathcal{O}_{F^{\prime},S^{\prime}}$-module
(see \cite[Thm. 4.3 ($2^{\prime}$)]{MR0215016}). In particular, so is its
Pontryagin dual. This proves that Eq. \ref{lmips32} gives an alternative
construction of the pullback. It also settles that the relevant tensor product
is indeed locally compact. We remind the reader that the tensor product of
\cite{MR0215016} in general outputs a topological module which has no reason
to be locally compact.
\end{proof}

\begin{example}
If $\mathcal{O}_{F^{\prime}}$ is a free $\mathcal{O}_{F}$-module (which in
general need not be the case), one can take any $\mathcal{O}_{F}$-module
isomorphism $\alpha\colon\mathcal{O}_{F^{\prime}}\simeq\mathcal{O}%
_{F}^{[F^{\prime}:F]}$ and equip%
\[
\mathcal{O}_{F^{\prime},S^{\prime}}\otimes_{\mathcal{O}_{F,S}}M\cong%
\mathcal{O}_{F^{\prime}}\otimes_{\mathcal{O}_{F}}M\overset{\alpha}{\simeq
}M^{[F^{\prime}:F]}%
\]
with the product topology. Since $\mathcal{O}_{F^{\prime}}$ is discrete, this
agrees with the previous definition. It is easy to see that the resulting
topology is independent of the choice of $\alpha$. If $\mathcal{O}_{F^{\prime
}}$ fails to be free over $\mathcal{O}_{F}$, one can still use this technique,
applied to a free cover. However, it gets increasingly cumbersome to show that
this yields a well-defined topology.
\end{example}

\begin{lemma}
\label{lemma_TransferCompat}Suppose we are in the setting described at the
beginning of \S \ref{subsect_Transfers}. Consider the fiber sequence of Thm.
\ref{thm_main_strong} for both $F$ and $F^{\prime}$. Then the subdiagrams of%
\[%
\xymatrix{
K(\mathcal{O}_{F^{\prime},S^{\prime}}) \ar@<-2ex>[d]_{f_{\ast}} \ar[r] &
{\left.  \underset{v^{\prime}\in S^{\prime}\;}{\prod\nolimits^{\prime}}%
\right.  \left(  K(F_{v^{\prime}}^{\prime}):K(\mathcal{O}_{F^{\prime
},v^{\prime}})\right)} \ar@<-2ex>[d]_{f_{\ast}} \ar[r] &
K(\mathsf{LCA}_{\mathcal{O}_{F^{\prime},S^{\prime}}}) \ar@<-2ex>[d]_{f_{\ast}}
\\
K(\mathcal{O}_{F,S}) \ar@<-2ex>[u]_{f^{\ast}} \ar[r] &
{\left.  \underset{v\in S\;}{\prod\nolimits^{\prime}}\right.  \left
(  K(F_{v}):K(\mathcal{O}_{v})\right)} \ar@<-2ex>[u]_{f^{\ast}} \ar[r] &
K(\mathsf{LCA}_{\mathcal{O}_{F,S}}) \ar@<-2ex>[u]_{f^{\ast}}
}%
\]
consisting only of (horizontal and upward) or (horizontal and downward)
arrows, commute. For the left and middle term, we define $f^{\ast}$ and
$f_{\ast}$ as usual for $K$-theory, on the right term we employ Def.
\ref{def_Transfers}.
\end{lemma}

\begin{proof}
For $f_{\ast}$ this is immediate. For $f^{\ast}$ one needs to compute the
pullback of Eq. \ref{lmips32} for adelic modules. One finds that it agrees
with the usual tensor product with respect to the $p$-adic (resp. real) vector
space structures. Once this has been checked, the rest is again immediate.
\end{proof}

\begin{proof}
[Proof of Thm. \ref{thmann_D}](Step 1) We will prove this by analyzing the
transfers on the level of the relevant descent spectral sequences for both
sides of the pairing in Eq. \ref{lmips22}. For $E_{r}^{i^{\prime},j^{\prime}%
}(\mathcal{O}_{S})\Rightarrow K(\mathcal{O}_{S})$ it is well-known that the
pullback corresponds to restriction and the pushforward to corestriction. For
$E_{r}^{i,j}(\mathsf{LCA}_{\mathcal{O}_{S}})\Rightarrow K(\mathsf{LCA}%
_{\mathcal{O}_{S}})$ the analogous property is shown in Lemma
\ref{lemma_TransferCompat}. In particular, it follows that we could
equivalently define the transfers by taking the induced map on the cofiber%
\[
\operatorname{cofib}\left(  K(\mathcal{O}_{S})\longrightarrow\left.
\underset{v\in S\;}{%
{\textstyle\prod\nolimits^{\prime}}
}\right.  \left(  K(F_{v}):K(\mathcal{O}_{v})\right)  \right)  \overset{\sim
}{\longrightarrow}K(\mathsf{LCA}_{\mathcal{O}_{S}})\text{.}%
\]
This reduces proving our claim to showing that restriction and corestriction
get interchanged by duality. In Galois cohomology the projection formula has
the concrete shape%
\begin{equation}
x\smile\operatorname*{cor}\nolimits_{F}^{F^{\prime}}(y)=\operatorname*{cor}%
\nolimits_{F}^{F^{\prime}}\left(  \operatorname*{res}\nolimits_{F}^{F^{\prime
}}x\smile y\right)  \text{.} \label{lmips29}%
\end{equation}
For every class formation the diagram%
\begin{equation}%
\xymatrix{
H^2(\mathcal{O}_{F^{\prime},S^{\prime}},C_{F^{\prime},S^{\prime}}%
) \ar[d]_{\operatorname{cor}^{F^{\prime}}_{F}} \ar[r]^-{\operatorname
{inv}_{F^{\prime}}}_-{\sim} & \mathbf{Q}/\mathbf{Z} \ar@{=}[d] \\
H^2(\mathcal{O}_{F,S},C_{F,S}) \ar[r]_-{\operatorname{inv}_F}^-{\sim}
& \mathbf{Q}/\mathbf{Z}
}
\label{lmips30}%
\end{equation}
commutes. Since $H^{2}(\mathcal{O}_{F,S},C_{S})\cong\operatorname*{Ext}%
\nolimits_{G_{S}}^{2}(\mathbf{Z},C_{S})\cong\operatorname*{Ext}%
\nolimits_{G_{S}}^{2}(\mathbf{Z},C_{S}^{\setminus\infty})$ and analogously for
$\mathcal{O}_{F^{\prime},S^{\prime}}$, we may use this to relate the pairings
in Eq. \ref{lmips19_pairing_form} for both $F$ and $F^{\prime}$. We compute%
\[
\operatorname*{inv}\nolimits_{F}\left(  x\smile\operatorname*{cor}%
\nolimits_{F}^{F^{\prime}}(y)\right)  \underset{\text{Eq. \ref{lmips29}}}%
{=}\operatorname*{inv}\nolimits_{F}\left(  \operatorname*{cor}\nolimits_{F}%
^{F^{\prime}}\left(  \operatorname*{res}\nolimits_{F}^{F^{\prime}}x\smile
y\right)  \right)  \underset{\text{Eq. \ref{lmips30}}}{=}\operatorname*{inv}%
\nolimits_{F^{\prime}}\left(  \operatorname*{res}\nolimits_{F}^{F^{\prime}%
}x\smile y\right)  \text{.}%
\]
It follows that restriction is left adjoint (in the sense of a bilinear
pairing) to corestriction. As these are the Galois cohomological counterparts
of pullback and pushforward, the corresponding claim holds for the $E_{2}%
$-page of the relevant descent spectral sequences, and therefore for
$L_{K(1)}K(-)$. As this argument essentially relies on a property of class
formations, it is also valid in the setting of Thm.
\ref{thmann_ThmA_LocalVersion}.
\end{proof}

\begin{conjecture}
Let $K$ denote non-connective $K$-theory. Let $p$ be an odd prime, $F$ a
number field, $S$ a finite set of finite places of $F$ such that $\frac{1}%
{p}\in\mathcal{O}_{S}$. Then there is an equivalence of fiber sequences%
\[%
\xymatrix{
L_{K(1)}K(\mathcal{O}_{S}) \ar[r] \ar@{=>}[d] &
L_{K(1)}K(\mathsf{LC}\mathcal{O}_{S,ad}) \ar[r] \ar@{<=>}[d] &
L_{K(1)}K(\mathsf{LC}\mathcal{O}_{S}) \\
I_{\mathbf{Z}_{p}}L_{K(1)}K(\mathsf{LC}\mathcal{O}_{S}) \ar[r] &
I_{\mathbf{Z}_{p}}L_{K(1)}K(\mathsf{LC}\mathcal{O}_{S,ad}) \ar[r] & I_{\mathbf
{Z}_{p}}L_{K(1)}K(\mathcal{O}_{S}) \ar@{=>}[u]
}%
\]
where the lower row is the Anderson dual of the top row (perhaps up to sign).
\end{conjecture}

By an equivalence of fiber sequences we mean that the underlying bi-Cartesian
square of spectra, including the null homotopy, comes with an equivalence to
the dual square. With the current method of proof, setting this up appears to
be a torturous undertaking.

\subsection{Expectations}

For all regular schemes $X$ there should exist a category $\mathsf{C}_{X}$ of
\textit{derived locally compact} $\mathcal{O}_{X}$-modules.

\begin{conjecture}
\label{c1}We restrict our attention to regular schemes which are flat, finite
type and separated over some ring of $S$-integers $\mathcal{O}_{S}$ of a
number field.

\begin{enumerate}
\item The association $U\mapsto\mathsf{C}_{U}$ should satisfy a Zariski
cosheaf property, i.e., if $U\subseteq U^{\prime}$ are opens, instead of a
restriction, there should be an extension of sections, an exact functor%
\[
\mathsf{C}_{U}\longrightarrow\mathsf{C}_{U^{\prime}}\text{.}%
\]

\item For every proper morphism $f\colon X\rightarrow Y$ there should be a
pullback\footnote{so exactly opposite to what one would usually expect}
$f^{\ast}\colon\mathsf{C}_{Y}\rightarrow\mathsf{C}_{X}$.

\item For every smooth morphism $f\colon X\rightarrow Y$ there should be a
pushforward $f_{\ast}\colon\mathsf{C}_{X}\rightarrow\mathsf{C}_{Y}$.

\item There should be a motivic weight filtration on $K(\mathsf{C}_{X})$ and
\textquotedblleft locally compact motivic cohomology\textquotedblright%
\ $H_{\operatorname*{lc}}^{\ast}(-,\mathbf{Z}(n))$ and a Atiyah--Hirzebruch
type spectral sequence $H_{\operatorname*{lc}}^{\ast}(-,\mathbf{Z}%
(\ast))\Rightarrow K(\mathsf{C}_{X})$ for this filtration.

\item Once $\frac{1}{p}\in\mathcal{O}_{S}$, $K(\mathsf{C}_{X})$ and $K(X)$
should be $K(1)$-local Anderson duals for the prime $p$, up to a shift.
\end{enumerate}

For regular scheme of finite type and separated over $\mathbf{F}_{p}$, there
should be an analogous formalism, valid away from $p$.
\end{conjecture}

In order to incorporate the prime $p$, one should switch to Selmer $K$-theory,
as envisioned by Clausen \cite{clausen}. There should be a definition of such
a category for all qcqs schemes, but then it should be a bit more tricky than
just demanding some form of derived local compactness. This can be seen from
\cite{MR4028830,MR4118150}, which show how for non-regular schemes one does
not get the right object even in the situation of a one-dimensional arithmetic
scheme. For example, $\mathsf{LCA}_{\mathfrak{A}}$ is \textit{not} the right
concept if one were to adapt the present text to working with non-maximal
orders $\mathfrak{A}\subsetneqq\mathcal{O}_{F}$ in a number field.

There is also a non-commutative version of all this. We have decided not to
include this into this manuscript, but see \cite{clausennc} for some first
steps. The geometric counterpart should be a version of the above conjecture,
twisted by Azumaya sheaves.

\appendix

\section{Proof of Theorem \ref{thmann_ThmA_LocalVersion}}

We have deferred the proof of Thm. \ref{thmann_ThmA_LocalVersion} to the
appendix because it is largely independent of the much more involved case of
number fields.

\textit{There is another subtle point: }Throughout this entire text, for any
topological ring $R$ the category $\mathsf{LCA}_{R}$ always refers to locally
compact topological $R$-modules with continuous $R$-module homomorphisms. In
the main body of the text, we only use this for rings with the discrete
topology. This is, at times, different in the Appendices (and only in them).
As a result, some simplifications are not longer possible. In particular, the
alternative description of $\mathsf{LCA}_{R}$ in Rmk.
\ref{rmk_WhatDoesLCAMeanForSomeRing} as $\mathsf{Fun}(\left\langle
\mathcal{O}_{S}\right\rangle $,$\mathsf{LCA}_{\mathbf{Z}})$ is only valid if
$R$ carries the discrete topology.

\begin{proof}
[Proof of Theorem \ref{thmann_ThmA_LocalVersion}]\textit{(Local version,
non-archimedean)} Suppose $F$ is a finite extension of $\mathbf{Q}_{\ell}$. By
\cite[Thm. 1.4]{MR4121155} the spectrum $K(F)$ is Anderson self-dual in the
$K(1)$-local homotopy category, namely%
\[
I_{\mathbf{Z}_{p}}L_{K(1)}K(F)\cong L_{K(1)}K(F)\text{.}%
\]
By \cite[\S A.2, notably Prop. A.4]{kthyartin} there is an exact equivalence
of exact categories $\mathsf{Proj}_{F,fg}\overset{\sim}{\rightarrow
}\mathsf{LCA}_{F}$, for topological locally compact $F$-vector spaces. This
proves the claim. \textit{(Local version, archimedean)} Straightforward direct
verification. \textit{(Finite field version) }The Anderson dual of
$K(\mathbf{F}_{q})$ is also studied by Blumberg--Mandell--Yuan \cite[Prop.
6.5]{bmychromaticconvergence}. The shift $\Sigma$ in the dual they observe can
be interpreted on the locally compact side by a shift which exists
\textit{before} localizing to the $K(1)$-local world. It then becomes an
instance of the delooping property of Tate categories. We explain this in a
self-contained way: Suppose $F$ is a finite extension of $\mathbf{F}_{\ell}$.
By \cite[proof of Prop. A.1]{kthyartin} there is an exact equivalence of exact
categories $\Gamma\colon\mathsf{LCA}_{F}\overset{\sim}{\rightarrow
}\mathsf{Tate}(\mathsf{Mod}_{F,fg})$, which is also called $\Gamma$ loc. cit.,
so%
\begin{equation}
K(\mathsf{LCA}_{F})\cong K(\mathsf{Tate}(\mathsf{Mod}_{F,fg}))\cong\Sigma
K(\mathsf{Mod}_{F,fg})=\Sigma K(F)\text{,} \label{l_sDeloopClassical}%
\end{equation}
where the second equivalence is Sho Saito's delooping theorem for Tate
categories \cite{MR3317759}. The cohomological dimension of finite fields is
just one (the absolute Galois group is $\widehat{\mathbf{Z}}$), so the
Thomason descent spectral sequence looks like in Diag. \ref{lwui1}, but with
all the $H^{2}$-terms being zero. Combining this with the equivalence of Eq.
\ref{l_sDeloopClassical} (run backwards), we obtain%
\begin{align*}
\pi_{2j+1}L_{K(1)}K(\mathsf{LCA}_{F})\otimes\mathbb{S}/p^{k}  &  \cong
H^{0}(F,\mathbf{Z}/p^{k}\mathbf{Z}(j))\\
\pi_{2j}L_{K(1)}K(\mathsf{LCA}_{F})\otimes\mathbb{S}/p^{k}  &  \cong
H^{1}(F,\mathbf{Z}/p^{k}\mathbf{Z}(j))
\end{align*}
for all $j\in\mathbf{Z}$ and $k\geq1$. For $j=0$ the second isomorphism yields
a compatible system of isomorphisms to $\mathbf{Z}/p^{k}\mathbf{Z}$, so for
the colimit we obtain%
\[
\pi_{0}\left(  L_{K(1)}K(\mathsf{LCA}_{F})\otimes\mathbb{S}\mathbf{Q}%
_{p}/\mathbf{Z}_{p}\right)  \cong H^{1}(F,\mathbf{Q}_{p}/\mathbf{Z}_{p}%
)\cong\mathbf{Q}_{p}/\mathbf{Z}_{p}\text{.}%
\]
Rephrasing this slightly using $\mathbb{S}\mathbf{Q}_{p}/\mathbf{Z}_{p}%
\cong\Sigma\mathcal{N}$, this defines a map
\[
\pi_{0}\left(  L_{K(1)}\Sigma K(\mathsf{LCA}_{F})\otimes\mathcal{N}\right)
\longrightarrow\mathbf{Q}_{p}/\mathbf{Z}_{p}%
\]
and thus $L_{K(1)}\Sigma K(\mathsf{LCA}_{F})\rightarrow I_{\mathbf{Q}%
_{p}/\mathbf{Z}_{p}}\mathbb{S}_{\widehat{p}}\cong\Sigma I_{\mathbf{Z}_{p}%
}\mathbb{S}_{\widehat{p}}$ by the properties of Brown--Comenetz and Anderson
duals in the $p$-complete setting (concretely, this has used Eq. \ref{lqwxt3}
and Eq. \ref{lrrb1}). Desuspension yields the map%
\begin{equation}
L_{K(1)}K(\mathsf{LCA}_{F})\longrightarrow I_{\mathbf{Z}_{p}}\mathbb{S}%
_{\widehat{p}}\text{,} \label{lrrb2}%
\end{equation}
which will act as the finite field analogue of the role of the trace map of
\S \ref{sect_TraceMap}. The pairing of \S \ref{subsect_RingStructures} can
more generally be set up for arbitrary rings $R$ with the discrete topology,
giving $K(\mathsf{LCA}_{R})$ a $K(R)$-module structure. The resulting pairing%
\[
K(R)\otimes_{\mathsf{Sp}}K(\mathsf{LCA}_{R})\longrightarrow K(\mathsf{LCA}%
_{R})
\]
for our special case $R:=F$, composed with the trace map of Eq. \ref{lrrb2},
then induces a pairing%
\[
L_{K(1)}K(F)\otimes_{\mathsf{K}}L_{K(1)}K(\mathsf{LCA}_{F})\longrightarrow
L_{K(1)}K(\mathsf{LCA}_{F})\longrightarrow I_{\mathbf{Z}_{p}}\mathbb{S}%
_{\widehat{p}}\text{.}%
\]
Unravelling what this pairing does, one shows that this is a perfect pairing
by showing that it induces a perfect pairing on homotopy groups, which in turn
follows from the existence of a perfect pairing on the $E_{2}$-page of the
associated Thomason descent spectral sequence. Said pairing (for
$L_{K(1)}K\otimes\mathbb{S}/p^{k}$) is
\[
H^{i}(F,\mathbf{Z}/p^{k}\mathbf{Z}(j))\otimes_{\mathbf{Z}}H^{1-i}%
(F,\mathbf{Z}/p^{k}\mathbf{Z}(-j))\longrightarrow H^{1}(F,\mathbf{Z}%
/p^{k}\mathbf{Z}(0))\cong\mathbf{Z}/p^{k}\mathbf{Z}%
\]
and is perfect by the Galois cohomological duality for finite fields (or one
could just say that this holds because $\operatorname*{Gal}%
(F^{\operatorname*{sep}}/F)\cong\widehat{\mathbf{Z}}$ being a profinitely free
profinite group of rank one). This finishes the proof.
\end{proof}

The above argument is entirely analogous to the one in \S \ref{sect_Duality}
and as is explained there, it is indeed sufficient to show that the perfect
pairings are perfect $\operatorname{mod}p$.

\begin{corollary}
\label{cor_KOneSelfDualityOfAdelicModules}Let $p$ be an odd prime, $F$ a
number field, $S$ a finite set of places of $F$ containing the infinite places
and such that $\frac{1}{p}\in\mathcal{O}_{S}$. Then we have the self-duality%
\[
I_{\mathbf{Z}_{p}}L_{K(1)}K(\mathcal{O}_{S,ad})\cong L_{K(1)}K(\mathsf{LCA}%
_{\mathcal{O}_{S},ad})\text{.}%
\]

\end{corollary}

The self-duality does not hold for infinite sets $S$, as explained in Pitfall
\ref{pitfall_l1}.

\begin{proof}
Use Prop. \ref{prop_IdentifyAdelicBlocks} and take the sum of the
self-dualities provided by Thm. \ref{thmann_ThmA_LocalVersion}.
\end{proof}

\begin{pitfalls}
\label{pitfall_l1}\label{rmk_TheStoryAboutTateCohomology3}The proof of Cor.
\ref{cor_KOneSelfDualityOfAdelicModules} is easy enough not to raise concerns.
Nonetheless, there are two subtle points:\ (1) One might wonder whether this
self-duality holds for infinite sets of places. This is particularly tempting
since there is a very general Pontryagin self-duality%
\begin{equation}
P^{i}(F,M)\longrightarrow P^{2-i}(F,M)^{\vee}\qquad\text{(for all }%
i\in\mathbf{Z}\text{)} \label{lww8}%
\end{equation}
for the groups $P^{i}$ for any finite Galois module $M$, once equipped with
natural locally compact topologies, see \cite[Ch. VIII, Prop. 8.5.2]%
{MR2392026}. This duality is valid for any choice of $S$ and clearly one would
want to model Cor. \ref{cor_KOneSelfDualityOfAdelicModules} on this. However,
as this is a topological Pontryagin duality, one needs to reduce it to a
merely algebraic duality with respect to $\operatorname*{Hom}_{\mathsf{Ab}%
}(-,\mathbf{Q}/\mathbf{Z})$ on the underlying groups to reformulate it in
terms of Brown--Comenetz duals (and, as discussed in
\S \ref{subsect_BCvsAndersonDuals}, the same applies to $p$-complete Anderson
duals). As we explain in detail in Example \ref{ex_PontryaginWithQZCoeffs},
this reformulation is only possible if the compact (profinite) groups in
question are topologically finitely generated. This means that our
$p$-complete spectra need to have finitely generated $\mathbf{Z}_{p}$-modules
as homotopy groups, and this fails once $S$ is infinite. (2) At infinite
places, the duality in Eq. \ref{lww8} relies on a perfect duality of Tate
cohomology groups $\widehat{H}^{0}\otimes\widehat{H}^{2}\rightarrow\widehat
{H}^{2}$, which moreover are $2$-torsion at worst. As we have already
explained in Rmk. \ref{rmk_TheStoryAboutTateCohomology1} and Rmk.
\ref{rmk_TheStoryAboutTateCohomology2}, this type of duality does not surface
in the homotopy groups of $K(1)$-localizations. Instead, the duality in Thm.
\ref{thmann_ThmA_LocalVersion} for $\mathbf{R}$ and $\mathbf{C}$ uses just the
self-duality of $H^{0}$ with itself in ordinary Galois cohomology,
contributing $p$-primary torsion. It might appear confusing to replace a
$2$-torsion phenomenon by a $p$-torsion one, but note that these are two
entirely independent stories: The $2$-torsion duality in Tate cohomology has
actual arithmetic significance in Poitou--Tate duality (e.g., the Hilbert
symbol at $\mathbf{R}$ stems from it), but is irrelevant for us, since our $p$
is odd. The $p$-primary torsion self-duality on $H^{0}$ is rather dull. It has
no arithmetic relevance and has no counterpart in Poitou--Tate duality.
\end{pitfalls}

\begin{pitfalls}
The approach of Vista \ref{vista_1} should permit to generalize Cor.
\ref{cor_KOneSelfDualityOfAdelicModules} to infinite sets $S$.
\end{pitfalls}

\section{Locally compact modules with cardinality
bounds\label{sect_Appendix_SetTheory}}

Let $R$ be a Noetherian ring. One might feel more comfortable to avoid working
with the category of \emph{all} $R$-modules, and instead impose some
cardinality bound. Similarly, it might be desirable to restrict to locally
compact modules satisfying some size constraint. In this appendix we sketch
how this can be done, and we ultimately see that cardinality bounds, once big
enough, do not affect the $K$-theory (or other localizing invariants) at all.

Let $\kappa$ be an infinite cardinal and let $\mathsf{Mod}_{R}^{\leq\kappa}$
denote the abelian category of $R$-modules which admit a $\kappa$-generating
set. Alternatively, $\mathsf{Mod}_{R}^{\leq\kappa}\cong\mathsf{Ind}_{\kappa
}(\mathsf{Mod}_{R,fg})$, where $\mathsf{Ind}_{\kappa}$ denotes the
$\operatorname*{Ind}$-category based on diagrams of cardinality at most
$\kappa$. We recall that every locally compact abelian group $M$ admits a
non-canonical decomposition $M\simeq\mathbf{R}^{n}\oplus N$ with%
\begin{equation}
C\hookrightarrow N\twoheadrightarrow D\text{,} \label{lsets1}%
\end{equation}
where $C$ is a compact clopen subgroup of $N$ and $D$ is discrete. The
isomorphism class of $N$ is uniquely determined \cite{MR578649} (a nice proof
can be found in \cite[Prop. 33]{MR1856638}).

\begin{definition}
\label{dsets1}Let $\kappa$ be an infinite cardinal. Let $\mathsf{LCA}%
_{\mathbf{Z}}^{\leq\kappa}$ be the category of locally compact abelian groups
which admit a decomposition as in Eq. \ref{lsets1} such that $D\in
\mathsf{Mod}_{\mathbf{Z}}^{\leq\kappa}$ and $C^{\vee}\in\mathsf{Mod}%
_{\mathbf{Z}}^{\leq\kappa}$.
\end{definition}

Since $C$ is compact, its Pontryagin dual $C^{\vee}$ is discrete and therefore
can be regarded as an object in $\mathsf{Mod}_{\mathbf{Z}}$. If we have a
second decomposition $C^{\prime}\hookrightarrow N\twoheadrightarrow D^{\prime
}$ as in Eq. \ref{lsets1} such that $C^{\prime}\subseteq C$, then
$C/C^{\prime}$ must be finite (since $C$ is compact and $C^{\prime}$ is open
in $N$, and therefore $C/C^{\prime}$ carries the discrete topology).
Correspondingly, both $C$ and $D$ are well-defined up to a zig-zag of finite
subquotients. Hence, if the conditions in Def. \ref{dsets1} hold for one
sequence as in Eq. \ref{lsets1}, they hold for any. Using this, one easily
shows the following.

\begin{lemma}
$\mathsf{LCA}_{\mathbf{Z}}^{\leq\kappa}$ is a quasi-abelian category.
\end{lemma}

As always, taking closed injections as admissible monics and open surjections
as admissible epics, $\mathsf{LCA}_{\mathbf{Z}}^{\leq\kappa}$ carries an exact
structure (which agrees with the maximal exact structure). Next, define
$\mathsf{LCA}_{\mathcal{O}_{S}}^{\leq\kappa}:=\mathsf{LCA}_{\mathcal{O}_{S}%
}\cap\mathsf{LCA}_{\mathbf{Z}}^{\leq\kappa}$. Note that Pontryagin duality
restricts to an exact equivalence of exact categories%
\[
(-)^{\vee}\colon\mathsf{LCA}_{\mathcal{O}_{S}}^{\leq\kappa}\longrightarrow
\mathsf{LCA}_{\mathcal{O}_{S}}^{\leq\kappa,op}\text{,}%
\]
just by the self-dual nature of Def. \ref{dsets1}. All the arguments in the
present text (as well as \cite{MR4028830,kthyartin}) work in $\mathsf{LCA}%
_{\mathcal{O}_{S}}^{\leq\kappa}$ for any choice of infinite cardinal $\kappa$.

\begin{example}
Our definition of $\mathsf{LCA}_{\mathcal{O}_{S}}^{\leq\kappa}$ does not imply
that all modules in the category have $\kappa$-generating sets. For example,
$\mathbf{R}$, $\mathbf{Z}_{p}$, $\mathbf{Q}_{p}\in\mathsf{LCA}_{\mathbf{Z}%
}^{\leq\kappa}$ even for $\kappa=\aleph_{0}$, although these abelian groups
are uncountable.
\end{example}

Most crucially, the\ Eilenberg swindles which we use at various points in
proofs, only require the existence of countable products of compact modules,
and countable coproducts of discrete modules.\ These indeed exist in any such
$\mathsf{LCA}_{\mathcal{O}_{S}}^{\leq\kappa}$. Finally, for cardinals
$\kappa\leq\kappa^{\prime}$, there is the inclusion as a fully exact
subcategory%
\[
\mathsf{LCA}_{\mathcal{O}_{S}}^{\leq\kappa}\hookrightarrow\mathsf{LCA}%
_{\mathcal{O}_{S}}^{\leq\kappa^{\prime}}%
\]
(and analogously for adelic, ind-c.g., etc. modules), inducing maps to any
localizing invariant $K\colon\operatorname*{Cat}_{\infty}^{\operatorname*{ex}%
}\rightarrow\mathsf{A}$. As our computations work for any $\mathsf{LCA}%
_{\mathcal{O}_{S}}^{\leq\kappa}$, this shows that the value $K(\mathsf{LCA}%
_{\mathcal{O}_{S}}^{\leq\kappa})$ (resp. $\mathsf{LCA}_{\mathcal{O}_{S}%
,ad}^{\leq\kappa}$) remains constant along increasing $\kappa$. This justifies
why it is essentially irrelevant to discuss cardinality constraints in this text.

As an aside, we mention a different perspective.

\begin{lemma}
For $\kappa=\aleph_{0}$, $\mathsf{LCA}_{\mathcal{O}_{S}}^{\leq\aleph_{0}}$ is
equivalent to the category of second countable locally compact $\mathcal{O}%
_{S}$-modules.
\end{lemma}

\begin{proof}
First, a discrete space is second countable iff it is countable. Hence, the
condition $D\in\mathsf{Mod}_{\mathbf{Z}}^{\leq\aleph_{0}}$ in Def.
\ref{dsets1} is equivalent to $D$ being second countable. Next, Pontryagin
duality sends second countable spaces to second countable spaces, so the
compact piece $C$ is second countable iff $C^{\vee}$ is countable. Finally,
vector modules are always second countable.
\end{proof}

\section{\label{sect_Appendix_DualityInKOneLocalHptyCat}Duality in the
$K(1)$-local homotopy category}

Let $p$ be an odd prime. We write $(\mathsf{Sp},\otimes_{\mathsf{Sp}})$ for
the symmetric monoidal stable $\infty$-category of spectra with respect to the
smash product. Concretely, we take symmetric spectra as the implementation and
largely follow Schwede's notation \cite{symspec}. The homotopy category
$Ho(\mathsf{Sp},\otimes_{\mathsf{Sp}})$ then is naturally a tensor
triangulated category and, stronger, an algebraic stable homotopy category in
the sense of \cite[Def. 1.1.4]{MR1388895}, \cite[\S 1.3]{MR1601906}. Write
$\mathbb{S}A$ for a Moore spectrum of the abelian group $A$. We will only ever
need $\mathbb{S}\mathbf{Z}/p^{k}\cong\mathbb{S}/p^{k}$ and its limits and
colimits, and we may fix a concrete choice for these. We use the shorthand
$\mathcal{N}:=\Sigma^{-1}\mathbb{S}\mathbf{Q}_{p}/\mathbf{Z}_{p}$ and denote
periodic complex $K$-theory by $\operatorname*{KU}$. Let $\mathsf{K}$ be the
$K(1)$-local homotopy category for the prime $p$. We write $L_{K(1)}$ for the
corresponding localization functor,%
\[
L_{K(1)}\colon\mathsf{Sp}\longrightarrow\mathsf{Sp}\text{,}%
\]
so $\mathsf{K}$ is the subcategory of $K(1)$-local objects. Then $\mathsf{K}$
again is a symmetric monoidal stable $\infty$-category, now under%
\[
X\otimes_{\mathsf{K}}Y:=L_{K(1)}(X\otimes_{\mathsf{Sp}}Y)\text{.}%
\]
and $Ho(\mathsf{K},\otimes_{\mathsf{K}})$ is again an algebraic stable
homotopy category \cite[Thm. 7.1]{MR1601906}. While its tensor structure
requires localizing again, the function spectrum%
\begin{equation}
F_{\mathsf{Sp}}(X,Y)=F_{\mathsf{K}}(X,Y) \label{lchh1}%
\end{equation}
is automatically $K(1)$-local. The $K(1)$-local spectra are in particular
$p$-complete. We also use a `dual counterpart' $L_{\operatorname*{KU}%
}\mathsf{Sp}_{\operatorname*{tor}}$, the localization of $p$-torsion spectra
at complex $K$-theory. The natural adjunction for smashing with $\mathcal{N}$,%
\[
F_{\mathsf{Sp}}(\mathcal{N}\otimes_{\mathsf{Sp}}X,Y)\cong F_{\mathsf{Sp}%
}(X,F_{\mathsf{Sp}}(\mathcal{N},Y))\text{.}%
\]
The adjoint functor pair induces an equivalence of stable $\infty$-categories%
\[
\mathsf{K}=L_{K(1)}\mathsf{Sp}\overset{\sim}{\longrightarrow}%
L_{\operatorname*{KU}}\mathsf{Sp}_{\operatorname*{tor}}\text{,}%
\]
where $X\mapsto X\otimes_{\mathsf{Sp}}\mathcal{N}$ (resp. $F_{\mathsf{Sp}%
}(\mathcal{N},X)\leftarrow X$ in the opposite direction). This equivalence at
first exists between merely $p$-complete vs. $p$-torsion spectra in
$\mathsf{Sp}$, but it also adapts to the $\operatorname*{KU}$-local setting.
See \cite{luriedagcomp} for a much more general setting.

\subsection{Brown--Comenetz and Anderson
Duality\label{subsect_BCvsAndersonDuals}}

Let us review Brown--Comenetz and Anderson duality in the $p$-complete
setting, as set up in \cite{MR1601906,MR2327028}. A very readable overview can
be found in \cite[\S 1, \S 2]{MR2946825}. Suppose $J$ is an injective abelian
group. Write $\mathsf{Sp}_{\widehat{p}}$ for the stable $\infty$-category of
$p$-complete spectra. The functor $\mathfrak{D}_{J}\colon\mathsf{Sp}%
_{\widehat{p}}^{op}\longrightarrow\mathsf{Ab}$ given by%
\begin{equation}
Z\mapsto\operatorname*{Hom}\nolimits_{\mathsf{Ab}}\left(  \pi_{0}%
(Z\otimes_{\mathsf{Sp}}\mathcal{N}),J\right)  \label{lqwxa1}%
\end{equation}
is a cohomological functor. In more detail: Fiber sequences in $Z$ induce a
long exact sequence of homotopy groups. Since $J$ is injective,
$\operatorname*{Hom}\nolimits_{\mathsf{Ab}}(-,J)$ is an exact functor, so all
in all we obtain a contravariant functor sending fiber sequences to exact
sequences. Moreover, arbitrary coproducts in $Z$ commute with the tensor
product\footnote{which is a smash product in the present situation}, and with
$\pi_{0}$. As $\operatorname*{Hom}\nolimits_{\mathsf{Ab}}(-,J)$ maps
coproducts to products, the claim follows. As a result, there is a unique
representing spectrum $D_{J}$ such that $\mathfrak{D}_{J}%
(Z)=\operatorname*{Hom}\nolimits_{Ho(\mathsf{Sp})}(Z,D_{J})$. Finally, note
that since $Z$ is $p$-complete, $\pi_{0}(Z\otimes_{\mathsf{Sp}}\mathcal{N})$
is a $p$-torsion group, so that $\operatorname*{Hom}\nolimits_{\mathsf{Ab}%
}(\pi_{0}(Z\otimes_{\mathsf{Sp}}\mathcal{N}),J)$ is (classically and derived)
$p$-complete. Hence, $D_{J}$ is itself a $p$-complete spectrum. If
$J\rightarrow J^{\prime}$ is a morphism of injective abelian groups, we obtain
an induced morphism of cohomological functors $\mathfrak{D}_{J}\rightarrow
\mathfrak{D}_{J^{\prime}}$, and a corresponding representing map of spectra
$D_{J}\rightarrow D_{J^{\prime}}$. The $p$-complete \emph{Brown--Comenetz
dual} of a $p$-complete spectrum $X$ is defined as%
\begin{equation}
I_{\mathbf{Q}_{p}/\mathbf{Z}_{p}}X:=F_{\mathsf{Sp}}(X,D_{\mathbf{Q}%
_{p}/\mathbf{Z}_{p}}\mathbb{)}\text{.} \label{lqwxt2}%
\end{equation}
In particular, instead of $D_{\mathbf{Q}_{p}/\mathbf{Z}_{p}}$ we can also
write $I_{\mathbf{Q}_{p}/\mathbf{Z}_{p}}\mathbb{S}_{\widehat{p}}$,
where$\mathbb{\ S}_{\widehat{p}}$ denotes the $p$-completed sphere spectrum.
Since both $X$ and $D_{\mathbf{Q}_{p}/\mathbf{Z}_{p}}$ are $p$-complete, so is
$I_{\mathbf{Q}_{p}/\mathbf{Z}_{p}}X$. Using this notation and unravelling the
representability yields the equation%
\begin{equation}
\operatorname*{Hom}\nolimits_{\mathsf{Sp}}(Z,I_{\mathbf{Q}_{p}/\mathbf{Z}_{p}%
}\mathbb{S}_{\widehat{p}})\cong\operatorname*{Hom}\nolimits_{\mathsf{Ab}%
}\left(  \pi_{0}(Z\otimes_{\mathsf{Sp}}\mathcal{N}),\mathbf{Q}_{p}%
/\mathbf{Z}_{p}\right)  \text{,} \label{lqwxt3}%
\end{equation}
which we use to set up morphisms to $I_{\mathbf{Q}_{p}/\mathbf{Z}_{p}%
}\mathbb{S}_{\widehat{p}}$. For brevity, we write $G^{\ast}%
:=\operatorname*{Hom}_{\mathsf{Ab}}(X,\mathbf{Q}_{p}/\mathbf{Z}_{p})$.

\begin{remark}
\label{rmk_ClassicalBrownComenetzDual}The original Brown--Comenetz dual is
defined on arbitrary spectra, and based on representing the functor%
\[
Z\mapsto\operatorname*{Hom}\nolimits_{\mathsf{Ab}}\left(  \pi_{0}%
Z,\mathbf{Q}/\mathbf{Z}\right)
\]
instead of the one in Eq. \ref{lqwxa1}. Restricting to its $p$-typical part,
i.e., $\mathbf{Q}_{p}/\mathbf{Z}_{p}\subset\mathbf{Q}/\mathbf{Z}$, the functor
$Z\mapsto\operatorname*{Hom}\nolimits_{\mathsf{Ab}}(\pi_{0}Z,\mathbf{Q}%
_{p}/\mathbf{Z}_{p})$ still has the property that if $Z$ happens to be a
$p$-complete spectrum, its so defined Brown--Comenetz dual will usually fail
to be $p$-complete. One might therefore prefer to denote the dual of Eq.
\ref{lqwxt2} by $\widehat{I}_{\mathbf{Q}_{p}/\mathbf{Z}_{p}}$ and call it the
$p$-complete dual, as to avoid misunderstandings. This is a matter of taste.
\end{remark}

\begin{corollary}
\label{cor_duals_c1}

\begin{enumerate}
\item $\pi_{n}(I_{\mathbf{Q}_{p}/\mathbf{Z}_{p}}X)\cong\pi_{-n}(X\otimes
_{\mathsf{Sp}}\mathcal{N})^{\ast}$

\item $\pi_{n}(I_{\mathbf{Q}_{p}/\mathbf{Z}_{p}}X/p^{k})\cong\pi_{-n}%
(X/p^{k})^{\ast}$
\end{enumerate}
\end{corollary}

Property (2) reflects the properties of the classical Brown--Comenetz dual of
Rmk. \ref{rmk_ClassicalBrownComenetzDual}.

We can also work with $J:=\mathbf{Q}_{p}$. The exact sequence%
\[
\mathbf{Z}_{p}\hookrightarrow\mathbf{Q}_{p}\twoheadrightarrow\mathbf{Q}%
_{p}/\mathbf{Z}_{p}\text{,}%
\]
which can be read as an injective resolution of $\mathbf{Z}_{p}$, motivates
defining a further dual as%
\[
I_{\mathbf{Z}_{p}}X:=F_{\mathsf{Sp}}\left(  X,\operatorname*{fib}\left(
D_{\mathbf{Q}_{p}}\rightarrow D_{\mathbf{Q}_{p}/\mathbf{Z}_{p}}\right)
\right)  \text{.}%
\]
This corresponds to the $p$-complete \emph{Anderson dual }of a $p$-complete
spectrum. However, nothing much interesting happens here because
$\mathfrak{D}_{\mathbf{Q}_{p}}(Z)=\operatorname*{Hom}\nolimits_{\mathsf{Ab}%
}(\pi_{0}(Z\otimes_{\mathsf{Sp}}\mathcal{N}),\mathbf{Q}_{p})$ is the zero
functor since $\pi_{0}$ is always $p$-torsion for $p$-complete input spectra
$Z$. Hence, both duals just differ by a shift:%
\begin{align}
I_{\mathbf{Z}_{p}}X  &  \cong F\left(  X,\operatorname*{fib}\left(
0\rightarrow D_{\mathbf{Q}_{p}/\mathbf{Z}_{p}}\right)  \right)  \cong F\left(
X,\Sigma^{-1}D_{\mathbf{Q}_{p}/\mathbf{Z}_{p}}\right) \label{lrrb1}\\
&  \cong\Sigma^{-1}F\left(  X,D_{\mathbf{Q}_{p}/\mathbf{Z}_{p}}\right)
\cong\Sigma^{-1}I_{\mathbf{Q}_{p}/\mathbf{Z}_{p}}X\text{.}\nonumber
\end{align}
As a result, in the setting of this text, it does not make too much of a
difference whether we talk about Brown--Comenetz or Anderson duals. Finally,
the above discussion for $p$-complete spectra applies verbatim to the
$K(1)$-local homotopy category. Once we work in the $K(1)$-local setting, by
Gross--Hopkins even switching to $p$-complete Spanier--Whitehead duals is
roughly equivalent to taking any of the aforementioned duals\footnote{Recall
that there is no difference between using $F_{\mathsf{K}}$ or writing
$F_{\mathsf{Sp}}$ for $K(1)$-local input.}%
\begin{equation}
F_{\mathsf{K}}(X,I_{\mathbf{Q}_{p}/\mathbf{Z}_{p}}\mathbb{S}_{\widehat{p}%
}\mathbb{)}\cong F_{\mathsf{K}}(X,\mathbb{S}_{\widehat{p}})\otimes
_{\mathsf{K}}I_{\mathbf{Q}_{p}/\mathbf{Z}_{p}}\mathbb{S}_{\widehat{p}}\text{,}
\label{l_GHDuality}%
\end{equation}
where $\mathbb{S}_{\widehat{p}}:=L_{K(1)}\mathbb{S}$ now denotes the
$K(1)$-local version of the $p$-complete sphere, which now has become
$\otimes_{\mathsf{K}}$-invertible in the monoidal structure. See \cite[Cor.
10.6]{MR1601906} or the articles \cite{MR1763961}, \cite{MR1217353}.

\subsection{Perfect pairings\label{subsect_PerfectPairings}}

A pairing $\mu\colon X\otimes Y\rightarrow Z$ is called \emph{left perfect} if
the Hom-tensor adjunction $\operatorname*{Hom}(X\otimes Y,Z)\cong%
\operatorname*{Hom}(X,\operatorname*{Hom}(Y,Z))$ sends $\mu$ to an equivalence
$X\overset{\sim}{\rightarrow}\operatorname*{Hom}(Y,Z)$. We call the pairing
$\mu$ \emph{right perfect} if the opposite pairing $\mu^{\operatorname*{op}%
}\colon Y\otimes X\rightarrow Z$ is a left perfect pairing. Call a pairing
\emph{perfect} if it is both left and right perfect.

\begin{remark}
\label{rmk_CanTestPerfectPairingModP}For $p$-complete spectra it is equivalent
to check whether the pairing is left perfect after tensoring with
$\mathbb{S}/p$. This is seen as follows:\ $\mu$ is left perfect if $f\colon
X\overset{\sim}{\rightarrow}\operatorname*{Hom}(Y,Z)$ is an equivalence, but a
morphism $f$ of $p$-complete spectra is an equivalence iff $f\otimes
\mathbb{S}/p$ is an equivalence. To see this, note that fiber
$\operatorname*{fib}(f)/p\simeq0$ implies that multiplication by $p$ is an
equivalence, hence so it is for any power, so $\operatorname*{fib}%
(f)/p^{n}\simeq0$. But then $\operatorname*{fib}(f)\cong\lim
\operatorname*{fib}(f)/p^{n}$ by $p$-completeness, but the latter vanishes.
The converse is clear. Thus, $f$ is an equivalence iff $X/p\overset{\sim
}{\rightarrow}\operatorname*{Hom}(Y,Z)/p$ is an equivalence, and running the
definition of a left perfect pairing backwards, this holds iff $\mu/p\colon
X/p\otimes Y/p\rightarrow Z/p$ is left perfect.
\end{remark}

\subsection{Reflexivity}

In this section, we discuss why the direction $L_{K(1)}K(\mathcal{O}_{S})\cong
I_{\mathbf{Z}_{p}}L_{K(1)}K(\mathsf{LCA}_{\mathcal{O}_{S}})$ of Thm.
\ref{thmann_ThmA} only holds for finite $S$ and how one might resolve this
issue. Fix an object $D$ in some (stable $\infty$-)category. Define
$I(X):=\operatorname*{Hom}(X,D)$. Then by adjunction there is a natural
induced \emph{double dual map} $X\rightarrow I^{2}X$. We call $X$
\emph{reflexive} (for the chosen type of duality) if the double dual map is an isomorphism.

\begin{example}
Let $R$ be any discrete ring. Then in $\mathsf{LCA}_{R}$, every object is
reflexive for Pontryagin duality, $D:=\mathbf{R}/\mathbf{Z}$ (the circle group).
\end{example}

\begin{example}
The $K(1)$-local Brown--Comenetz dual is reflexive for spectra such that all
their homotopy groups $\pi_{i}X$ are finitely generated $\mathbf{Z}_{p}%
$-modules\footnote{Since the homotopy groups are derived $p$-complete, they
carry a \textit{canonical} $\mathbf{Z}_{p}$-module structure.}. The reason for
this is solely a phenomenon within abelian groups and we explain it in Example
\ref{ex_PontryaginWithQZCoeffs}. Since the Brown--Comenetz dual and the
Anderson dual only differ by a shift, the same applies there.
\end{example}

\begin{example}
\label{ex_PontryaginWithQZCoeffs}Take $R:=\mathbf{Z}$ in the previous example.
Then Pontryagin duality induces an anti-equivalence between the following full
subcategories of $\mathsf{LCA}_{\mathbf{Z}}$:
\[
\text{discrete torsion abelian groups}\leftrightarrow\text{profinite abelian
groups.}%
\]
One can try to reconstruct this duality from the underlying abelian group
structures \emph{without the topology}. On the left, the underlying abelian
group of the dual can be computed by%
\[
G^{\vee}=_{def}\operatorname*{Hom}\nolimits_{\mathsf{LCA}_{\mathbf{Z}}%
}(G,\mathbf{R}/\mathbf{Z})=\operatorname*{Hom}\nolimits_{\mathsf{Ab}%
}(G,\mathbf{Q}/\mathbf{Z})
\]
since $G$ is torsion and there is no difference between plain algebraic and
continuous group homomorphisms. The reverse direction is tricky. Since $G$ is
profinite, one might be tempted to define the dual by $\operatorname*{Hom}%
\nolimits_{\mathsf{LCA}_{\mathbf{Z}}}(\underleftarrow{\lim}_{H}G/H,\mathbf{R}%
/\mathbf{Z})$, where $H$ runs through all finite index subgroups of $H$. Then
each $G/H$ is finite, so the inverse limit is naturally a profinite group.
This idea is based on the hope that%
\[
G=\underleftarrow{\lim}_{H}G/H\text{,}%
\]
would reconstruct $G$ purely from the group structure. This idea indeed works
if $G$ is topologically finitely generated, since then profinite completion is
an idempotent functor. However, in general, the profinite completion of a
profinite group yields a different group. The introduction to \cite{MR2276769}
provides more background on the general problem. If $G$ happens to be
topologically finitely generated, one indeed gets%
\begin{align*}
G^{\vee}=_{def}\operatorname*{Hom}\nolimits_{\mathsf{LCA}_{\mathbf{Z}}%
}(G,\mathbf{R}/\mathbf{Z})  &  \cong\operatorname*{Hom}\nolimits_{\mathsf{LCA}%
_{\mathbf{Z}}}(\underleftarrow{\lim}_{H}G/H,\mathbf{R}/\mathbf{Z})\\
&  \cong\operatorname*{Hom}\nolimits_{\mathsf{Ab}}(G,\mathbf{Q}/\mathbf{Z}%
)\text{,}%
\end{align*}
i.e., both duals can be defined purely on the level of their underlying
abelian groups.
\end{example}

\begin{example}
By Lemma \ref{lemma_specseq_for_restricted_product_inside} (for $j=1$) we have%
\[
\pi_{1}L_{K(1)}(K/p)(\mathsf{LCA}_{\mathcal{O}_{S},ad})\cong P_{S}%
^{1}(F,\mathbf{Z}/p(1))\text{.}%
\]
For all finite places $v\in S$ we may write $F_{v}^{\times}\simeq\left\langle
\pi_{v}\right\rangle \times\mathcal{O}_{v}^{\times}$ for $\pi_{v}$ a local
uniformizer, and thus%
\[
\simeq\underset{v\in S}{\left.  \prod\nolimits^{\prime}\right.  }%
(F_{v}^{\times}/p:\mathcal{O}_{v}^{\times}/p)\cong\bigoplus_{v\in S\setminus
S_{\infty}}\mathbf{F}_{p}\oplus\underset{v\in S\setminus S_{\infty}}{\left.
\prod\right.  }\mathcal{O}_{v}^{\times}/p
\]
and the direct sum fails to be a finitely generated $\mathbf{Z}_{p}$-algebra
as soon as $S$ is infinite.
\end{example}

\begin{vista}
\label{vista_1}I would hope that one can develop a category of condensed
spectra so that for each $X$, each $\pi_{i}X$ is a condensed abelian group.
Then we may speak of \emph{locally compact spectra} $\mathsf{Sp}_{lc}$ as
those where all $\pi_{i}X$ come from an LCA\ group. In this subcategory, a
suitable form of Brown representability should yields a genuine Pontryagin
dualizing object $I_{\mathbf{R/Z}}$, modifying Eq. \ref{lqwxt3} to%
\[
\operatorname*{Hom}\nolimits_{\mathsf{Sp}_{lc}}(Z,I_{\mathbf{R/Z}}%
\mathbb{S})\cong\operatorname*{Hom}\nolimits_{\mathsf{LCA}_{\mathbf{Z}}%
}\left(  \pi_{0}Z,\mathbf{R/Z}\right)  \text{.}%
\]
The proof of reflexivity \textit{for all locally compact spectra} then should
be reduced to classical Pontryagin duality. To prove all this, one can
probably use a version of Clausen--Scholze's fully faithful embedding
$D^{b}(\mathsf{LCA}_{\mathbf{Z}})\hookrightarrow D(\mathsf{Cond}%
(\mathsf{Ab}))$ \cite[Lec. IV, Cor. 4.9]{condensedmath}. This is however not a
perfectly fitting counterpart to this putative subcategory%
\[
\mathsf{Sp}_{lc}\hookrightarrow\mathsf{Cond}(\mathsf{Sp})\text{,}%
\]
because of the much laxer boundedness condition:\ The spectra would not be
assumed to be connective or co-connective. The local compactness assumption
should be a purely levelwise one. Next, one would need to develop chromatic
homotopy theory in this setting, at least to the extent to have $L_{K(1)}$
available. I envison this being applied to non-connective $K$-theory, but then
instead of $p$-complete spectra one should get the $p$-completion as
topologically $p$-torsion compact ($\Leftrightarrow$ pro-$p$ profinite
spectra) spectra sitting inside $\mathsf{Sp}_{lc}$. I thank Ko Aoki and Achim
Krause for some conversations around this matter.
\end{vista}

\section{Proof of Lemma \ref{lemma_crit}\label{appendix_ProofOfLemma_Crit}}

We repeat the statement of the lemma.

\begin{lemma}
\label{lemma_crit copy}Let $K$ denote non-connective $K$-theory. Suppose
$k\geq1$. Let $S^{\prime}$ be a finite set of finite places such that
$S^{\prime}\subseteq S$. Assume that $S^{\prime}$ contains all the places of
$S$ which lie over the prime $p$. Then the natural morphism%
\[
L_{K(1)}\left(
{\textstyle\prod\nolimits_{v\in S\setminus S^{\prime}}}
(K/p^{k})(\mathcal{O}_{v})\right)  \longrightarrow%
{\textstyle\prod\nolimits_{v\in S\setminus S^{\prime}}}
L_{K(1)}(K/p^{k})(\mathcal{O}_{v})
\]
is an equivalence. As the intrinsic product in the $K(1)$-local homotopy
category can be computed in $\mathsf{Sp}$, the meaning of the product on the
right side is unambiguous.
\end{lemma}

\begin{proof}
Before we begin, let us recall that while $K$-theory commutes with infinite
products of categories, it does \textit{not} commute with infinite products of
rings\footnote{to see this, note that for any ring $R$, the tuple
$(R^{1},R^{2},R^{3},\ldots)$ is an object in $\prod_{\mathbf{N}}%
\mathsf{Proj}_{R,fg}$, but this is not even a finitely generated module over
the product ring $R_{\infty}:=\prod_{\mathbf{N}}R$, as anything surjected onto
by a finite direct sum $R_{\infty}^{\oplus n}$ has at most rank $n$ in each
slot. Thus, the underlying categories are different. Based on this idea, one
can build concrete counterexamples, as in \cite[Ch. II, Example 2.2.3]%
{MR3076731}, since $K(\prod_{\mathbf{N}}\mathsf{Proj}_{R,fg})\cong%
\prod_{\mathbf{N}}K(R)$.}, so one cannot unravel the product on the left side
as $K/p^{k}\left(  \prod\mathcal{O}_{v}\right)  $. While $\operatorname*{Spec}%
\prod\mathcal{O}_{v}$ is qcqs, this means that the qcqs descent results of
Clausen--Mathew \cite{MR4296353} do not directly apply. We use a different
idea and work over a Dedekind domain and realize the infinite product of
$K$-theory spectra as a product sheaf in this setting. This product sheaf is a
\textit{module} over ordinary $K$-theory, and then the method of Thomason's
construction of a convergent \'{e}tale descent spectral sequence for computing
the $K(1)$-localization, in the guise of inverting a Bott element $\beta$,
goes through for this module sheaf, having the corresponding infinite products
on its $E_{1}$-page. We chose to implement this using the technology of
\cite{MR4444265}. (Step 1) Our claim is proven if we can prove it for any set
of places $S$ of $F$ which

\begin{itemize}
\item contains all infinite places, and

\item contains no place which lies over $p$,
\end{itemize}

and $S^{\prime}=\varnothing$ (because the claim only refers to the difference
set $S\setminus S^{\prime}$). Thus, henceforth, we will assume that
$S^{\prime}=\varnothing$. Define $T:=\{$all places of $F\}\setminus S$. The
set $T$ contains only finite places and contains all places which lie above
$p$. We will now use \cite[Thm. 1.6]{MR4444265}. Our base scheme\footnote{in
\cite{MR4444265} the base scheme is called $S$, conflicting with our notation}
is $\operatorname*{Spec}\mathcal{O}_{T}$, where $\mathcal{O}_{T}$ is defined
as in Eq. \ref{l_Def_RingOS}, and the prime which is called $\ell$ loc. cit.
is taken to be $p$. Since $F$ is a number field and $\frac{1}{p}\in
\mathcal{O}_{T}$, all assumptions of the cited theorem are met. Write
$\operatorname*{SH}_{\operatorname*{Nis}}(\mathcal{O}_{T})$ for the stable
motivic category of \cite[\S 2.1.8]{MR4444265} and $\mathsf{KGL}%
_{\mathcal{O}_{T}}$ for the motivic spectrum representing Algebraic
$K$-theory. We identify the finite places $v\in S$ with closed points in
$\operatorname*{Spec}\mathcal{O}_{T}$. The collection of closed immersions
$i_{v}\colon\operatorname*{Spec}\kappa(v)\hookrightarrow\operatorname*{Spec}%
\mathcal{O}_{T}$ for $v\in S$ induces a natural map\footnote{which is a closed
immersion if and only if $S$ is finite}%
\[
\coprod_{v\in S}\operatorname*{Spec}\kappa(v)\overset{i_{S}}{\longrightarrow
}\operatorname*{Spec}\mathcal{O}_{T}%
\]
from the universal property of coproducts. The scheme on the left side is just
the disjoint union of the $\operatorname*{Spec}\kappa(v)$. Attached to the
adjunction $i_{S}^{\ast}\rightleftarrows i_{S\ast}$ in $\operatorname*{SH}%
_{\operatorname*{Nis}}(\mathcal{O}_{T})$ we get a unit $1\rightarrow i_{S\ast
}i_{S}^{\ast}$ and applied to $\mathsf{KGL}_{\mathcal{O}_{T}}/p$, this induces
a map%
\begin{equation}
\mathsf{KGL}_{\mathcal{O}_{T}}/p\longrightarrow i_{S\ast}i_{S}^{\ast
}\mathsf{KGL}_{\mathcal{O}_{T}}/p \label{lmark_1}%
\end{equation}
and%
\begin{equation}
\mathcal{F}:=i_{S\ast}i_{S}^{\ast}\mathsf{KGL}_{\mathcal{O}_{T}}/p\cong%
\prod_{v\in S}i_{v\ast}i_{v}^{\ast}\mathsf{KGL}_{\mathcal{O}_{T}}/p
\label{lmips6}%
\end{equation}
as $i_{v\ast}$, $i_{v}^{\ast}$ are exact. The map in Eq. \ref{lmark_1} renders
$i_{S\ast}i_{S}^{\ast}\mathsf{KGL}_{\mathcal{O}_{T}}/p$ an $\mathsf{MGL}%
_{\mathcal{O}_{T}}$-module.\footnote{Up to deleting various factors, including
the one of the generic point, this map is just the first map in the Godement
resolution for the Nisnevich topology.} We find%
\[
\operatorname*{Hom}\nolimits_{Ho(\operatorname*{SH}_{\operatorname*{Nis}}%
)}(\mathbf{1}_{\mathcal{O}_{T}}[n],\mathcal{F})\cong\pi_{n}\left(  \prod_{v\in
S}(K/p)(\mathcal{O}_{v})\right)
\]
In particular,\footnote{We have no use for the bigraded homotopy groups
$\pi_{n,q}\mathcal{F}$. Concretely, $\pi_{n}\mathcal{F}$ refers to $\pi
_{n,0}\mathcal{F}$.}%
\begin{equation}
\pi_{n}\mathcal{F}:=\operatorname*{Hom}\nolimits_{Ho(\operatorname*{SH}%
_{\operatorname*{Nis}})}(\mathbf{1}_{\mathcal{O}_{T}}[n],\mathcal{F})\cong%
\pi_{n}\left(  \prod_{v\in S}(K/p)(\mathcal{O}_{v})\right)  \text{.}
\label{lmips7}%
\end{equation}
\newline The computation readily generalizes to

\begin{itemize}
\item $\mathcal{O}_{T}$-schemes $X\rightarrow\operatorname*{Spec}%
\mathcal{O}_{T}$ (so that $\operatorname*{Hom}\nolimits_{Ho(\operatorname*{SH}%
\nolimits_{\operatorname*{Nis}})}(\mathbf{1}_{X}[n],\mathcal{F})$ computes the
$K$-theory of the respective basechange to $X$),

\item the \'{e}tale hypersheafification $\mathcal{F}^{\acute{e}t}%
:=\varepsilon_{\ast}\varepsilon^{\ast}\mathcal{F}$, where $\varepsilon$ is the
change-of-topology functor%
\[
\varepsilon^{\ast}\colon\operatorname*{SH}\nolimits_{\operatorname*{Nis}%
}\rightleftarrows\operatorname*{SH}\nolimits_{\mathrm{\acute{e}t}}%
\colon\varepsilon_{\ast}\text{.}%
\]

\end{itemize}

(Step 2) Temporarily, let us assume that $F$ contains a primitive $p$-th root
of unity $\zeta_{p}\in F$. Choosing this root, we can choose a Bott element
$\beta$, i.e., any element%
\begin{equation}
(\pi_{2}K(F))/p\hookrightarrow\overset{\beta\in}{\pi_{2}(K(F)\otimes
\mathbb{S}/p)}\twoheadrightarrow\left.  _{p}\pi_{1}K(F)\right.
\label{l_kx_1c}%
\end{equation}
which maps to $\zeta_{p}\in\pi_{1}K(F)\cong F^{\times}$. There is a convergent
\'{e}tale descent spectral sequence the hypersheafification $\mathcal{F}%
^{\acute{e}t}$. Evaluated on $\mathcal{O}_{T}$, it takes the shape%
\begin{equation}
E_{1}^{a,b}:=H^{2a-b}(\mathcal{O}_{T},\underline{\pi}_{a}(\mathcal{F}%
^{\acute{e}t}))\Rightarrow\pi_{b-a}(\mathcal{F}^{\acute{e}t})\text{.}
\label{lmf3pre}%
\end{equation}
Next, we invert the Bott element. We may model inverting $\beta$ as a
telescope localization%
\begin{equation}
\mathcal{F}^{\acute{e}t}[\beta^{-1}]\cong\operatorname*{colim}\left(
\mathcal{F}^{\acute{e}t}\overset{\cdot\beta}{\longrightarrow}\Sigma
^{-2}\mathcal{F}^{\acute{e}t}\overset{\cdot\beta}{\longrightarrow}\Sigma
^{-4}\mathcal{F}^{\acute{e}t}\overset{\cdot\beta}{\longrightarrow}%
\cdots\right)  \text{.} \label{l_kx_1d}%
\end{equation}
We can form the corresponding colimit of spectral sequences as in Eq.
\ref{lmf3pre}. On the one hand, this yields a new convergent spectral
sequence, baptized $E_{\bullet}^{\bullet,\bullet}[\beta^{-1}]$, such that
\begin{equation}
E_{1}^{a,b}[\beta^{-1}]:=H^{2a-b}(\mathcal{O}_{T},\underline{\pi}%
_{a}(\mathcal{F}^{\acute{e}t}[\beta^{-1}]))\Rightarrow\pi_{b-a}(\mathcal{F}%
^{\acute{e}t}[\beta^{-1}])\text{.} \label{l_kx_1e}%
\end{equation}
Next, one verifies the following:\medskip

\textbf{Claim A}: \textit{We claim that the }$E_{1}^{a,b}$\textit{-page of the
previous spectral sequence in Eq. \ref{lmf3pre} simplifies to}%
\begin{equation}
E_{1}^{a,b}=\prod_{v\in S}H^{2a-b}\left(  \mathcal{O}_{T},\mathbf{Z}/p\right)
\label{l_ky_m1}%
\end{equation}
\textit{and that multiplication with }$\beta$\textit{ acts through isomorphism
in the direct system of Eq. \ref{l_kx_1d}.}\newline(Step 3) By \textsl{Claim
A}, the telescopic colimit of the respective $E_{1}$-pages in Eq.
\ref{l_kx_1d} is equivalent to a constant inductive system. Hence, in Eq.
\ref{l_kx_1e} inverting $\beta^{-1}$ does nothing on the $E_{1}$-page and
moreover by Eq. \ref{l_ky_m1} the remaining terms simplify to%
\begin{equation}
E_{1}^{a,b}[\beta^{-1}]:=\prod_{v\in S}H^{2a-b}\left(  \mathcal{O}%
_{T},\mathbf{Z}/p\right)  \Rightarrow\pi_{b-a}(\mathcal{F}^{\acute{e}t}%
[\beta^{-1}])\text{.} \label{l_kx_1f}%
\end{equation}
It remains to identify the limit of this spectral sequence. From \cite[Thm.
1.6 (i)]{MR4444265} we deduce that the Bott-inverted $\mathcal{F}[\beta^{-1}]$
satisfies \'{e}tale hyperdescent. It follows that the natural morphism to the
hypersheafification must be an equivalence. This provides the first
equivalence in $\mathcal{F}[\beta^{-1}]\overset{\sim}{\longrightarrow
}\mathcal{F}[\beta^{-1}]^{\acute{e}t}\overset{\sim}{\longleftarrow}%
\mathcal{F}^{\acute{e}t}[\beta^{-1}]$. The second equivalence uses that
hypersheafification (a left adjoint) commutes with colimits. Using these
equivalences, the limit term in Eq. \ref{l_kx_1f} becomes%
\begin{equation}
\Rightarrow\pi_{b-a}(\mathcal{F}[\beta^{-1}])\text{.} \label{lmips8}%
\end{equation}
Next, we replace the Bott inversion by $K(1)$-localization.\medskip

\textbf{Claim B}:\textit{ We have}
\[
L_{K(1)}\prod_{v\in S}(K/p)(\mathcal{O}_{v})\cong\left(  \prod_{v\in
S}(K/p)(\mathcal{O}_{v})\right)  [\beta^{-1}]\text{.}%
\]

To see this, recall that the localization functor $L_{K(1)}$ can be realized
as $p$-completion, followed by $\mathsf{\operatorname*{KU}}$-localization. The
spectrum $\prod_{v\in S}(K/p)(\mathcal{O}_{v})$ is bounded $p$-torsion, so the
$p$-completion acts as the identity. The $\operatorname*{KU}$-localization is
smashing, i.e.,%
\begin{align}
L_{K(1)}\prod_{v\in S}(K/p)(\mathcal{O}_{v})  &  \cong L_{\operatorname*{KU}%
}\prod_{v\in S}(K/p)(\mathcal{O}_{v})\nonumber\\
&  \cong\left(  \prod_{v\in S}K(\mathcal{O}_{v})\right)  \otimes_{\mathsf{Sp}%
}\mathbb{S}/p\otimes_{\mathsf{Sp}}\mathbb{S}_{\operatorname*{KU}}%
\label{lmf4}\\
&  \cong\left(  \prod_{v\in S}(K/p)(\mathcal{O}_{v})\right)  [\beta^{-1}]
\label{lmf5}%
\end{align}
for $\mathbb{S}_{\operatorname*{KU}}:=L_{\operatorname*{KU}}\mathbb{S}$ the
$\operatorname*{KU}$-local sphere. The identification of the Bott element with
the inversion of Adams maps in Eq. \ref{lmf5} is a famous computation entirely
on the Moore spectrum $\mathbb{S}/p$ and $\mathbb{S}_{\operatorname*{KU}}$ (a
guide through this is for example in Mitchell \cite[Thm. 1.4 of Bousfield and
Thm. 2.3 of Snaith]{MR1069739} or more briefly in Waldhausen \cite[\S 4
Appendix, p. 193]{MR764579} or see Jardine's textbook \cite[\S 7.5, p.
276]{MR1437604} for a sketch). In particular, while these computations are
usually employed for the $K$-theory of a scheme, they also apply to the
infinite product in Eq. \ref{lmf4}. This proves \textsl{Claim B}.

(Step 4) Using a version of Eq. \ref{lmips7} (adapted to the $\beta$-inverted
$\mathcal{F}$), \textsl{Claim B} shows that the homotopy groups of
$L_{K(1)}\prod_{v\in S}(K/p)(\mathcal{O}_{v})$ in the ordinary stable homotopy
category agree with the homotopy groups $\pi_{n,0}$ of $\mathcal{F}[\beta
^{-1}]$ in $\operatorname*{SH}_{\operatorname*{Nis}}$. Thus, the spectral
sequence of Eq. \ref{l_kx_1f}, along with the new limit term of Eq.
\ref{lmips8}, becomes
\[
E_{1}^{a,b}[\beta^{-1}]:=\prod_{v\in S}H^{2a-b}\left(  \mathcal{O}%
_{T},\mathbf{Z}/p\right)  \Rightarrow\pi_{b-a}(\mathcal{F}[\beta^{-1}%
])\cong\pi_{b-a}L_{K(1)}\left(  \prod_{v\in S}(K/p)(\mathcal{O}_{v})\right)
\text{.}%
\]
This shows that the natural morphism%
\begin{equation}
L_{K(1)}\left(
{\textstyle\prod\nolimits_{v\in S\setminus S^{\prime}}}
(K/p)(\mathcal{O}_{v})\right)  \longrightarrow%
{\textstyle\prod\nolimits_{v\in S\setminus S^{\prime}}}
L_{K(1)}(K/p)(\mathcal{O}_{v}) \label{l_ky_1c}%
\end{equation}
induces an isomorphism on all homotopy groups and therefore is an
equivalence.\newline(Step 5) In Step 2 we had assumed that $\zeta_{p}\in F$.
This limitation can be removed in the standard way: The $K$-groups have finite
transfers and the degree of the field extension $F(\zeta_{p})/F$ is a divisor
of $p-1$, and in particular coprime to $p$. As all homotopy groups which occur
in our claim are $p$-primary groups, the standard transfer argument allows us
to reduce to proving our claim for $F(\zeta_{p})$. For details, see Jardine
\cite[Ch. 7, Cor. 7.29]{MR1437604} and use this for each factor in the
infinite product. Secondly, \ref{l_ky_1c} only proves our claim for $K/p$.
However, since both sides of Eq. \ref{lwui3} are $p$-complete spectra, one can
test whether the map is an equivalence by proving that it is an equivalence
mod $p$. This proves the lemma.
\end{proof}

\begin{acknowledgement}
I thank Markus Spitzweck for several e-mails helping me with some ideas and
technical hurdles. This work would not exist without the profound support by
M. Groechenig, K. R\"{u}lling and M. Wendt. I\ deeply appreciate their help.
Conversations with K. Aoki, D. Clausen, A. Krause, H. Krause, M. Spitzweck and
M. Wendt have repeatedly inspired and renewed my interest in this project.
\end{acknowledgement}

%

{\today}%

\bibliographystyle{amsalpha}
\bibliography{ollinewbib}

\def\cprime{$'$} \def\polhk#1{\setbox0=\hbox{#1}{\ooalign{\hidewidth
  \lower1.5ex\hbox{`}\hidewidth\crcr\unhbox0}}} \def\cprime{$'$}
  \def\cprime{$'$} \def\cprime{$'$} \def\cprime{$'$}
\providecommand{\bysame}{\leavevmode\hbox to3em{\hrulefill}\thinspace}
\providecommand{\MR}{\relax\ifhmode\unskip\space\fi MR }
\providecommand{\MRhref}[2]{%
  \href{http://www.ams.org/mathscinet-getitem?mr=#1}{#2}
}
\providecommand{\href}[2]{#2}
\begin{thebibliography}{ELS{\O}22}

\bibitem[AA78]{MR578649}
D.~Armacost and W.~Armacost, \emph{Uniqueness in structure theorems for {LCA}
  groups}, Canadian J. Math. \textbf{30} (1978), no.~3, 593--599. \MR{578649}

\bibitem[AB19]{kthyartin}
P.~Arndt and O.~Braunling, \emph{On the automorphic side of the {$K$}-theoretic
  {A}rtin symbol}, Selecta Math. (N.S.) \textbf{25} (2019), no.~3, Art. 38, 47.
  \MR{3954369}

\bibitem[Arm81]{MR637201}
D.~Armacost, \emph{The structure of locally compact abelian groups}, Monographs
  and Textbooks in Pure and Applied Mathematics, vol.~68, Marcel Dekker, Inc.,
  New York, 1981. \MR{637201}

\bibitem[BF15]{MR3402336}
T.~Barthel and M.~Frankland, \emph{Completed power operations for {M}orava
  {$E$}-theory}, Algebr. Geom. Topol. \textbf{15} (2015), no.~4, 2065--2131.
  \MR{3402336}

\bibitem[BGT13]{MR3070515}
A.~Blumberg, D.~Gepner, and G.~Tabuada, \emph{A universal characterization of
  higher algebraic {$K$}-theory}, Geom. Topol. \textbf{17} (2013), no.~2,
  733--838. \MR{3070515}

\bibitem[BGW16]{MR3510209}
O.~Braunling, M.~Groechenig, and J.~Wolfson, \emph{Tate objects in exact
  categories}, Mosc. Math. J. \textbf{16} (2016), no.~3, 433--504, With an
  appendix by Jan {\v{S}}{\v{t}}ov{\'{\i}}{\v{c}}ek and Jan Trlifaj.
  \MR{3510209}

\bibitem[BHv21]{clausennc}
O.~Braunling, R.~Henrard, and A.-C. {van {R}oosmalen}, \emph{A non-commutative
  analogue of {C}lausen's view on the id\`ele class group},
  \texttt{https://arxiv.org/abs/2109.04331} (2021).

\bibitem[BK72]{MR0365573}
A.~Bousfield and D.~Kan, \emph{Homotopy limits, completions and localizations},
  Lecture Notes in Mathematics, Vol. 304, Springer-Verlag, Berlin-New York,
  1972. \MR{0365573}

\bibitem[BM20]{MR4121155}
A.~Blumberg and M.~Mandell, \emph{{$K$}-theoretic {T}ate-{P}oitou duality and
  the fiber of the cyclotomic trace}, Invent. Math. \textbf{221} (2020), no.~2,
  397--419. \MR{4121155}

\bibitem[BMS67]{MR244257}
H.~Bass, J.~Milnor, and J.-P. Serre, \emph{Solution of the congruence subgroup
  problem for {${\rm SL}_{n}\,(n\geq 3)$} and {${\rm Sp}_{2n}\,(n\geq 2)$}},
  Inst. Hautes \'{E}tudes Sci. Publ. Math. (1967), no.~33, 59--137. \MR{244257}

\bibitem[BMY21]{bmychromaticconvergence}
A.~Blumberg, M.~Mandell, and A.~Yuan, \emph{{C}hromatic convergence for the
  algebraic {$K$}-theory of the sphere spectrum},
  \texttt{https://arxiv.org/abs/2110.03733} (2021).

\bibitem[Bra18]{obloc}
O.~Braunling, \emph{K-theory of locally compact modules over rings of
  integers}, International Mathematics Research Notices (2018), rny083.

\bibitem[Bra19]{MR4028830}
\bysame, \emph{On the relative {$K$}-group in the {ETNC}}, New York J. Math.
  \textbf{25} (2019), 1112--1177. \MR{4028830}

\bibitem[Bra20]{MR4118150}
\bysame, \emph{On the relative {$K$}-group in the {ETNC}, {P}art {III}}, New
  York J. Math. \textbf{26} (2020), 656--687. \MR{4118150}

\bibitem[Bra23]{lcahighdim}
\bysame, \emph{{$K$}-theory of locally compact modules for higher-dimensional
  varieties}, in preparation (2023).

\bibitem[B{\"u}h10]{MR2606234}
T.~B{\"u}hler, \emph{Exact categories}, Expo. Math. \textbf{28} (2010), no.~1,
  1--69. \MR{2606234 (2011e:18020)}

\bibitem[Cla13]{clausenthesis}
D.~Clausen, \emph{Thesis: {A}rithmetic {D}uality in algebraic {$K$}-theory}.

\bibitem[Cla17]{clausen}
\bysame, \emph{A {K}-theoretic approach to {A}rtin maps}, arXiv:1703.07842
  [math.KT] (2017).

\bibitem[CM21]{MR4296353}
D.~Clausen and A.~Mathew, \emph{Hyperdescent and \'{e}tale {$K$}-theory},
  Invent. Math. \textbf{225} (2021), no.~3, 981--1076. \MR{4296353}

\bibitem[CS19]{condensedmath}
D.~Clausen and P.~Scholze, \emph{{L}ectures on {C}ondensed {M}athematics (after
  {C}lausen--{S}cholze)}, 2019.

\bibitem[ELS{\O}22]{MR4444265}
E.~Elmanto, M.~Levine, M.~Spitzweck, and P.~A. {\O}stvaer, \emph{Algebraic
  cobordism and \'{e}tale cohomology}, Geom. Topol. \textbf{26} (2022), no.~2,
  477--586. \MR{4444265}

\bibitem[Har20]{MR4174395}
D.~Harari, \emph{Galois cohomology and class field theory}, Universitext,
  Springer, Cham, [2020], Translated from the 2017 French original by Andrei
  Yafaev. \MR{4174395}

\bibitem[HG94]{MR1217353}
M.~Hopkins and B.~Gross, \emph{The rigid analytic period mapping,
  {L}ubin-{T}ate space, and stable homotopy theory}, Bull. Amer. Math. Soc.
  (N.S.) \textbf{30} (1994), no.~1, 76--86. \MR{1217353}

\bibitem[HM07]{MR2327028}
R.~Hahn and S.~Mitchell, \emph{Iwasawa theory for {$K(1)$}-local spectra},
  Trans. Amer. Math. Soc. \textbf{359} (2007), no.~11, 5207--5238. \MR{2327028}

\bibitem[Hov08]{MR2342007}
M.~Hovey, \emph{Morava {$E$}-theory of filtered colimits}, Trans. Amer. Math.
  Soc. \textbf{360} (2008), no.~1, 369--382. \MR{2342007}

\bibitem[HPS97]{MR1388895}
M.~Hovey, J.~Palmieri, and N.~Strickland, \emph{Axiomatic stable homotopy
  theory}, Mem. Amer. Math. Soc. \textbf{128} (1997), no.~610, x+114.
  \MR{1388895}

\bibitem[HR79]{MR551496}
E.~Hewitt and K.~Ross, \emph{Abstract harmonic analysis. {V}ol. {I}}, second
  ed., Grundlehren der Mathematischen Wissenschaften [Fundamental Principles of
  Mathematical Sciences], vol. 115, Springer-Verlag, Berlin-New York, 1979,
  Structure of topological groups, integration theory, group representations.
  \MR{551496}

\bibitem[HS99]{MR1601906}
M.~Hovey and N.~Strickland, \emph{Morava {$K$}-theories and localisation}, Mem.
  Amer. Math. Soc. \textbf{139} (1999), no.~666, viii+100. \MR{1601906}

\bibitem[HS07]{MR2329311}
N.~Hoffmann and M.~Spitzweck, \emph{Homological algebra with locally compact
  abelian groups}, Adv. Math. \textbf{212} (2007), no.~2, 504--524. \MR{2329311
  (2009d:22006)}

\bibitem[Hv19a]{hr2}
R.~Henrard and A.-C. {van Roosmalen}, \emph{Derived categories of one-sided
  exact categories and their localizations}, \texttt{arXiv:1903.12647} (2019).

\bibitem[Hv19b]{hr}
\bysame, \emph{Localizations of one-sided exact categories},
  \texttt{arXiv:1903.10861} (2019).

\bibitem[Jar97]{MR1437604}
J.~Jardine, \emph{Generalized \'{e}tale cohomology theories}, Progress in
  Mathematics, vol. 146, Birkh\"{a}user Verlag, Basel, 1997. \MR{1437604}

\bibitem[Kel96]{MR1421815}
B.~Keller, \emph{Derived categories and their uses}, Handbook of algebra,
  {V}ol.\ 1, North-Holland, Amsterdam, 1996, pp.~671--701. \MR{1421815
  (98h:18013)}

\bibitem[Lam91]{MR1125071}
T.~Y. Lam, \emph{A first course in noncommutative rings}, Graduate Texts in
  Mathematics, vol. 131, Springer-Verlag, New York, 1991. \MR{1125071
  (92f:16001)}

\bibitem[Lan90]{MR1029028}
S.~Lang, \emph{Cyclotomic fields {I} and {II}}, second ed., Graduate Texts in
  Mathematics, vol. 121, Springer-Verlag, New York, 1990, With an appendix by
  Karl Rubin. \MR{1029028}

\bibitem[Lur11]{luriedagcomp}
J.~Lurie, \emph{{D}erived {A}lgebraic {G}eometry {XII}: {P}roper {M}orphisms,
  {C}ompletions, and the {G}rothendieck {E}xistence {T}heorem}.

\bibitem[Mil06]{MR2261462}
J.~Milne, \emph{Arithmetic duality theorems}, second ed., BookSurge, LLC,
  Charleston, SC, 2006. \MR{2261462}

\bibitem[Mit92]{MR1069739}
S.~Mitchell, \emph{{H}armonic localization of algebraic {$K$}-theory spectra},
  Trans. Amer. Math. Soc. \textbf{332} (1992), no.~2, 823--837. \MR{1069739}

\bibitem[Mor77]{MR0442141}
S.~Morris, \emph{Pontryagin duality and the structure of locally compact
  abelian groups}, Cambridge University Press, Cambridge-New York-Melbourne,
  1977, London Mathematical Society Lecture Note Series, No. 29. \MR{0442141}

\bibitem[Mos67]{MR0215016}
M.~Moskowitz, \emph{Homological algebra in locally compact abelian groups},
  Trans. Amer. Math. Soc. \textbf{127} (1967), 361--404. \MR{0215016}

\bibitem[NS07]{MR2276769}
N.~Nikolov and D.~Segal, \emph{On finitely generated profinite groups. {I}.
  {S}trong completeness and uniform bounds}, Ann. of Math. (2) \textbf{165}
  (2007), no.~1, 171--238. \MR{2276769}

\bibitem[NSW08]{MR2392026}
J.~Neukirch, A.~Schmidt, and K.~Wingberg, \emph{Cohomology of number fields},
  second ed., Grundlehren der mathematischen Wissenschaften [Fundamental
  Principles of Mathematical Sciences], vol. 323, Springer-Verlag, Berlin,
  2008. \MR{2392026}

\bibitem[Qui73]{MR0338129}
D.~Quillen, \emph{Higher algebraic {$K$}-theory. {I}}, Algebraic {$K$}-theory,
  {I}: {H}igher {$K$}-theories ({P}roc. {C}onf., {B}attelle {M}emorial {I}nst.,
  {S}eattle, {W}ash., 1972), Springer, Berlin, 1973, pp.~85--147. Lecture Notes
  in Math., Vol. 341. \MR{0338129 (49 \#2895)}

\bibitem[Rum01]{MR1856638}
W.~Rump, \emph{Almost abelian categories}, Cahiers Topologie G\'{e}om.
  Diff\'{e}rentielle Cat\'{e}g. \textbf{42} (2001), no.~3, 163--225.
  \MR{1856638}

\bibitem[Sai15]{MR3317759}
S.~Saito, \emph{On {P}revidi's delooping conjecture for {$K$}-theory}, Algebra
  Number Theory \textbf{9} (2015), no.~1, 1--11. \MR{3317759}

\bibitem[Sch]{symspec}
S.~Schwede, \emph{Symmetric spectra}.

\bibitem[Sch04]{MR2079996}
M.~Schlichting, \emph{Delooping the {$K$}-theory of exact categories}, Topology
  \textbf{43} (2004), no.~5, 1089--1103. \MR{2079996 (2005k:18023)}

\bibitem[Sch06]{MR2206639}
\bysame, \emph{Negative {$K$}-theory of derived categories}, Math. Z.
  \textbf{253} (2006), no.~1, 97--134. \MR{2206639}

\bibitem[Sch10]{MR2600285}
\bysame, \emph{Hermitian {$K$}-theory of exact categories}, J. K-Theory
  \textbf{5} (2010), no.~1, 105--165. \MR{2600285}

\bibitem[Sto12]{MR2946825}
V.~Stojanoska, \emph{Duality for topological modular forms}, Doc. Math.
  \textbf{17} (2012), 271--311. \MR{2946825}

\bibitem[Str00]{MR1763961}
N.~Strickland, \emph{Gross-{H}opkins duality}, Topology \textbf{39} (2000),
  no.~5, 1021--1033. \MR{1763961}

\bibitem[Wal84]{MR764579}
F.~Waldhausen, \emph{Algebraic {$K$}-theory of spaces, localization, and the
  chromatic filtration of stable homotopy}, Algebraic topology, {A}arhus 1982
  ({A}arhus, 1982), Lecture Notes in Math., vol. 1051, Springer, Berlin, 1984,
  pp.~173--195. \MR{764579}

\bibitem[Wei13]{MR3076731}
C.~Weibel, \emph{The {$K$}-book}, Graduate Studies in Mathematics, vol. 145,
  American Mathematical Society, Providence, RI, 2013, An introduction to
  algebraic $K$-theory. \MR{3076731}

\end{thebibliography}

\end{document}